\documentclass[12pt,a4paper]{amsart}

\makeatletter
\usepackage[english]{babel}
\usepackage[utf8]{inputenc}
\usepackage[square,sort,comma,numbers]{natbib}
\usepackage{graphicx}
\usepackage{amsfonts}
\usepackage{amsthm}
\usepackage{sansmath}
\usepackage{hyperref}
\usepackage{stmaryrd}

\usepackage{amsmath,amssymb,euscript}
\usepackage{xparse}
\usepackage{tikz-cd}
\usepackage{verbatim}
\usepackage{mathtools}
\usepackage[new]{old-arrows}

\usepackage{fancyhdr}
\setcitestyle{square}

\numberwithin{equation}{section}

\theoremstyle{plain}
	\newtheorem{theo}[equation]{Theorem}
	\newtheorem{prop}[equation]{Proposition}
	
	\newtheorem{lemm}[equation]{Lemma}

	\newtheorem{corr}[equation]{Corollary}

	\newtheorem{lem/defn}[equation]{Lemma/Definition}

\theoremstyle{definition}
	\newtheorem{defi}[equation]{Definition}
	\newtheorem*{ackn}{Acknowledgements}
		\newtheorem*{outl}{Outline}
	\newtheorem{exam}[equation]{Example}
	\newtheorem{hypo}[equation]{Hypothesis}
	
	\newtheorem{cons}[equation]{Construction}
	\newtheorem{caut}[equation]{Caution}
	\newtheorem{nota}[equation]{Notation}

\theoremstyle{remark}
	\newtheorem{rema}[equation]{Remark}

	\newtheorem{var}[equation]{Variant}

\def\nc{\newcommand}
\def\on{\operatorname}
\def\co{\colon\thinspace}

\newcommand{\RomanNumeralCaps}[1]
    {\MakeUppercase{\romannumeral #1}}
\newcommand*{\rom}[1]{\expandafter\@slowromancap\romannumeral #1@}

\nc{\edit}[1]{\marginpar{\footnotesize{#1}}}
\newcommand{\lv}{\lvert}
\newcommand{\rv}{\rvert}

\nc{\Z}{\mathbb{Z}}
\nc{\z}{\mathbb{Z}}
\nc{\zp}{\mathbb{Z}_p}
\nc{\zpl}{\mathbb{Z}_{(p)}}
\nc{\PP}{\mathbb{P}}
\nc{\R}{\mathbb{R}}
\nc{\ot}{\otimes}
\nc{\I}{\mathbb{I}}
\nc{\f}{\mathbb{F}}
\nc{\fp}{\mathbb{F}_p}
\nc{\hz}{H\mathbb{Z}}
\nc{\hzp}{H\mathbb{Z}_p}
\nc{\hzpl}{H\mathbb{Z}_{(p)}}
\nc{\fq}{\mathbb{F}_q}
\nc{\fpn}{\mathbb{F}_{p^n}}

\nc{\wfq}{W(\mathbb{F}_q)}
\nc{\wfpn}{W(\mathbb{F}_{p^n})}

\nc{\pmot}{\frac{p-1}{2}}

\nc{\sph}{\mathbb{S}}
\nc{\sphp}{\mathbb{S}_{p}}
\nc{\sphpl}{\mathbb{S}_{(p)}}

\nc{\sphwq}{\mathbb{S}_{W(\mathbb{F}_q)}}
\nc{\sphwpn}{\mathbb{S}_{W(\mathbb{F}_{p^n})}}
\nc{\sphwk}{\mathbb{S}_{W(k)}}
\nc{\sphxk}{\sph[x_k]}
\nc{\sphxm}{\sph[x_{m}]}
\nc{\sphxmk}{\sph[x_{mk}]}
\nc{\sphsk}{\sph[\sigma_k]}
\nc{\sphsmk}{\sph[\sigma_{mk}]}
\nc{\sphsell}{\sph[\sigma_\ell]}
\nc{\sphpsk}{\sph_{(p)}[\sigma_k]}
\nc{\sphpsmk}{\sph_{(p)}[\sigma_{mk}]}
\nc{\sphpsell}{\sph_{(p)}[\sigma_\ell]}
\nc{\musgmpn}{MU[\sigma_{2p^n-2}]}
\nc{\muallspi}{MU[\sigma_{2p^i-2} \mid i< n]}
\nc{\muallspinpo}{MU[\sigma_{2p^i-2} \mid i< n+1]}
\nc{\wdgmuallspi}{\wedge_{MU[\sigma_{2p^i-2} \mid i< n]}}
\nc{\thhmuallspi}{\ensuremath{\textup{THH}^{MU[\sigma_{2p^i-2} \mid i<n]}}}
\nc{\musi}{MU[\sigma_i]}
\nc{\musk}{MU[\sigma_k]}
\nc{\muspi}{MU[\sigma_{2p^i-2}]}

\nc{\sgmk}{\sigma_k}
\nc{\sgmmk}{\sigma_{mk}}
\nc{\sgml}{\sigma_{\ell}}
\nc{\sgmpi}{\sigma_{2p^i-2}}

\nc{\xrtmx}{X(\sqrt[m]{x})}
\nc{\artmx}{A(\sqrt[m]{a})}

\nc{\hzplsk}{H\mathbb{Z}_{(p)}[\sigma_{k}]}

\nc{\hzplsmk}{H\mathbb{Z}_{(p)}[\sigma_{mk}]}

\nc{\wdgsphsmk}{\wdg_{\sph[\sigma_{mk}]}}
\nc{\wdgsphsk}{\wdg_{\sph[\sigma_{k}]}}
\nc{\wdgsphpsmk}{\wdg_{\sph_{(p)}[\sigma_{mk}]}}
\nc{\wdgsphpsk}{\wdg_{\sph_{(p)}[\sigma_{k}]}}

\nc{\bpn}{BP \langle n \rangle}
\nc{\hfp}{H\mathbb{F}_p}
\nc{\T}{\mathbb{T}}
\nc{\vo}{V(1)}
\nc{\vos}{V(1)_*}
\nc{\ttw}{T(2)}
\nc{\pis}{\pi_*}
\nc{\ttws}{T(2)_*}
\nc{\gdna}{\gdn(A\hmod)}

\nc{\gdnr}{\gdn(R\hmod)}

\nc{\wdg}{\wedge}

\nc{\wdgp}{\wedge_{\mathbb{S}_p}}
\nc{\wdgfp}{\wedge_{H\mathbb{F}_p}}
\nc{\wdgmu}{\wedge_{MU}}

\nc{\AAA}{\mathbb{A}}
\nc{\LL}{\mathbb{L}}
\nc{\OO}{\mathcal{O}}

\nc{\X}{\EuScript{X}}
\nc{\sZ}{\EuScript{Z}}

\nc{\id}{{\on{id}}}

\nc\cone{{\on{cone}}}
\nc{\Rep}{{\on{Rep}}}
\nc\Ob{{\on{Ob}}}
\nc\Spec{{\on{Spec}}}

\mathchardef\mhyphen="2D
\newcommand{\hmod}{\mhyphen\mathsf{Mod}}
\nc\coMod{{\on{coMod}}}
\nc\Perf{{\on{Perf}}}
\nc\End{{\on{End}}}
\nc{\into}{\hookrightarrow}
\nc{\tr}{\on{tr}}
\nc{\ev}{\on{ev}}
\nc{\im}{\on{im}}
\nc{\hfps}{H{\mathbb{F}_p}_*}
\nc{\Mot}{\on{Mot}}
\nc{\pt}{\on{pt}}
\nc{\coker}{\on{coker}}
\nc{\rk}{\on{rank}}
\nc{\TOP}{\on{Top}_{\mathbb{C}}^{s}}
\nc{\gr}{\on{Gr}}
\nc{\Catperf}{\text{Cat}^{\text{perf}}}
\nc{\Sym}{\on{Sym}}
\nc{\xra}{\xrightarrow}
\nc{\lra}{\xleftarrow}
\nc{\Bet}{\mathbf{Betti}_{X}}
\nc{\codim}{\on{codim}}
\nc{\Fred}{\on{Fred}}
\nc{\colim}{\on{colim}}
\nc{\KK}{{\bf K}}

\nc{\onto}{\twoheadrightarrow}
\nc{\A}{\mathbb{A}}
\nc{\Aff}{\on{Aff}}
\nc{\SH}{\on{SH}}
\nc{\QCoh}{\on{QCoh}}
\nc{\Alg}{\on{Alg}}
\nc{\alg}{\on{Alg}}
\nc{\lmod}{\on{LMod}}

\nc{\rmod}{\on{RMod}}
\nc{\Br}{\on{Br}}
\nc{\ta}{\widetilde{\a}}
\nc{\Shv}{\on{Shv}}
\nc{\GG}{\mathbb{G}}
\nc{\red}{\color{red}}
\nc{\blue}{\color{blue}}
\nc{\an}{\on{an}}
\nc{\D}{\on{D}}
\nc{\Pre}{\on{Pre}}
\nc{\assact}{\on{Ass}_{\on{act}}^{\otimes }}
\nc{\spact}{\on{Sp}_{\on{act}}^{\otimes }}
\nc{\spactprod}{{\on{Sp}^{\Z}_{\on{act}}}^{\otimes }}
\nc{\qc}{\on{qc}}
\nc{\op}{\on{op}}
\nc{\shEnd}{{\mathcal End}}
\nc{\Sph}{\mathbb{S}}
\nc{\Top}{\on{Top}}
\nc{\Map}{\on{Map}}
\nc{\Vect}{\on{Vect}}
\nc{\holim}{\on{holim}}
\newcommand{\cat}[1]{\ensuremath{\EuScript #1}}

\DeclareMathOperator{\map}{\ensuremath{\textup{Map}}}

\DeclareMathOperator{\thh}{\ensuremath{\textup{THH}}}
\newcommand{\thhmu}{\ensuremath{\textup{THH}^{MU}}}
\DeclareMathOperator{\dy}{\ensuremath{\textup{Day}}}
\DeclareMathOperator{\idy}{\ensuremath{\mathbb{I}_\textup{Day}}}

\DeclareMathOperator{\otdy}{\ensuremath{\otimes_\textup{Day}}}
\DeclareMathOperator{\gd}{\ensuremath{\textup{Gr}}}
\DeclareMathOperator{\grn}{\ensuremath{\textup{Gr}}_n}

\DeclareMathOperator{\ntc}{\ensuremath{\textup{TC}^-}}
\DeclareMathOperator{\tp}{\ensuremath{\textup{TP}}}
\DeclareMathOperator{\tc}{\ensuremath{\textup{TC}}}

\newcommand{\kth}{K}
\let\plusminus\pm
\renewcommand{\pm}{{\plusminus1}}
\newcommand{\kup}{ku_p}

\newcommand{
\ltnpo}{L_{T(n+1)}}

\DeclareMathOperator{\hh}{\ensuremath{\textup{HH}}}
\DeclareMathOperator{\hhfp}{\ensuremath{\textup{HH}^{\fp}}}
\newcommand{\hhhfps}{\ensuremath{\textup{HH}_*^{\fp}}}
\DeclareMathOperator{\hhhfp}{\ensuremath{\textup{HH}^{H\fp}}}
\newcommand{\hhfps}{\ensuremath{\textup{HH}^{\fp}_*}}

\nc{\C}{\cat C}
\nc{\V}{\cat V}

\def\A{\mathcal{A}}

\def\a{\alpha}

\def\Perf{\on{Perf}}

\def\Sp{\on{Sp}}
\def\sp{\on{Sp}}

\def\V{\EuScript{V}}

\nc{\W}{\mathbb{W}}

\def\QCoh{\on{QCoh}}

\title{Adjunction of roots, algebraic $\kth$-theory and chromatic redshift} 
\author{Christian Ausoni, Haldun Özgür Bayındır, Tasos Moulinos}
\usepackage{natbib}
\usepackage{graphicx}

\begin{document}
\maketitle
%\maketitle
\begin{abstract}
Given an $E_1$-ring $A$ and a class $a \in \pi_{mk}(A)$
satisfying a suitable hypothesis, we define 
a map of $E_1$-rings $A\to \artmx$ realizing
the adjunction of an $m$th root of $a$.
We define a form of logarithmic THH for $E_1$-rings, and show that 
root adjunction is log-THH-\'etale for suitably tamely ramified extension,
which provides a formula for $\thh(\artmx)$ in terms of THH and log-THH of $A$.
If $A$ is connective, we prove that the induced map
$\kth(A) \to \kth(\artmx)$ in algebraic $\kth$-theory 
is the inclusion of a wedge summand. Using this, we obtain $V(1)_*\kth(ko_p)$ for $p>3$ and also, we deduce that if $\kth(A)$
exhibits chromatic redshift, so does $\kth(\artmx)$.
We interpret several extensions of ring spectra as examples of root
adjunction, and use this to obtain a new proof of the fact that Lubin-Tate spectra satisfy the redshift conjecture.
\end{abstract}

\section{Introduction}

Let $A$ be an $E_1$-algebra spectrum, and let $a \in \pi_{mk}(A)$ with $m>0$ and
even $k\geq0$.  In this paper, we define under a certain Hypothesis~\ref{hypo
root adjunction}, an $E_1$-algebra extension $A\to \artmx$ realizing the
adjunction of an $m$th-root of $a$ in homotopy rings, \[\pis A\to \pis
 \big(\artmx\big)\cong \pis(A)[z]/(z^m-a)  ,\] and then study how the algebraic
$K$-theory of $\artmx$, or its topological Hochschild and cyclic homology,
relates to that of $A$.  Hypothesis~\ref{hypo root adjunction} holds for example
if $A$ is an $E_2$-ring for which $\pis A$ is concentrated in even degrees. 

In general, the existence of a suitable root adjunction to a ring spectrum and its
effect on such invariants is an intriguing question; for example it has been shown by Schwänzl,
Vogt and Waldhausen~\cite[Proposition~2]{schwanzl1998adjroots}, precisely by considering
topological Hochschild homology, that it is not possible, in $E_\infty$-ring spectra, 
to adjoin a fourth root of unity $i$ to the sphere spectrum $\sph$ (in a sense made
precise in loc.~cit.). Nevertheless, Lawson~\cite{lawson2020adjoining} introduces a construction 
that allows, under some assumption, to adjoin roots of a homotopy degree zero unit
in $E_\infty$ ring spectra.

For classes in positive homotopy degrees, examples exist and have shown to be relevant, in
particular in studying redshift for algebraic $K$-theory. 
We have the Adams splitting of connective complex
$K$-theory $ku$ completed at an odd prime $p$,
$$
ku_p \simeq \bigvee_{1\leq i<p-1} \Sigma^{2i} \ell_p\,,
$$
and the extension $\ell_p\to ku_p$ can be interpreted, on homotopy rings, as the root adjunction
$$
\pis \ell_p \cong \Z_p[v_1]\to \Z_p[u]\cong \pis ku_p\,,
$$
were $v_1\mapsto u^{p-1}$, giving \[\pis ku_p\cong(\pis
\ell_p)[u]/(u^{p-1}-v_1).\] Sagave showed
in~\cite[4.15]{sagave2014logarithmistrctsconcomplexkthry} that
$ku_p$ can be constructed as an extension of $\ell_p$, establishing
how $\ell_p\to ku_p$ qualifies as a tamely ramified extension in
$E_\infty$-rings. 
In \cite{ausonirogneskthryoftopologicalkthry,ausoni2010kthryofcomplexkthry},
Rognes and the first author had computed the algebraic $K$-theory of $\ell_p$ and
$ku_p$ with coefficients in a Smith-Toda complex $V(1)=\sph/(p,v_1)$, for $p\geq5$. 
Taking $T(2)=V(1)[v_2^{-1}]$, one has the formula
$$
T(2)_* \kth(ku_p)\cong \big(T(2)_* \kth(\ell_p)\big)[b]/(b^{p-1}+v_2)
$$
relating the two computations, hinting at a chromatic shift (or redshift) of
this tamely ramified root adjunction.

After our construction of root adjunction in $E_1$-rings, we offer an investigation of how 
the obtained extension is reflected in algebraic $K$-theory.
In particular, we have the following spectrum-level splitting of algebraic
$K$-theory in the tamely ramified case, which applies to a wide array of examples.

\begin{theo}[Theorem \ref{theo root adj k theory inclusion}]\label{theo INTRO root adj k theory inclusion}
Assume Hypothesis \ref{hypo root adjunction} with $p\nmid m$ and $\lv a \rv >0$.
Furthermore, assume that $A$ is $p$-local and connective. In this situation, the
map in algebraic $K$-theory \[\kth(A) \to \kth(\artmx)\]
induced by the extension $A\to \artmx$ is the inclusion of a wedge summand.
\end{theo}
This is deduced from the corresponding result for topological cyclic homology,
see Theorem~\ref{theo weight zero splitting for tc}.
An analogous splitting result for algebraic $K$-theory in the non-connective case is
provided in Corollary~\ref{corr nonconnective alg kthry inclusion for root adj}.

For an integer $n>0$, we say that a spectrum $E$ is \emph{height} $n$ if $L_{T(n)}
E \not \simeq 0$ and $L_{T(m)}E \simeq 0 $ for $m>n$, where, $T(n)$ denotes a
height $n$ telescope. We say an $E_1$-ring $A$ of height $n$ \emph{exhibits
redshift} if $\kth(A)$ is of height $n+1$. Due to \cite[Purity
Theorem]{land2020purity}, $\kth(A)$ is of height at most $n+1$, so $A$ will
exhibit redshift if $L_{T(n+1)} \kth(A)\not\simeq0$.
The following result is thus an immediate consequence of our splitting results,
Theorem~\ref{theo root adj k theory inclusion} and Corollary~\ref{corr
nonconnective alg kthry inclusion for root adj}:

\begin{corr}\label{corr intro root adj also satisfy redshift}
Assume Hypothesis \ref{hypo root adjunction} with $p\nmid m$ and $\lv a \rv >0$.
If $A$ exhibits redshift, then so does $\artmx$.
\end{corr}

A key feature of the defined root adjunction
is that $\artmx$ is endowed with the structure of an $E_1$-algebra in the
symmetric monoidal category $\on{Fun}(\z/m,\sp)$ of $m$-graded spectra, which is reflected in an Adams' type splitting of spectra
$$
\artmx\simeq\bigvee_{0\leq i <m}\Sigma^{ik}A\,.
$$
Such a grading on a spectrum, compatible with further additional algebraic
structures, has already proven to be very useful:
let us mention the computations of Hesselholt-Madsen
\cite{hesselholt1997polytopesandtruncatedpolynomial} of the $\kth$-theory of
truncated polynomial algebras, and of the second and third authors
\cite{bayindir2020kthryofthh} of the $K$-theory of the free $E_1$-algebras in degree $2$ over finite fields. 
In the present paper, the grading corresponding to the root adjunction, and the
induced splitting of $\thh$ as developed in~\cite[Appendix
A]{antieau2020beilinson}, are essential ingredients in the proofs
of our various splitting results.

We use the theory of \emph{logarithmic topological Hochschild homology} for an in depth study of the THH of $\artmx$. Hesselholt and Madsen \cite{hesselholt2003ktheoryoflocalfields}
introduced logarithmic topological Hochschild
homology for studying the algebraic $K$-theory of complete discrete valuation rings in mixed
characteristic, and proved a descent property of log~THH in the
case of tamely ramified extensions.
Rognes~\cite{rognes2009topologicallogarithmic} then initiated a study 
of logarithmic structures, logarithmic André-Quillen homology and 
of $\mathrm{log}\thh$ in the context of $E_\infty$ ring spectra.
With Sagave and Schlichtkrull~\cite{rognes2015localization,rognes2018logthhofku}, they then established
the existence of localization sequences for $\mathrm{log}\thh$ and proved that
it satisfies tamely ramified descent in the example of the extension $\ell_p\to
ku_p$.

Here, we offer an alternative definition of log THH that applies to a more general class of ring spectra. More precisely, we define log THH for a given $E_1$-ring spectrum $A$ and a
class $a\in\pi_{mk}(A)$ satisfying Hypothesis~\ref{hypo root adjunction},
associated to the pre-log structure
given by the monoid generated, under multiplication, by $a\in \pi_*(A)$ (see 
Definition~\ref{defi log thh of general X}); it is denoted $\thh(A \mid
a)$. For instance, this applies to the Morava $K$-theory spectrum $k(n)$ for $v_n \in \pis k(n)$.  We prove the following form of tamely ramified descent (see also Theorem \ref{theo root adj is log thh etale }):

\begin{theo}[Theorem \ref{theo thh after root adjunction}]\label{theo intro thh after root adjunction}
If $A$ is $p$-local and $p\nmid m$, there is an equivalence of spectra:
\[\thh(\artmx) \simeq \thh(A) \vee \big(\bigvee_{0<i<m}\Sigma^{ik}\thh(A \mid a) \big).\]
\end{theo}
We also prove the existence of a localization cofibre sequence 
 \[\thh(A) \to \thh(A \mid a) \to \Sigma \thh(A/a),\]
see Theorem \ref{theo cofiber seq for thh to log thh}. This is an analogue in the
present setting of the localization sequences constructed
in~\cite{rognes2015localization}, but note that our definition only applies to
the case of a pre-log structure given by a monoid on a single generator. 
We would like to point out also the recent preprint of Binda, Lundemo, Park and
Østv{\ae}r~\cite{lundemo}, where a version
of logarithmic Hochschild homology for simplicial commutative rings is
constructed.
%as an affine derived scheme underlying the derived self
%intersections of the ``replete diagonal". 
%We expect of course that our construction is comparable with theirs. 

We now mention examples of application of the above results.

\subsection*{Topological $K$-theory}
We prove in Theorems~\ref{theo complex kthry as root adjunction} and~\ref{prop real kthry as root adjunction}
that, at an odd prime $p$, there are  equivalences of $E_1$-ring spectra
\begin{equation*}
	ku_{p} \simeq \ell_p (\sqrt[p-1]{v_1}) \ \ \textup{and}\ \ \
	ko_p \simeq \ell_p (\sqrt[\frac{p-1}{2}]{v_1})\,.
\end{equation*}
These equivalences upgrade the splitting result of Adams
into a $\z/(p-1)$-graded, respectively $\Z/(\frac{p-1}{2})$-graded $E_1$-ring structure.
We prove that when $p=1$ in $\z/m$, the splitting 
in Theorem~\ref{theo INTRO root adj k theory inclusion} can be improved to a more refined splitting of $\kth(\artmx)$. 
In the case of $ku_p\simeq \ell_p (\sqrt[p-1]{v_1})$, this more refined splitting reads as
\begin{equation}\label{intro-split-Kku}
\kth(ku_p)\simeq \bigvee_{0 \leq i <p-1} \kth(ku_p)_i\,.
\end{equation}
Here $\kth(ku_p)_0 \simeq \kth(\ell_p)$, and for $p>3$, and we can compute the
$V(1)$-homotopy groups of each of the $i$-th-graded piece $\vos\kth(ku_p)_i$
using first author's computation of $\vos \kth(ku_p)$. 

Using this refined splitting, the second author makes a simplified
computation of $T(2)_*\kth(ku)$ in \cite{bayindir2023algkthryofmrvykthry}. 
We also obtain  complete descriptions of $\vos
\kth(ko_p)$ and $\ttws \kth(ko_p)$, see Theorem \ref{theo kthry of ko}. 

We note that the splitting~\eqref{intro-split-Kku} can be considered as a version of Adams'
splitting for the cohomology theory represented by $K(ku)$, with classes
corresponding to $2$-categorical complex vector bundles, as developped in~\cite{baas2011stablebundlesrig}.

\subsection*{Johnson-Wilson spectra and Morava $E$-theory}
Let $n\geq 1$ be an integer, and let $E(n)$ and $E_n$ be the Johnson-Wilson and Morava $E$-theory
spectra. Let $\widehat{E(n)}$ be the $K(n)$-localization of $E(n)$. These
spectra have coefficient rings given as
$$
\pis E(n) \cong \zp[v_1,...,v_{n-1},v_n^\pm] , \hspace{7.5mm} 
\pis E_n \cong W(\fpn)[\lv u_1, \dots,u_{n-1}\rv][u^\pm]
$$
and $\pis \widehat{E(n)} \cong \pis E(n)^{\wedge}_{I_n}$, 
where $\lv u_i\rv = 0$, $\lv u\rv = -2$, and $I_n=(p,v_1,\dots,v_n)$.
The Galois group $Gal=\text{Gal}(\fpn | \fp)$ acts on $E_n$, and let
$E_n^{hGal}$ be the homotopy fixed point spectrum with coefficients
$\pis E_n^{hGal} \cong \Z_p[\lv u_1, \dots,u_{n-1}\rv][u^\pm]$.

We prove in Theorem~\ref{theo morava E theories as root adjunction}
that there are equivalences of $E_1$-rings 
\begin{equation}\label{eq intro mrv ethry as root adj}
\begin{aligned}
    E_n& \simeq   \sphwpn \wdgp \widehat{E(n)}(\sqrt[p^n-1]{v_n}),\ \
    \textup{and}\\
E_n^{hGal}&\simeq \widehat{E(n)}(\sqrt[p^n-1]{v_n})\,.
\end{aligned}
\end{equation}  
This promotes the $E_1$-ring structure on  Morava
$E$-theories to a non-trivial $\z/(p^n-1)$-graded $E_1$-ring structure.
 
Using this description of $E_n$ and the log THH \'etaleness of root adjunction, we show
in Theorem \ref{theo thh basechange for morava e theory}
that  the canonical map 
 \[\thh(\widehat{E(n)})\wdg_{\widehat{E(n)}} E_n \xrightarrow{\simeq_p}
 \thh(E_n),\]
 is an equivalence after $p$-completion.
 The relationship between such equivalences and the Galois
 descent question for THH are studied in \cite{mathew2017thhbasechng}.

Applying Corollary~\ref{corr nonconnective alg kthry inclusion for root adj} to
these root adjunctions of non-connective spectra, we deduce the following
result.
\begin{theo}[Theorem \ref{theo K theory of Morava E theories}]\label{theo intro ktrh of mrv e theory}
The canonical maps:
\begin{align*}
\kth(E(n)) \to&\  K(E_n^{hGal})\\
\kth(\sphwpn \wdg E(n)) \to & \  \kth(E_n)
\end{align*}
are inclusions of  wedge summands after $T(n+1)$-localization.
\end{theo}

\subsection*{Lubin-Tate spectra}
We can also apply our results to Lubin-Tate spectra that can be constructed, in
several steps, from the truncated Brown-Peterson spectra $\bpn$, 
with coefficients $\pis\bpn\cong\Z_{(p)}[v_1,\dots,v_n]$.
In more precise terms, we consider an $E_3$ $MU[\sigma_{2(p^n-1)}]$-algebra form of
$\bpn$ as constructed by Hahn and Wilson~\cite[Remark 2.1.2]{hahn2020redshift}.
In this situation, we can construct $\bpn(\sqrt[p^n-1]{v_n})$ as an $E_3$ $MU[\sigma_2]$-algebra
(Remark \ref{rema bpnadj root is ethree}). 

Let be $k$ a
perfect field of characteristic $p$, and let $\sph_{W(k)}$ denote the spherical Witt vectors spectrum.
We prove in Proposition~\ref{prop bpn adj root is lubin tate}
that the $MU[\sigma_2]$-orientation above provides a formal group law $\Gamma$ of
height $n$ over $k$, and that there is an equivalence of $E_3$-rings 
\[L_{K(n)} (\sph_{W(k)} \wdg\bpn)(\sqrt[p^n-1]{v_n}) \simeq E_{(k,\Gamma)}\,,\]
where $E_{(k,\Gamma)}$ denotes the Lubin-Tate spectrum corresponding to $\Gamma$.
Due to  \cite[Purity Theorem]{land2020purity}, we have an equivalence
\[L_{T(n+1)}\kth\big(\sph_{W(k)}\wdg\bpn(\sqrt[p-1]{v_n})\big) \simeq
L_{T(n+1)}\kth(E_{(k,\Gamma)}).\]
 By \cite{hahn2020redshift}, $\bpn$ satisfies the redshift conjecture; following an argument suggested to us by Hahn, we show that $\sph_{W(k)} \wdg\bpn$ also satisfies
the red-shift conjecture, c.f. Proposition~\ref{prop bpn with witt coefficients satisfy the redshift conjecture}.
By Corollary~\ref{corr intro root adj also satisfy redshift}, we deduce that $E_{(k,\Gamma)}$ satisfies the redshift conjecture. Indeed, we  deduce from this (see Theorem~\ref{theo lubin tate
redshift}) a new proof of Yuan's result~\cite{yuan2021examples} that
 all Lubin-Tate spectra satisfy the redshift conjecture. We also obtain the following from Corollary \ref{corr nonconnective alg kthry inclusion for root adj}.
\begin{theo}[Theorem~\ref{theo kthry splitting from bpn to lubin tate}]% \label{}
The induced map 
$$
L_{T(n+1)} \kth(\sph_{W(k)}\wdg\bpn) \to L_{T(n+1)}\kth(E_{(k,\Gamma)})
$$
is the inclusion of a \emph{non-trivial} wedge summand.
\end{theo}

We expect the above result to be relevant also for explicit computations. For example, in
\cite{angelini2022kthryofelliptic}, the authors compute $V(2)_*\tc(BP\langle
2\rangle)$ for $p\geq 7$ which, provides an explicit description of
$T(3)_*\kth(BP\langle 2 \rangle)$. Through the inclusion above, we deduce that
$T(3)_*\kth(BP\langle 2 \rangle)$ maps isomorphically to a summand of
$T(3)_*\kth(E_{(\fp,\Gamma)})$ for a height $2$ formal group law $\Gamma$ over
$\fp$. To our knowledge, this is the first explicit, quantitative result on the
algebraic $K$-theory groups of Lubin-Tate spectra for height larger than~$1$.

Note that it is not known if the $E_3$ $MU$-algebra forms of $\bpn$ constructed
in \cite{hahn2020redshift} map into the Morava $E$-theories mentioned in the
preceding sub-section.

%%%%%%%%%%%%%%%%%%%%%%%%%%%%%%%%%%%%%%%%%%%%%%%%%%%%%%%%%%%%%%%%%%%%%%%%%%%%%%%%%%%%%
%%%%%%_{CA_:_Iam_not_sure_if__the_following_remark_shouldn't_be_moved_to_the_body.
%%%%%%%%%%%%%%%%%%%%%%%%%%%%%%%%%%%%%%%%%%%%%%%%%%%%%%%%%%%%%%%%%%%%%%%%%%%%%%%%%%%%%%
\begin{rema}
    In \cite[Theorem G]{burklund2022chromaticnullstellensatz}, the authors construct an $E_\infty$-map $MU[\sigma_2] \to E_{(\overline{\mathbb{F}}_p,\Gamma')}$ for the unique height $n$ formal group law $\Gamma'$ on $\overline{\mathbb{F}}_p$ solving an open question on the existence of orientations on Lubin-Tate spectra. Our constructions provide a similar orientation for a ``smaller'' form of  Lubin-Tate spectra. Namely, we obtain that $E_{(\fp,\Gamma)}$ is an $E_3$ $MU[\sigma_2]$-algebra where $\sigma_2$ acts through $u^{-1} \in \pis E_{(\fp,\Gamma)}$; see Example \ref{exa for the telescope conjecture}. Moreover, there is a grading on $E_{(\fp,\Gamma)}$  that respects this structure. To be precise, $E_{(\fp,\Gamma)}$ is an $E_3$ $MU[\sigma_2]$-algebra in the $\infty$-category of $\z/(p^n-1)$-graded spectra where $u^{-1} \in \pis E_{(\fp,\Gamma)}$ is of weight $1$. % This structure carries over to every Lubin-Tate spectra $E_{(k,\Gamma')}$ whenever $k$ is an algebraically closed field of characteristic $p$ and $\Gamma'$ is the unique height $n$ formal group law over $k$.
\end{rema}

%%%%%%%%%%%%%%%%%%%%%%%%%%%%%%%%%%%%%%%%%%%%%%%%%%%%%%%%%%%%%%%%%%%%%%%%%%%%%%%%%%%%%
%%%%%%_{CA_:_I_would_remove_this_remark.
%%%%%%%%%%%%%%%%%%%%%%%%%%%%%%%%%%%%%%%%%%%%%%%%%%%%%%%%%%%%%%%%%%%%%%%%%%%%%%%%%%%%%%
%%%%%%%%%%%%%%%%%%%%%%%%%%%%%%%%%%%%%%%%%%%%%%%%%%%%%%%%%%%%%%%%%%%%%%%%%%%%%%%%%%%%%%

\begin{rema}\label{rema telescope}
The above construction of Lubin-Tate spectra by root adjunction is used in an
essential way in the construction of a counter-example to the telescope
conjecture in forthcoming work by Burklund, Hahn, Levy and Schlank.
\end{rema}

\subsection*{Morava $K$-theory}

In Section \ref{sec two periodic morava ktheory}, we construct two-periodic Morava $K$-theories  from Morava $K$-theories   through root adjunction. By Corollary \ref{corr intro root adj also satisfy redshift}, we deduce that two-periodic Morava $K$-theories satisfy the redshift conjecture if the redshift conjecture holds for Morava $K$-theories   (Corollary \ref{corr redshift for two periodic mrv kthry}). 

For $p>3$, the $V(1)$-homotopy of $\kth(k(1))$ is computed by the first author and Rognes in \cite{ausoni2021kthrymrvKthry} where it is also shown that $k(1)$ satisfies the redshift conjecture. From this, we deduce that the two-periodic first Morava $K$-theory  $ku/p$ also satisfies the redshift conjecture (Corollary \ref{corr first two perdc
mrv ktrhy redshift}). Moreover, through the interpretation of  $ku/p$ as $k(1)(\sqrt[p-1]{v_1})$, the second author makes the first  computation of 
$T(2)_*\kth\big(ku/p\big)$ in  \cite{bayindir2023algkthryofmrvykthry}.

%We should mention in a remark that we are answering a question of Devalapurkar
%on adjoining a root of the unit in $K(n)$.
\begin{comment}
\begin{rema}
In \cite{devalapurkar2020rootsofunity}, Devalapurkar shows that there are no
non-trivial $K(n)$-local $H_\infty$-rings with a primitive $p$th root of unity
(in degree $0$ homotopy) for odd $p$ and $n>0$. After this, Devalapurkar asks if
there are $K(n)$-local $E_k$-rings with a primitive $p$th root of unity in
degree $0$ homotopy for $k\geq 1$ \cite[Question
1.4]{devalapurkar2020rootsofunity}. We answer this question affermatively for $k
=1$ by adjoining a $p$th root to the unit in $\pi_0K(n)$; this is the $E_1$-ring
we denote by $K(n)(\sqrt[p]{1})$. 
\end{rema}
\end{comment}

\begin{outl}
We begin with a quick introduction of graded objects in Section \ref{sect graded ring spectra}.  In Section \ref{sect e2 polynomial algebras}, we construct a family of graded $E_2$ ``polynomial" algebras and establish their even cell decompositions. In Section \ref{sect adjoining roots and thh},  we provide our central construction for root adjunctions (Construction \ref{cons adjroots}) and prove our first splitting result on the THH of ring spectra obtained via a root adjunction. In Section \ref{sect adj roots and algebraic k theory} we prove Theorem \ref{theo intro ktrh of mrv e theory}. Section \ref{sect log thh} is devoted to studying the variant of log THH  we set forth, as well as the logarithmic THH-\'etaleness of root adjunctions. Section \ref{sec alg k theory of cmplx and real k theories} contains our results on the algebraic $\kth$-theory of real and complex topological $\kth$-theories. We apply our results to Lubin-Tate spectra in Section \ref{sect kthry and redshift for LT spectra}. In Section \ref{sec algebraic kthry of morava e theories}, we study the THH and the algebraic $K$-theory of Morava $E$-theories. 
\end{outl}

\begin{nota}\label{notation introduction}
\begin{enumerate}

\item We work freely in the setting of $\infty$-categories and higher algebra from \cite{lurie2009higher, lurie2016higher}. 
\item For an $E_2$-algebra $R$ in a symmetric monoidal $\infty$-category, when we say $T$ is an $R$-algebra (or an $E_1$ $R$-algebra), we mean that it is an $E_1$-algebra in \textit{right} $R$-modules. If we mean an $E_1$-algebra in left $R$-modules, we call this a \textit{left} $E_1$ $R$-algebra. If $R$ is $E_\infty$, we do not need to denote the distinction.
\item When we say $E_n$-ring, we mean an $E_n$-algebra in the $\infty$-category of spectra $\sp$.

\end{enumerate}
\end{nota}

\begin{ackn}
We would like to thank Sanath Devalapurkar for many of the ideas in Section
\ref{sect e2 polynomial algebras}. We also thank Jeremy Hahn for showing us the
proof of the redshift conjecture for $\bpn$ with Witt vector coefficients. We
benefited from various conversations with Andrew Baker and Robert Burklund and
we would like to thank them as well. We are very grateful to Robert Burklund
for pointing out a mistake in the proof of the claim, in a previous version of
this paper, that the Morava $K$-theories satisfy redshift; the claim has been
removed from the present version. 
We would like to thank Oscar
Randal-Williams for explaining to us in depth the constructions in
\cite{galatius2018cellular}. 
Finally, we would like to thank anonymous referrees for many very useful
suggestions or corrections.

The first and second authors acknowledge support from the project
ANR-16-CE40-0003 ChroK. The second author acknowledges support from the
Engineering and Physical Sciences Research Council (EPSRC)  grant EP/T030771/1.
The third author acknowledges support from the grant NEDAG ERC-2016-ADG-741501.
\end{ackn}
\section{Recollections on graded objects}\label{sect graded ring spectra}

Let $m\geq 0$ be an integer and let $\Z/m$ denote the discrete $\infty$-groupoid whose objects are the elements of the set of integers modulo $m$. For a presentably symmetric monoidal $\infty$-category $\cat V$,  we define the \textbf{$\infty$-category of $m$-graded objects in $\cat V$} to be the functor category $\on{Fun}(\Z/m, \cat V)$. For a functor $F$ in $\on{Fun}(\Z/m,\cat V)$, we denote $F(i)$ by $F_i$ for every $i \in \Z/m$. Since $\z/m$ is discrete, we have:
\[\on{Fun}(\Z/m, \cat V)\simeq \prod_{i \in \Z/m}\cat V.\]

For $m=0$, this is given by $\on{Fun}(\z, \cat V)$ where $\z$ is the corresponding discrete $\infty$-groupoid. In this case, we omit $m$ and we call $\on{Fun}(\z, \cat V)$ the \textbf{$\infty$-category of graded objects in $\cat V$}. We are mainly interested in the case $\cat V = \sp$. We call an object of $\on{Fun}(\z/m,\sp)$ an \textbf{$m$-graded spectrum}; for $m=0$, we drop $m$ and call it a \textbf{graded spectrum}. 

Using the symmetric monoidal structure on $\z/m$ given by addition, we equip  $\on{Fun}(\Z/m, \cat V)$ with the Day convolution closed symmetric monoidal structure \cite{glasman2016dayconvolution}. Since  $\Z/m$ is discrete, this boils down to the following. 

\begin{equation*}
    (F \ot_{\textup{Day}} G)_k = \coprod_{i+j = k \textup{\ in\ } \Z/m} F_i \ot G_j
\end{equation*}

\subsection{Algebras in  graded spectra}

We are interested in $E_n$-algebras in the $\infty$-category of $m$-graded spectra and the algebras over these $E_n$-algebras. 
\begin{defi}
An \textbf{$m$-graded $E_n$-ring} $A$ is an $E_n$-algebra in $\on{Fun}(\Z/m, \sp)$. For $k<n$, an $m$-graded $E_k$ $A$-algebra is an $E_k$ $A$-algebra in  $\on{Fun}(\Z/m, \sp)$. Similarly, an $m$-graded (left) right $A$-module is a (left) right $A$-module in   $\on{Fun}(\Z/m, \sp)$.
\end{defi}

\begin{rema}
Note that the notion of an $m$-graded $E_k$-algebra in $\sp$ is in general different than the notion of an $m$-graded object in the $\infty$-category of $E_k$-algebras in  $\sp$.
\end{rema}

\begin{comment}

In the category of graded ring spectra, there is a notion of homotopy groups.
\begin{cons}
 For a commutative ring spectrum $R$, there is a functor
\[\pi_* \co  \grn(R\hmod) \to \grn(\pis R \hmod).\]
For $M$ in $\grn(R\hmod)$, this functor is given levelwise as follows.
\[\pi_*(M)_i = \pis(M_i)\]
\end{cons}

\end{comment}

\subsection{Manipulations on graded objects}\label{subsec manipulations on grade objects}
For a symmetric monoidal functor $\z/m\to \z/m'$, there is an induced adjunction between $\on{Fun}(\z/m,\sp)$ and $\on{Fun}(\z/m',\sp)$ where the left adjoint is symmetric monoidal and given by left Kan extension \cite[Corollary 3.8]{nikolaus2016stablemultyoneda}. The right adjoint is given by restriction. This provides the following adjunctions which
allow us to move between various gradings. Let $n>0$ and $s \geq 0$ be integers.
\begin{itemize}
\item  We let $D^n_{sn} \dashv Q$   denote the adjunction induced by the quotient map $\z/sn \to \z/n$ sending $1$ to $1$.\begin{comment}
\begin{equation*} \label{diag restriction and kan extension}
 \begin{tikzcd}
 { \on{Fun}(\z/sn,\sp) } \arrow[r,"D^{n}_{sn}",shift left]
 & {  \on{Fun}(\z/n,\sp),  }\arrow[l,"Q", shift left] 
 \end{tikzcd}
 \end{equation*}
where $Q$ is given by the canonical restriction functor and $D^{n}_{sn}$ is the symmetric monoidal functor given by left Kan extension along $\z/sn \to \z/n$. 
\end{comment}
\item We often use  $D^{n}_{sn}$ for $s=0$ which allows us to obtain an $n$-graded object out of a graded object $X$ in $\sp$. We let $D^{n}$ denote $D^n_{0} \co \on{Fun}(\z,\sp) \to \on{Fun}(\z/n,\sp)$ and we have
\[D^n(X)_i \simeq \bigvee_{j\in \Z \mid j \equiv i\textup{\ mod n}}X_j.\] 
\item For $n = 1$, we denote $D^1_s$ by $D$. In this case, 
\[D \co \on{Fun}(\z/s,\sp) \to \sp\]  is given by $D(X) \simeq \vee_{j \in \z/s} X_j$, i.e.\ left Kan extension along $\z/s \to 0$. For an $s$-graded (spectrum) $E_n$-ring $X$, we call $D(X)$ the \textbf{underlying (spectrum) $E_n$-ring} of $X$. We often omit $D$ in our notation.

\item For $s \in \z$, let  $L_s \dashv R_s $ denote the adjunction on  $\on{Fun}(\z, \sp)$ induced by the  map $\z \xrightarrow{\cdot s} \z$ given by multiplication by $s$. For a graded spectrum  $X$, we have
\[L_s(X)_{si} \simeq X_i\]
for every $i$ and $L_s(X)_{j} \simeq 0$ whenever $s \nmid j$. 

\item Let $F\dashv G$ denote the adjunction induced by the trivial map $0 \to \z/m$. We have $G(X) = X_0$. For an $m$-graded $E_n$-ring $A$,  $F(G(A))$ is given by $A_0$ in weight $0$ and it is trivial on the other degrees. Therefore, we sometimes abuse notation and denote the $m$-graded $E_n$-ring $F(G(A))$ by $A_0$. The counit of this adjunction provides a map 
\[A_0 \to A\]
of $m$-graded $E_n$-rings. If $A_i \simeq 0$ for $i \neq 0$, then this map is an equivalence and we say that $A$ is \textbf{concentrated in weight zero}. The following lemma states that in this situation, there is an equivalence of $E_n$-rings between the underlying $E_n$-ring of $A$ and the weight zero piece $G(A) = A_0$ of $A$. Therefore, we often do not distinguish between $A$, $G(A) = A_0$ and $D(A)$ in our notation when $A$ is concentrated in weight zero.
\end{itemize}
\begin{lemm}\label{lem underlying of conc in weight zero}
Let $A$ be an $m$-graded $E_n$-ring concentrated in weight zero. There is an equivalence of $E_n$-rings
\[A_0 \simeq D(A)\]
where $A_0$ denotes $G(A)$. In particular,  we have  $F(D(A)) \simeq A$ as $m$-graded $E_n$-rings.
\end{lemm}
\begin{proof}
Since $A$ is concentrated in weight zero, we have $D(A)\simeq DFG(A)$. As $D \circ F$ Kan extends through the composite $0\to \z/m \to 0$, it is equivalent to the identity functor. We obtain that $DFG(A) \simeq G(A) \simeq A_0$. 

For the second statement, note that $F(D(A)) \simeq F(G(A))$ due to the first statement. Since $A$ is concentrated in weight $0$, we have $F(G(A))\simeq A$.
\end{proof}

\section{A family of \texorpdfstring{$E_2$}{e2} polynomial rings in graded spectra}\label{sect e2 polynomial algebras}
In this section, we introduce the construction of a family of $E_2$-algebras in graded spectra. These have appeared in the work of Hahn and Wilson in \cite{hahn2020redshift} and are also studied in greater depth in \cite{logstructures}. These will be central to our constructions. For every $r, w \in \Z$, one constructs a graded  $E_2$-ring $\mathbb{S}[\sigma_{2r}]$ which may be thought of as a ``polynomial" algebra with a generator in homotopical degree $2r$ and grading weight $w$. However, these are not polynomial algebras in the precise sense, as they are only demonstrated to admit $E_2$ structures.  By mapping these $E_2$-rings into each other, we will be able to construct  $E_2$-ring extensions.

\subsection{Shearing preliminaries}
The main mechanism underlying this construction is that of shearing, which we now briefly review. It has appeared in \cite{raksit2020hochschild}, and is also studied in \cite{logstructures}. In what follows, $\on{Gr}(\sp)$ denotes $\on{Fun}(\z,\sp)$, i.e.\ the $\infty$-category of graded spectra.

\begin{prop} \label{shearing} 
There exists an endofunctor on graded spectra  
$$
\on{sh}: \on{Gr}(\Sp) \to \on{Gr}(\Sp) 
$$
given by 
$$
\on{sh}(M)_i := M_i[-2i]
$$
 with the following properties:
\begin{itemize}
    \item $\on{sh}$ is an equivalence, with inverse given by $\on{sh}^{-1}(M_i) = M_{i}[2i]$
    \item $\on{sh}$ admits an $E_2$-monoidal structure, with respect to the Day convolution product on $\on{Gr}(\Sp)$.
\end{itemize}
\end{prop}

\begin{proof}
This appears in the $\Z$-linear setting in  \cite{raksit2020hochschild} and is also studied in \cite{logstructures}. However, for the sake of completeness, we sketch the basic ideas underlying the construction. 
In \cite{lurie2015rotation}, Lurie constructs an $E_2$-monoidal map of spaces 
$$
\phi: \Z \to \on{Pic}(\Sp) 
$$
sending $n \mapsto  \Sph^{-2n}$.  We now define $\on{sh}$ as the functor, obtained by adjunction, from the
assignment 
$$
\Z \times \gr(\Sp) \to \Sp
$$
given by the composition 
$$
\Z \times \gr(\Sp) \xrightarrow{(\phi, \on{ev})}  \on{Pic}(\Sp) \times \Sp \xrightarrow{\otimes} \Sp.
$$
Here, the first map sends $(n,M) \mapsto (\phi(n),M_n)$. The fact that this latter composition is $E_2$ follows from the fact that $\phi$ is itself $E_2$. This further implies that $\on{sh}$ is itself $E_2$ monoidal.
To see that this is an equivalence, one displays, as in \cite{raksit2020hochschild}, an inverse in the same way by precomposing $\phi$ with the map $\z\xrightarrow{-1\cdot}\z$.
\end{proof}

\begin{var}\label{var higher shearing}
One can precompose the map $\phi: \Z \to \on{Pic}(\Sp)$ with the map $\cdot (k) : \Z \to \Z$. We denote the composition by 
$$
\phi^k: \Z \to \on{Pic}(\Sp)
$$
As in the above, we use this to define an endofunctor 
$$
\on{sh}^k : \on{Gr}(\Sp) \to \on{Gr}(\Sp).
$$
This acquires the same formal properties as above, e.g it will be an $E_2$-monoidal autoequivalence on $\on{Gr}(\Sp)$. Furthermore, one  has the description  $(\on{sh}^k M)_i \simeq M_i[-2ki]$.

\end{var}
\subsection{Sheared polynomial algebras} 
Recall that there exists an $E_\infty$ algebra $\sph[t] \in \on{Gr}(\Sp)$, which gives a graded enhancement of the ``flat" polynomial algebra. One can obtain this, for example, by observing that the restriction map from filtered spectra to graded spectra 
$$
\on{Res}: \on{Fil}(\Sp) \to  \on{Gr}(\Sp)
$$
is lax symmetric monoidal. In more detail, this will be the restriction 
$$
\on{Fil}(\Sp) = \on{Fun}((\Z, \leq), \Sp) \to \on{Fun}(\Z^{\on{ds}}, \Sp) = \on{Gr}(\Sp) 
$$
along $\Z \hookrightarrow (\Z, \leq)$ so that in particular we forget the structure maps of the filtration, cf \cite{lurie2015rotation}. We remind the reader that this is different from the associated graded functor. One then sets $\sph[t] := \on{Res}(\mathbf{1})$, where $1$ denotes the unit of the symmeric monoidal structure on $\on{Fil}(\Sp)$. Thus, $\mathbb{S}[t]$ (which is given by $\Sph$ in nonpositive weights and $0$ in positive weights) acquires the structure of  an $E_{\infty}$-algebra in graded spectra.

\begin{cons}\label{cons of graded e2 polynomials}
As described in Proposition \ref{shearing}, there exists an $E_2$-monoidal autoequivalence  $\on{sh}$. We set 
$$
\Sph[\sigma_2] := \on{sh}(\sph[t]) ,
$$
and more generally for $k>0$,
$$
\Sph[\sigma_{2k}] : =  \on{sh}^{k}(\sph[t]);
$$
that is, one applies $\on{sh}^{k}$ to $\sph[t]$ to obtain a family of $E_2$-algebras in graded spectra. For $k$=0, we set $\sph[\sigma_0] := \sph[t]$. It follows by inspection that the underlying graded $E_1$-ring of $\sph[\sigma_{2k}]$ is the free graded $E_1$-ring on $\sph^{2k}(-1)$ where $\sph^{2k}(-1)$ is $\sph^{2k}$ concentrated in  weight $-1$.
\end{cons}

\begin{rema}
For $w\in \z$ and even $k\geq 0$, we apply the functor $L_{-w}$ which left Kan extends along the multiplication map 
$$
\cdot (-w): \Z \to \Z
$$
to obtain weight sifted variants of $\sph[\sigma_k]$. We often omit $L_{-w}$ when the weight of $\sigma_k$ is clear from the context but when we wish to be explicit, we write $\Sph[\sigma_{k, w}]$ for the graded $E_2$-ring $L_{-w}\sph[\sigma_k]$ where $\sigma_{k,w}$ is in weight $w$. As before, the underlying graded $E_1$-ring of $\Sph[\sigma_{k, w}]$ is the free graded $E_1$-ring  $\on{Free}_{E_1}(  \Sph^{k}(w))$.  
\end{rema}

To adjoin roots, we often start with $\sphsmk$ and $\sphsk$ with $\sigma_{mk}$ and $\sigma_k$ in weights $m$ and $1$ respectively where $m>0$ and $k>0$ is even.

\begin{prop}\label{prop map between free spheres}
In the situation above, there exists a map of graded $E_2$-rings 
\[\sphsmk \to \sphsk\]
that carries $\sgmmk$ to $\sgmk^m$ in homotopy. This provides a map of  $m$-graded $E_2$-rings 
\[D^m(\sphsmk) \to D^m(\sphsk)\]
where $\sgmk \in \pis D^m(\sphsk)$ is of weight $1$ and $D^m(\sphsmk)$ is concentrated in weight $0$.

Furthermore,  we have  $D^m(\sphsmk)\simeq F(D(\sphsmk))$ as $m$-graded $E_2$-rings; here $F$ Kan extends through $0 \to \z/m$. % {\red I think this last sentence is redundant} Through this equivalence, we may consider the map above as a map $F(D_1(\sphsmk)) \to D(\sphsk)$ of $m$-graded $E_2$-rings.
\end{prop}

\begin{proof}
We first remind the reader that there is an identification 
$$
\Sph[\sigma_{mk,-1}] := \on{sh}^{mk}(\Sph[t]) \simeq R_m(\Sph[\sigma_{k,-1}])
$$
of graded $E_2$-rings. This follows from the very definition the $k$th shearing functor $\on{sh}^{k}$, in particular that it sends  $\Sph[t]$ to the negative weight part of the  graded spectrum given by the map 
$$
\phi^k: \Z  \xrightarrow{\times k} \Z \to \on{Pic}(\Sp). 
$$
Thus we may identify  $\Sph[\sigma_{mk,-1}]$ with $R_m\Sph[\sigma_{k,-1}] \simeq R_m\on{sh}^{k}(\Sph[t])$, negative weight part of the  graded spectrum given by the map 
$$
\phi^{mk}: \Z  \to \on{Pic}(\Sp). 
$$
Therefore, the counit of the adjunction $L_m \dashv R_m$ provides a map of graded $E_2$-rings
$$
\Sph[\sigma_{mk,-m}] \to \Sph[\sigma_{k,-1}].
$$
Applying the functor $L_{-1}$ to this map, we obtain the map of graded $E_2$-rings $\sphsmk \to \sphsk$ claimed in the proposition. 

The functor $D^m$ gives the desired map
$$
D^m(\Sph[\sigma_{mk}]) \to D^m(\Sph[\sigma_{k}])
$$
of $m$-graded $E_2$-algebras. Note that $\sigma_{mk} \in \pi_{*}D^m(\Sph[\sigma_{mk}])$ is of weight $0$, which ensures that $D^m(\sphsmk)$ is concentrated in weight $0$.

To see the last statement, we have
\begin{equation}\label{eq interm triple kan extension}
    F DD^m(\sphsmk) \simeq D^m(\sphsmk)
\end{equation}
due to Lemma \ref{lem underlying of conc in weight zero} since $D^m(\sphsmk)$ is conentrated in weight $0$. Furthermore, $D D^m$ is Kan extension through the composite $\z\to \z/m \to 0$ which is the same as the Kan extending through $\z\to 0$. Therefore, $DD^m(\sphsmk) \simeq D(\sphsmk)$. This, together with \eqref{eq interm triple kan extension} provides the desired equivalence $D^m(\sphsmk) \simeq F(D(\sphsmk))$.
\end{proof}
We often omit the functor $D^m$ in our notation and denote the map of $m$-graded $E_2$-rings $D^m(\sphsmk) \to D^m(\sphsk)$ as $\sphsmk \to \sphsk$.

\begin{rema}\label{rema for degree zero root adjunction}
To adjoin a root to a degree $0$ class, we need the $k=0$ case of the proposition above. In other words, we need an analogous $E_2$-map $\sphsmk \to \sphsk$ for $k=0$. For this, we start with the graded $E_2$-map $\sph[\sigma_{m2}] \to \sph[\sigma_2]$   and apply the functor $\text{sh}$. This procedure provides a graded $E_2$-map $\sph[\sigma_{{0},m}] \to \sph[\sigma_{0,1}]$ that carries $\sigma_{0,m}$ to $\sigma_{0,1}^m$ as desired. 
\end{rema}

\subsection{Cell structures on sheared polynomial algebras}
A key technical result for us will be the even cell decomposition of  $\sph[\sigma_{k}]$ \emph{as an $E_2$-algebra}. As we will see in the remainder of this section, this is what will allow for us to define $E_2$-algebra maps to a given $E_2$-algebra $A$, along which we will then adjoin roots.  
\begin{rema}
    In the second arxiv version of \cite{hahn2020redshift}, Hahn and Wilson also construct $E_2$ even cell decompositions 
on the free $E_1$-algebra $\sphsk$ using a Koszul duality argument for even $k \geq 0$. However, they removed this 
result in the later versions of their paper since they found a simpler argument for their redshift
results that avoid the use of these even cell decompositions. Since this 
does not appear in the published version of \cite{hahn2020redshift}, we give a proof of the $E_2$ even cell decompositions on 
$\sphsk$ that we use. We would like to note that our methods are different than the ones used in 
\cite[Arxiv version 2]{hahn2020redshift}.  

\end{rema}
Before doing this, we make precise what exactly we mean by even cell decomposition. The following notions are heavily inspired by Section 6.3 of \cite{galatius2018cellular}.

\begin{defi} \label{cellularalgebras}
Let $f: S \to R \in \Alg_{E_2}(\Sp^{\Z})$ be a map of $E_2$-algebras in graded spectra. We say $f$ has a \emph{filtered cellular decomposition} if there exists  a tower in $\Alg_{E_2}(\on{Fil}(\on{Fun}(\Z, \Sp)))$ for which 
$$
S  = \on{sk}_{-1}(f) \to \on{sk}_{0}(f) \to \on{sk}_{1}(f) \to \cdots  \to  \colim_{i}\on{sk}_i(f) =: \on{sk}(f) \simeq R
$$
such that each $\on{sk}_{i}(f)$ is obtained from $\on{sk}_{i-1(f)}$ via the following pushout diagram: 

\begin{equation*}
\begin{tikzcd}
    \on{Free}_{E_2}(\bigsqcup_{\alpha \in I_{n_i} }\partial D^{g_{\alpha},n_i}[i-1] ) \ar[r]\ar[d] & \on{sk}_{i-1}(f) \ar[d]\\
    \on{Free}_{E_2}( \bigsqcup_{\alpha \in I_{n_i} } D^{g_{\alpha},n_i}[i] ) \ar[r] & \on{sk}_{i}(f).
    \end{tikzcd}
\end{equation*}
The notation $X[n]$ in the above means that the object $X$ is placed in filtering degree $n$.  

In particular, in each degree $i$ of the tower we are adding  cells in increasing dimension $n_i$. Thus if $i \leq j$, then $n_i \leq n_j$, and $I_{n_i}$ refers to the set of $n_i$ cells of $R$. It may be the case, that $n_i = i$, but we do not require this for the sake of flexibility of the definition, which is a point of departure from the notion in \cite{galatius2018cellular}. If $f: \mathbf{1} \to R$ is the map from the unit, we call this a filtered cellular decomposition of $R$. 
\end{defi}

\begin{defi} \label{unfilteredcellular}
A map $f\co S \to R$ of graded $E_2$-rings admits a cell decomposition if it is the colimit of a tower 
$$
S = \on{sk}_{-1}(f) \to \on{sk}_{0}(f) \to \on{sk}_{1}(f) \to \cdots  \to  \colim_{i}\on{sk}_i(f)  \simeq R
$$
in graded $E_2$-rings where each stage is obtained from the previous via a cell attachement in graded $E_2$-rings. In particular, if $f$ admits a filtered cellular decomposition, taking levelwise colimits provides a cellular decomposition of $f$.
    \begin{comment}
    In the above definition, we may take the levelwise colimit of the above tower to obtain a tower 
$$
S = \on{sk}_{-1}(f) \to \on{sk}_{0}(f) \to \on{sk}_{1}(f) \to \cdots  \to  \colim_{i}\on{sk}_i(f)  
$$
in $\Alg_{E_2}(\Sp^{\Z})$ whose colimit is exactly $R$.  We will simply call this a cellular decomposition of $f$.
\end{comment}

\end{defi}

Our first step is to establish the decomposition for $\sph[t]$.

\begin{prop} \label{evencellS[t]}
As an $E_2$-algebra in graded spectra, $\sph[t]$ admits a (filtered) cellular decomposition with cells in even degrees. 
\end{prop}

\begin{proof}
The two key inputs for our argument are Theorem 11.21 and Theorem 13.7 of \cite{galatius2018cellular}. The former result applied to the map $f: 0 \to  I$  of  non-unital $E_2$-algebras in graded spectra (where $I$ is the augmentation ideal of the map  $\sph[t] \to \sph$)
says that there exists a relative CW decomposition
$$
0 \to \colim \on{sk}_n(f) \simeq I.
$$
Moreover, the proof of this fact in loc. cit. constructs a minimal cell structure, one that has the smallest possible number of cells in a given bidegree (the extra degree here arises since we are working in graded spectra). In particular the colimit they construct, $\colim \on{sk}_n(f)$, will have cells precisely in bidegree $b^{E_2}_{g,d}(\Sph[t]):= \on{dim}_k H^{E_2}_{g,d}(\sph[t], \sph, k) \in \mathbb{N} \cup \{ \infty\}$, where  $H^{\OO}_{g,d}(R ; k)$ is the $\OO$-homology of an $\OO$-algebra $R$, with coefficients in the ring $k$; we will in fact set $k = \Z$.

We sketch the argument given there applied to our particular case for the sake of completeness. 
One proceeds by inductively constructing a factorization
$$
0 = \on{sk}_{-1} \xrightarrow{h_0} \cdots \xrightarrow{h_\epsilon} \on{sk}_{\epsilon} \xrightarrow{f_{\epsilon}} I, 
$$
in the $\infty$-category of (increasingly) filtered objects of $\on{Fun}(\Z, \Sp)$. Here  $h_e: \on{sk}_{e-1} \to \on{sk}_{e}$ comes with the structure of a (filtered) CW attachment of dimension $e$, where $\on{sk}_{i}$ denotes the $i$th skeleton equipped with the skeletal filtration leading up to that degree. Taking the colimit along $\epsilon$ gives an induced map $f_\infty: \colim sk_\epsilon(f) \to I$, which is an equivalence. 

For the inductive step in their argument, they show that the Hurewicz map 
$$
\pi_{*, \epsilon}(\sph[t], \on{sk}_{\epsilon -1}) \to H_{*, \epsilon}^{E_2}(\sph[t],  \on{sk}_{\epsilon -1}; k) 
$$
from relative homotopy to relative $E_2$-homology with coefficients in $k$  is surjective. Using this, one is able to choose  a set of maps 
$$
\{E_\alpha: (D^{\epsilon}, \partial D^{\epsilon}) \to (\Sph[t](g), \on{sk}_{\epsilon -1}(g))\},
$$
whose images generate $H_{*, \epsilon}^{E_2}(\sph[t],  \on{sk}_{\epsilon -1}; k)$ as a $k$-module. 
The boundary maps are then used to attach filtered  cells  $(g, \epsilon)$ to $\on{sk}_{\epsilon -1}$ to form $\on{sk}_{\epsilon}$ and the corresponding $E_\alpha$ is used to extend $f_{\epsilon -1}$ to $f_{\epsilon}$.
Putting all this together, we see that the attachment of the cells is parameterized by the dimensions of the $E_2$-algebra homology groups with coefficients in $k$. 
To see, in our particular setup, that this cell decomposition is concentrated in even degrees, it is therefore enough to verify that $H^{E_2}_{g, e}(\sph[t], k)$ vanishes whenever $e \cong 1 \mod 2$. For this we apply \cite[Theorem 13.7]{galatius2018cellular} which states that the $k$-fold iterated bar construction of an $E_k$ algebra is equivalent to the $k$-suspension of the $E_k$-cotangent complex. By \cite[Proposition 5.4.9]{lurie2015rotation}, there is an equivalence  $\on{Bar}^{(2)}(\sph[t]) \simeq  \on{gr}(\sph[\mathbb{C}P^{n}]_{n \geq 0}) $ in graded spectra, where the right hand side is the associated graded of the filtration on spherical chains on $\mathbb{C}P^\infty$, with filtration induced by the skeletal filtration on infinite projective space. Tensoring this with $k = \Z$ in our particular situation, we obtain (a 2-fold shift) of chains on $\mathbb{C}P^\infty$ with coefficients in $\Z$ which has a cell in each bidegree $(-n, 2n-2)$. By taking into account units, we conclude that $\sph[t]$ may be constructed from $\sph$ by attaching the same cells. 
\end{proof}

\begin{rema}
We remark that we may take the levelwise colimit in the above filtered cellular decomposition to obtain an $E_2$ cellular decomposition for  $\sph \to \sph[t]$ in the sense of Definition \ref{unfilteredcellular}. 
\end{rema}

\begin{corr}
The degree zero piece of the above cellular decomposition is the free algebra $\on{Free}_{E_2}(\sph(-1))$, i.e. the free $E_2$-algebra with generator in degree $0$ and weight $-1$. Moreover, the map 
$f_0: \on{Free}_{E_2}((\sph(-1)) \to  \Sph[t]$ itself admits a cellular decomposition with even cells of positive dimensions 
\end{corr} 

\begin{proof}
In degree zero, we have the following pushout square in the $\on{Alg}_{E_2}(\on{Fun}(\z, \Sp))$:

\begin{equation*}
\begin{tikzcd}
    \sph \simeq \on{Free}_{E_2}(\emptyset) \ar[r]\ar[d] & \on{sk}_{-1}(\Sph[t]) \simeq \sph \ar[d]\\
    \on{Free}_{E_2}(\sph(-1)) \simeq \on{Free}_{E_2}( D^{0}  ) \ar[r] & \on{sk}_{0}(\Sph[t]).
    \end{tikzcd}
\end{equation*}
Since this is a pushout square we obtain an equivalence
$$
\on{sk}_0(\sph[t]) \simeq \on{Free}_{E_2}(\Sph(-1))
$$
Moreover, by starting in degree zero with the zero cells already attached, we may conclude that the map 
$$
f_0: \on{Free}_{E_2}(\Sph(-1)) \to \Sph[t].
$$
itself admits a cellular decomposition with even cells of positive dimension. 

\end{proof}

By the above corollary, we have a cellular decomposition on the map of graded $E_2$-rings $\on{Free}_{E_2}(\Sph(-1)) \to \sph[t]$. We can apply shearing to this map to obtain a map
$$
\on{Free}_{E_2}(\Sph^{k}(-1)) \to \Sph[\sigma_k].
$$

\begin{prop} \label{evencellS[sigma]}
Let $k>0$ be even. The map $f_0: \on{Free}_{E_2} \sph^{k}(-1) \to \sph[\sigma_k] $ admits a cell decomposition with cells concentrated in even degrees. Left Kan extending along the multiplication map  $\Z \xrightarrow{\times -w} \Z $, we conclude that 
$f_0: \on{Free}_{E_2} \sph^{k}(w) \to \sph[\sigma_{k,w}] $ admits a cell decomposition with cells concentrated in even degrees. 
\end{prop}

\begin{proof}
 By construction, $\Sph[t]$ may be written as a filtered colimit of a diagram of $E_{2}$ algebras, 
$$
\on{Free}_{E_2}(\sph(-1)) \to \on{sk}_1(f) \to \cdots \to \on{sk}_{i-1}(f) \to \on{sk}_{i}(f) \to  \cdots 
$$
where each $\on{sk}_{i}(f)$  is formed  as a pushout from  $\on{sk}_{i-1}$ along a map $\on{Free}_{E_2}(\sph^{2n+1}) \to \sph$.  
We may apply $\on{sh}^{k/2}$ to this diagram, and take note of the fact that this will commute with colimits along the filtered diagram, together with the free $E_2$-algebra functor. Thus we conclude with an even cell presentation for the induced map:
$$
\on{sh}^{k/2}(\on{Free}_{E_2}(\Sph(-1))) \simeq   \on{Free}_{E_2}(\Sph^k(-1)) \to \Sph[\sigma_{k}].
$$
By left Kan extending along the multiplication by $-w$ map on $\Z$ (i.e.\ applying $L_{-w}$), we conclude analogously for the map
$$
\on{Free}_{E_2}(\Sph^k(w)) \to \sph[\sigma_{k,w}] 
.$$
\end{proof}

\begin{prop} \label{prop map from free spherical to even spectra}
Let $A$ be a (graded) $E_2$-ring whose homotopy groups are concentrated in even degrees and let  $a \in \pi_{k} A$  be a weight $w$ class  for some even $k\geq0$. Then there  is a  (graded) $E_2$-ring map
$$
\sph[\sigma_{k,w}] \to A
$$
which carries $\sigma_{k}$ to $a$. 
\end{prop}

\begin{proof}
By Proposition \ref{evencellS[sigma]}, the map 
$$
f: \on{Free}_{E_2}(\Sph^k(w)) \to   \sph[\sigma_{k,w}]
$$
admits an even cell decomposition.  
Let $a \in \pi_{k}(A)$ be as in the hypothesis of the proposition. This induces an $E_2$-algebra map $\on{Free}_{E_2}(\Sph^{k}(w)) \to A $, which we would like to extend inductively along the above tower. In order to do this, it is enough to  note that in degree $i$, we would need to trivialize the induced map $\on{Free}_{E_2}(\sph^{k + 2i -1}) \to A$.  Using the free/forgetful adjunction between $\on{Alg}_{E_2}(\on{Fun}(\Z, \Sp))$ and $\on{Fun}(\Z, \Sp)$ algebras and graded spectra, this will now follow from the fact that $\pi_{2n-1}(A)= 0$ for all $n$. 
\end{proof}

\begin{rema}
 In Remark \ref{rema for degree zero root adjunction}, we mentioned that we are going to use  $\on{sh}(\sph[\sigma_{m2,m}])$ to adjoin roots to degree $0$ classes. We remark that $\on{sh}(\sph[\sigma_{m2,m}])$ also satisfies the lifting property in the proposition above. This follows by the fact that the even cell decomposition for $\sph[\sigma_{m2,m}]$ provides an even cell decomposition for $\on{sh}(\sph[\sigma_{m2,m}])$  since $\on{sh}$ is an $E_2$-monoidal left adjoint functor.
\end{rema}

\begin{rema}
   We remark that another way to construct an $E_2$-algebra map $\Sph[t] \to A$ comes from the filtration on $\Sph[t]$ given its filtered cell decomposition. The mapping space 
    $$
    \on{Map}_{\on{Alg}_{E_2}}(\Sph[t], A)
    $$ 
    obtains a filtration from the filtation on the source; this will have associated graded 
    $$
\on{gr} \Map_{\Alg_{E_2}}(\sph[t], A ) \simeq \Map(\on{Free}_{E_2}(\bigoplus_{n \geq 1} \sph^{2n }, A)  \simeq \Map(\bigoplus_{n \geq 1} \sph^{2n }, A).  
$$
Now if $X$ has even homotopy groups, then so does the associated graded, so that the resulting spectral sequence computing the homotopy groups of the limit collapses. Thus,  
$a \in \pi_{0}(A)$ gives a class in $x \in \pi_0 \Map(\Sph^{2k}, X) \subset \pi_0 \Map(\bigoplus_{n \geq 1} \sph^{2n +2nk -2}, X) $, which will be an infinite cycle, and thus corresponds to an $E_2$-algebra map $\Sph[t] \to A$. We remark that this approach should allow for one to define maps 
$ \Sph[\sigma_k] \to A$ in the general case as well. We thank Oscar Randal-Williams for suggesting this approach. 
\end{rema}

\section{Adjoining roots and THH}\label{sect adjoining roots and thh}
Here, we introduce our construction for adjoining roots to ring spectra and prove our first results on the THH of ring spectra obtained through this construction.

\subsection{Background on algebras over $E_n$-algebras}\label{subsec background on algebras over en algebras}
Here is a quick background on some of the standard facts that we often use from \cite{lurie2016higher}. 

 For an $E_\infty$-algebra $R$ in a presentably symmetric monoidal $\infty$-category $\cat C$, the $\infty$-category of $E_n$ $R$-algebras is a symmetric monoidal $\infty$-category  with the pointwise tensor product \cite[Example 3.2.4.4]{lurie2016higher}. Therefore, for two $E_n$ $R$-algebras $A$ and $B$, $A \otimes_R B$ is an $E_n$ $R$-algebra. 

In this work, we often consider algebras over an $E_n$-algebra $R$ and in this case, the $\infty$-category of $E_m$ $R$-algebras (for $m\leq n-1$)  are not known to carry an appropriate $E_{n-1}$-monoidal structure. To work around this problem, we use the following facts which can be extracted from \cite[Corollary 4.8.5.20]{lurie2009higher}.

The $\infty$-category of (left) right $R$-modules is an $E_{n-1}$-monoidal $\infty$-category. We call an $E_m$-algebra in the $\infty$-category of right $R$-modules an $E_m$ $R$-algebra where $m\leq n-1$.

Furthermore, for  a map  $f\co R \to S$ of  $E_n$-algebras in $\C$,  one obtains an $E_{n-1}$-monoidal functor $-\otimes_{R} S$ between the respective $\infty$-categories of modules.  For every $m\leq n-1$, this induces a functor:
\begin{equation}\label{eq general extension of scalars}
- \otimes_{R} S \co \on{Alg}_{E_{m}}(\rmod_{R}) \to \alg_{E_{m}}(\rmod_{S}).
\end{equation}
 In particular, for an $E_m$ $R$-algebra $A$, $A \otimes_R S$ is an $E_{m}$ $S$-algebra. Furthermore, the forgetful functor induced by $f$, i.e.\ the right adjoint of $- \otimes_R S$, is $E_{n-1}$-lax monoidal and therefore it induces a functor:
 \begin{equation}\label{eq forgetful functor general}
     \on{Alg}_{E_{m}}(\rmod_{S}) \to \alg_{E_{m}}(\rmod_{R}).
 \end{equation}
 The unit of this adjunction provides a map of $E_m$ $R$-algebras:
 \begin{equation}\label{eq unit of the extension of scalars functor}
 A \to A \otimes_R  S.    
 \end{equation}
 
 Since $S$ is the monoidal unit in $\rmod_S$, $S$ is an $E_{n-1}$ $S$-algebra, and forgetting through \eqref{eq forgetful functor general}, it is an $E_{n-1}$ $R$-algebra. In summary, an $E_n$-algebra map $R \to S$ equips $S$ with the structure of an $E_{n-1}$ $R$-algebra. 
 
\begin{comment}
\begin{cons}
{\red this might be reduntant} Let  $m$ be positive with $k$ be non-negative and  even. The map  $\sph[\sigma_{mk}] \to \sph[\sigma_{k}]$ sending $\sigma_{mk} \to \sigma_k^m$  provide $\sph[\sigma_k]$ with the structure of a $m$-graded $E_1$ $\sph[\sigma_{mk}]$-algebra. 
\end{cons}
\end{comment}

\subsection{A construction for adjoining roots to ring spectra}
We now introduce our construction for adjoining roots to ring spectra. For this we use the following hypothesis. Recall that we often omit the functor $D$ and let $\sph[\sigma_k]$ denote the underlying $E_2$-ring of the graded $E_2$-ring $\sphsk$.
\begin{hypo}[Root adjunction hypothesis]\label{hypo root adjunction}
Given an $E_1$-ring $A$ with an $a \in \pi_{mk}A$, there is  a chosen $\sphsmk$-algebra structure on $A$ for which the structure map $\sphsmk \to A$ carries $\sigma_{mk}$ to $a \in \pi_{mk}A$. Here, $m>0$ and $k\geq 0$ is even. See Proposition \ref{prop hypothesis} for the cases of interest where this is satisfied. 
\end{hypo}
The hypothesis above may not seem very natural but it is satisfied in the following general situations.
\begin{prop}\label{prop hypothesis}
Let  $k\geq 0$ be even and $m>0$, an $E_1$-ring $A$ satisfies Hypothesis \ref{hypo root adjunction} for $a \in \pi_{mk}A$ if:
\begin{enumerate}
    \item $A$ is an $E_2$-ring for which $\pis A$ is concentrated in even degrees, or
    \item $A$ is an $R$-algebra for an $E_2$-ring $R$ where $\pis R$ is concentrated in even degrees and $a$ is in the image of the map $\pis R \to \pis A$. 
\end{enumerate}
\end{prop}
\begin{proof}
Assume that $A$ is as in (2), let $r\in \pi_{mk}R$ detect $a$ through the map $\pis R \to \pis A$. We choose an $E_2$-ring map $g\co \sphsmk \to R$ that carries $\sigma_{mk}$ to $r$, see Proposition \ref{prop map from free spherical to even spectra}. Forgetting through $g$, see \eqref{eq forgetful functor general}, one obtains a $\sphsmk$-algebra structure on $A$. Indeed, through this structure, $\sigma_{mk}$ acts through $a$ as desired.

If $A$ is as in (1), then $A$ is an $A$-algebra and $A$ satisfies the assumption in (2). Therefore, $A$ satisfies Hypothesis \ref{hypo root adjunction}.
\end{proof}

For instance, the Morava $\kth$-theory spectrum $K(n)$ and all $E_1$ $MU_{(p)}$-algebra forms of $\bpn$ satisfy Hypothesis \ref{hypo root adjunction} with respect to their non-negative degree homotopy classes.

 Notice that we are not assuming any preexisting non-trivial grading on $A$; in fact this will allow us to view it as an $m$-graded spectrum concentrated in weight zero. Given $a \in \pi_{mk} A$, the following construction adjoins an $m$-root to $a$.

\begin{cons} \label{cons adjroots} 
Assume Hypothesis \ref{hypo root adjunction}. We consider $\sphsmk$ as an $m$-graded $E_2$-ring and $A$ as an $m$-graded $\sphsmk$-algebra, both concentrated in weight $0$, using the functor $F$ from Section \ref{subsec manipulations on grade objects}; we omit $F$ in our notation. 

Due to Proposition \ref{prop map between free spheres} (Remark \ref{rema for degree zero root adjunction} for $k=0$), there is a map  \[\phi: \sphsmk \to \sphsk\]
of $m$-graded $E_2$-rings that carries $\sigma_{mk}$ to $\sigma_k^m$ in homotopy where $\sigma_k$ is of weight $1$ and $\sigma_{mk}$ is of weight $0$. Note that we omit the functor $D^m$ in our notation. Considering the corresponding extension of scalars functor $- \wdg_{\sphsmk} \sphsk$ between the $\infty$-categories of $m$-graded $\sphsmk$-algebras and $m$-graded $\sphsk$-algebras, (see \eqref{eq general extension of scalars}), we define the $m$-graded $E_1$ $\sphsk$-algebra $A(\sqrt[m]{a})$ through:
\[A(\sqrt[m]{a}):= A \wdg_{\sphsmk} \sphsk.\]
This comes equipped with a map  $A \to \artmx$ of $m$-graded $E_1$ $\sphsmk$-algebras, see \eqref{eq unit of the extension of scalars functor}.

 Since $\pis (\sphsk)$ is free as a $\pis (\sphsmk)$-module, one obtains an isomorphism of rings: 
\[\pis A(\sqrt[m]{a})\cong \pis (A)[z]/(z^m-a).\]
 Therefore, we say \textbf{$A(\sqrt[m]{a})$ is obtained from $A$ by adjoining an $m$-root to $a$}.

When $A$ is $p$-local, observe that we have and equivalence of $m$-graded $E_1$ $\sphsk$-algebras:
\[A(\sqrt[m]{a}) \simeq A \wdg_{\sphpsmk} \sphpsk,\]
where $\sph_{(p)}[\sigma_i]$ denotes the $p$-localization of $\sph[\sigma_i]$.  
\end{cons}

It follows that the weight pieces of $\artmx$ are given by the following 
\begin{equation}\label{eq weight pieces of root adj}
\artmx_i \simeq \Sigma^{ik}A
\end{equation}
for each $0\leq i<m$.

Note that $\artmx$ might possibly depend on the $\sph[\sigma_{mk}]$-algebra structure chosen on $A$. Therefore, everytime we apply Construction \ref{cons adjroots}, we fix an $\sph[\sigma_{mk}]$-algebra structure on $A$.

\begin{rema}\label{rema e2rings root adj are xalgebras}
If $A$ is an $E_3$-ring with even homotopy, one may start with an $E_2$-map $\sphsmk \to A$ for a given $a \in \pi_{mk}A$ with $k\geq 0$. By extending scalars, one obtains an $E_2$-ring map $A \wdg \sphsmk \to A$ that equips $A$ with the structure of an $A \wdg \sphsmk$-algebra. Through this, $\artmx$ is weakly equivalent as an $m$-graded $\sphsk$-algebra to 
\[A \wdg_{A \wdg \sphsmk} A \wdg \sphsk.\]
In particular, $\artmx$ admits the structure of an $m$-graded $A \wdg \sphsk$-algebra.
\end{rema}

\begin{rema}
 In general, we do not expect the root adjunction $A\ \to \artmx$ to satisfy a universal property. On the other hand, if $A$ is an $E_3$-ring, $\artmx$ is an $A$-algebra and the map $\pis A \to \pis \artmx$ is \'etale, then it follows by \cite[Theorem 1.10]{hesselholt2022dirac} that there is a bijection between  homotopy classes of $A$-algebra maps $\artmx \to B$ and  $\pis A$-algebra maps $\pis \artmx \to \pis B$ for any \'etale $A$-algebra $B$.
\end{rema}
For the following, we fix an $E_2$-map $\sphpl[\sigma_{2(p-1)}] \to \ell$ carrying $\sigma_{2(p-1)}$ to $v_1$.
\begin{theo}\label{theo complex kthry as root adjunction}
There is an equivalence 
\[ku_p\simeq \ell_p(\sqrt[p-1]{v_1})\]
of $E_1$ $\ell_p$-algebras.
\end{theo}
\begin{proof}
By Remark \ref{rema e2rings root adj are xalgebras} above, $\ell_p(\sqrt[p-1]{v_1})$ is an $\ell_p$-algebra. Let $L_p$ denote the non-connective $p$-completed Adams summand. The $E_1$ $L_p$-algebra
\[\ell_p(\sqrt[p-1]{v_1}) \wdg_{\ell_p} L_p\]
is an \'etale $E_1$ $L_p$-algebra in the sense of Hesselholt-Pstragowski \cite{hesselholt2022dirac} and there is an isomorphism of $\pis (L_p)$-algebras
\[\pis (\ell_p(\sqrt[p-1]{v_1}) \wdg_{\ell_p} L_p) \cong \pis (KU_p).\]
 It follows by \cite[Theorem 1.10]{hesselholt2022dirac} that there is an equivalence of $L_p$-algebras \[\ell_p(\sqrt[p-1]{v_1}) \wdg_{\ell_p} L_p \simeq KU_p.\] Through this, $\ell_p(\sqrt[p-1]{v_1})$ serves as the connective cover of $KU_p$ in $E_1$ $\ell_p$-algebras. Hence, there is an equivalence of $E_1$ $\ell_p$-algebras $ku_p \simeq \ell_p(\sqrt[p-1]{v_1})$.
\end{proof}

In Theorem \ref{theo morava E theories as root adjunction} we show that the Morava $E$-theory spectrum $E_n$ is given by $\sphwpn \wdg_{\sphp} \widehat{E(n)}(\sqrt[p^n-1]{v_n})$ as an $E_1$-ring where $\widehat{E(n)}$ is the $K(n)$-localized Johnson-Wilson spectrum. 
\begin{comment}
\begin{prop}\label{prop p local root adjunction alt def}
Assume Hypothesis \ref{hypo root adjunction}, if $A$ is $p$-local, then there is an equivalence of $E_1$ $\sphsk$-algebras
\[\artmx:= A \wdgsphsmk \sphsk \simeq A \wdgsphpsmk \sphpsk.\]
\end{prop}
\begin{proof}

 Since $p$-localization is a smashing localization, it preserves colimits. Furthermore, $p$-localization is  symmetric monoidal, therefore, it preserves the two sided bar construction defining relative smash products. This provides the second equivalence in: 
\[\artmx \simeq \artmx \wdg \sphpl \simeq (A \wdg \sphpl) \wdgsphpsmk \sphpsk \simeq  A \wdgsphpsmk \sphpsk,\]
where the first equivalence  follows by the fact that $\artmx$ is $p$-local as it is a left $A$-module.

{\red Red part should be deleted: This follows by the following where the second and the last equalities follow by the fact that $p$-localization is a smashing localization.
\begin{equation*}
    \begin{split}
        A \wdg_{\sphsmk} \sphsk &\simeq A \wdgsphpsmk (\sphpsmk \wdgsphsmk \sphsk)\\
        & \simeq A \wdgsphpsmk \big((\sph_{(p)} \wdg \sphsmk) \wdgsphsmk \sphsk\big)\\
        & \simeq A \wdgsphpsmk  (\sph_{(p)} \wdg \sphsk)\\
        & \simeq A\wdgsphpsmk \sphpsk
    \end{split}
\end{equation*}
}
\end{proof}
\end{comment}
\begin{rema}\label{rema bpnadj root is ethree}
In certain cases, it is possible to equip $\artmx$ with the structure of an $E_n$-algebra for $n>1$. For this, one may use the graded $E_\infty$ $MU$-algebra  $MU[\sigma_k]$  \cite[Construction 2.6.1]{hahn2020redshift} where $k> 0$ is even.  Indeed, $MU[\sigma_k]$  is the free graded $E_1$ $MU$-algebra over $\Sigma^kMU$. There is a map of graded $E_\infty$-rings 
\[MU[\sigma_{2(p^n-1)}] := L_{p^n-1}R_{p^n-1} MU[\sigma_2]   \to MU[\sigma_2]\] and Proposition \ref{prop map from free spherical to even spectra} provides maps $\sphsk \to MU[\sigma_k]$ of graded $E_2$-rings.

It follows by \cite[Remark 2.1.2]{hahn2020redshift} that a form of $\bpn$ admits the structure of an $E_3$ $MU[\sigma_{2(p^n-1)}$]-algebra where $\sigma_{2(p^n-1)}$ acts through $v_n \in \pis \bpn$.  We obtain an equivalence of $p^n-1$-graded $E_1$ $\sph[\sigma_2]$-algebras 
\[\bpn(\sqrt[p^n-1]{v_n}) := \bpn \wdg_{\sph[\sigma_{2(p^n-1)}]}\sph[\sigma_2] \simeq \bpn \wdg_{MU[\sigma_{2(p^n-1)}]} MU[\sigma_2].\]
This equips $\bpn(\sqrt[p^n-1]{v_n})$ with the structure of a $p^n-1$-graded $E_3$ $MU[\sigma_2]$-algebra.

\begin{comment}

Using this, we adjoin a root to $v_n$ in $\bpn$ via the following formula:
\[\bpn(\sqrt[p^n-1]{v_n}) := \bpn \wdg_{MU[\sigma_{2(p^n-1)}]} MU[\sigma_2].\]
This is a $p^n-1$-graded $E_3$ $MU[\sigma_2]$-algebra since the map $MU[\sigma_{2(p^n-1)}] \to MU[\sigma_2]$, induced by restricting and then Kan extending along $\z \xrightarrow{\cdot p^n-1} \z$, is a map of graded $E_\infty$ $MU$-algebras. 

Since $MU[\sigma_k]$  has even homotopy, one obtains a map of graded $E_2$-ring spectra $\sphsk \to MU[\sigma_k]$ which provides an equivalence of graded $E_2$ $MU$-algebras $MU \wdg \sphsk \simeq MU[\sigma_k]$. Using this, we deduce that  $\bpn(\sqrt[p^n-1]{v_n})$ above  agrees with that obtained from Construction \ref{cons adjroots} due to the following. 
\begin{equation*}
    \begin{split}
       \bpn \wdg_{MU[\sigma_{2(p^n-1)}]} MU[\sigma_2] &\simeq \bpn \wdg_{MU \wdg \sph[\sigma_{2(p^n-1)}]} MU \wdg \sph [\sigma_2]\\ &\simeq \bpn \wdg_{\sph[\sigma_{2(p^n-1)]}}\sph[\sigma_2]
    \end{split}
\end{equation*}
\end{comment}
\end{rema}
\subsection{The weight zero piece of THH}

Here, we prove our first result regarding the topological Hochschild homology of the ring spectra obtained via root adjunction. Namely,  we show that   $\thh(A(\sqrt[m]{a}))$ contains $\thh(A)$ as a summand whenever $A$ is $p$-local and $p \nmid m$. 

%Furthermore, topological Hochschild homology of an $m$-graded ring spectrum is an $S^1$-equivariant $m$-graded spectrum in a canonical way. When $A$ is $p$-local and $p \nmid m$, we show that  the weight zero piece of $\thh(\artmx)$ is given by $\thh(A)$ . 

It follows by \cite[Example A.10]{antieau2020beilinson} that for an $m$-graded $E_1$-ring $Y$, the  $m$-grading on $\thh(Y)$ is obtained by applying the cyclic bar construction $b_\bullet(Y)$ of $Y$  in the $\infty$-category of $m$-graded spectra. In simplicial level $s$ and weight  $i$, the $m$-graded cyclic bar construction of $Y$ is given by the following.
\[b_s(Y)_i \simeq \bigvee_{k_0+ \cdots +k_s =i \in \z/m}Y_{k_0} \wdg \cdots \wdg Y_{k_s}\]
Due to \cite[Corollary A.15]{antieau2020beilinson}, one has the following equality
\begin{equation}\label{eq graded thh to non graded thh}
    \thh(D(Y)) \simeq D(\thh(Y))
\end{equation}
where the functor $D(-)$ provides the underlying spectrum as usual; we often omit $D$ in our notation. Furthermore, $\thh(Y)$ is an $S^1$-equivariant $m$-graded spectrum in a canonical way and the equivalence above  is an equivalence of $S^1$-equivariant spectra.

\begin{cons}\label{cons thh of an algebra is a module}

Let $R$ be an $E_2$-ring and let $S$ be an $E_1$ $R$-algebra. For us this will mean that the pair $(R,S) \in \rmod^{(2)}(\Sp)$, where 
$$
\rmod^{(2)}(\mathcal{C})= \Alg(\rmod(\mathcal{C}))
$$
for an arbitrary symmetric monoidal $\infty$-category $\mathcal{C}$. Here, $\rmod(\mathcal{C})$ is the $\infty$-category of pairs $(A,M)$ where $A$ is an $E_1$-algebra and $M \in \rmod_A(\mathcal{C})$. Thus, objects in $\rmod^{(2)}(\mathcal{C})$ may be identified with pairs $(A,M)$ where $A$ is an $E_2$-algebra, and $M$ is an $E_1$ A-algebra in $\mathcal{C}$.   

We remark that in general, $\rmod^{(2)}(\mathcal{C})$ may be written as $\Alg_{\mathcal{O}}(\mathcal{C})$ where $\mathcal{O}$ is the \emph{tensor product} of operads 
$$
\mathcal{O} := \rmod \times E_1
$$
This tensor product of operads, studied in depth in \cite[Section 2.2.5]{lurie2016higher} is symmetric and satisfies the following universal property at the level of algebra objects:
$$
\Alg_{\mathcal{O}}(\mathcal{C})  \simeq \Alg_{E_1}(\Alg_{\rmod}(\mathcal{C})) \simeq \Alg_{\rmod}(\Alg_{E_1}(\mathcal{C}))
$$
Hence, applying the discussion to $\mathcal{C} = \Sp $ and  $R$ and $S$ as above, we may view $S$ as a right $R$-module  in $E_1$-algebras. 

Since $\thh$ is a symmetric monoidal functor from $E_1$-rings to spectra \cite[Section \RomanNumeralCaps{4}.2]{nikolausscholze2018topologicalcyclic} we deduce that $\thh(S)$ is a right $\thh(R)$-module.

\end{cons}

\begin{prop}\label{prop thh of relative smash product} 
Let  $F$ be an $m$-graded $E_1$ $E$-algebra and $E\to F'$ be a map of $m$-graded $E_2$-rings. There is a natural equivalence of $m$-graded right $\thh(F')$-modules in $S^1$-equivariant spectra:
\[\thh(F\wdg_E F') \simeq \thh(F) \wdg_{\thh(E)} \thh(F'),\]
whose underlying undgraded equivalence is that of  right  $\thh(F')$-modules in cyclotomic spectra.
If $E$ and $F$ are concentrated in weight zero, then we have the following.
\[\thh(F\wdg_E F')_i \simeq \thh(F) \wdg_{\thh(E)} (\thh(F')_i)\]
\end{prop}
\begin{proof}

Let us recall that the functor 
$$
\thh: \on{Alg}_{\Sp} \to \on{CycSp} 
$$ 
is symmetric monoidal. Furthermore, it commutes with sifted colimits; indeed this can be seen from the fact that it can be decomposed into a composition of functors comprised of taking tensor products and realizations of simplicial objects, both of which commute with sifted colimits. Thus there will be a natural equivalence 
\begin{align*}
   \thh(F \wdg_{E} F') &\simeq  \thh( ||\on{Bar}_\bullet(F,E,F')||)  \\
   &\simeq
    ||\on{Bar}_{\bullet}(\thh(F)_\bullet,
  \thh(E)_\bullet, \thh(F')_\bullet)|| \\
  &\simeq \thh(F) \wdg_{\thh(E)}\thh(F') 
\end{align*}
This allows us to deduce that $\thh$ preserves the
sifted colimit given by the double sided Bar construction; this can be computed at the level of underlying spectra by the bilinear pairing
$$
{}_{F}\on{BMod}_E \times {}_{E}\on{BMod}_{F'} \to _{F} \on{BMod}_{F'} 
$$
corresponding to the relative tensor product. Furthermore, as $\thh$ preserves the sifted colimits corresponding to this relative tensor product, the above equivalence is compatible with right $\thh(F')$ module structures. 
The analogous claims all hold when accounting for additional gradings, by recalling that $\thh$ promotes to a sifted colimit preserving symmetric monoidal functor from algebras in graded spectra to $S^1$-equivariant objects in graded spectra.   
In particular, if $E$ and $F$ are concentrated in weight zero, we deduce the equivalence 
$$\thh(F\wdg_E F')_i \simeq \thh(F) \wdg_{\thh(E)} (\thh(F')_i).
$$
of graded $\thh(F)$-modules. 
\end{proof}
\begin{rema}
The $m=1$ case of the proposition above provides the non-graded case. 
\end{rema}
One may consider $\sphsk$  as an $E_1$-ring obtained by adjoining an $m$-root to $\sphsmk$.  Proposition \ref{prop spherical root degree zero thh} identifies the weight zero piece of $\thh(\sphpsk)$. Before Proposition \ref{prop spherical root degree zero thh}, we state and prove a well known fact.

\begin{lemm}\label{lemm map of bounded below spectra}
Let $\varphi \co M \to N$ be a map between bounded below spectra. Then $\varphi$ is an equivalence if and only if $\hz \wdg \varphi$ is an equivalence. If furthermore $M$ and $N$ are $p$-local, then $\varphi$ is an equivalence if and only if $\hzpl \wdg \varphi$ is an equivalence.
\end{lemm}
\begin{proof}
Let $\kth$ be the fiber of $\varphi$ and let $i$ be the smallest $i$ such that $\pi_iK \neq 0$. Due to the Tor spectral sequence of \cite[Theorem \RomanNumeralCaps{4}.4.1]{elmendorf2007rings}, we have $\pi_i (\hz \wdg K)= \pi_iK$. Therefore, if $\hz \wdg K \simeq 0$ then $K \simeq 0$ and $\varphi$ is an equivalence.

If $M$ and $N$ are $p$-local, then $\varphi$ is an equivalence if and only if $\sphpl \wdg \varphi$ is an equivalence. It follows by the previous result that $\varphi$ is an equivalence if and only if $\hz \wdg \sphpl \wdg \varphi \simeq \hzpl \wdg \varphi$ is an equivalence.
\end{proof}

For the following, let $k\geq0$ be even and let $m>1$. Furthermore, fix a map of  $m$-graded $E_2$-rings $\sphpsmk \to \sphpsk$ provided by Proposition \ref{prop map between free spheres} (Remark \ref{rema for degree zero root adjunction} for $k=0$). 
\begin{prop}\label{prop spherical root degree zero thh}
In the situation above, assume that $p \nmid m$. The induced map 
\[ \thh(\sphpsmk)_0  \xrightarrow{\simeq} \thh(\sphpsk)_0\]
is an equivalence of $E_1$-rings. Since $\sphpsmk$ is concentrated in weight zero, we obtain the following chain of equivalences
\[D(\thh(\sphpsmk)) \simeq \thh(\sphpsmk)_0  \xrightarrow{\simeq} \thh(\sphpsk)_0\]
of $E_1$-rings using Lemma \ref{lem underlying of conc in weight zero}.
\end{prop}
\begin{proof}
It suffices to prove that
the map 
\[ \hzpl \wdg \thh(\sphpsmk)_0  \to \hzpl \wdg \thh(\sphpsk)_0\]
is an equivalence, see Lemma \ref{lemm map of bounded below spectra}.

By the base change formula for THH,  this is equivalent to the following map 
\begin{equation*}
    \thh^{\hzpl}(\hzplsmk) \to \thh^{\hzpl}(\hzplsk)
\end{equation*}
being an equivalence in weight zero. Here,  $\hzplsk$ denotes the free $\hzpl$-algebra on $\sphpl^k$ given by $\hzpl \wdg \sphpsk$. 

The map $\hzplsmk \to \hzplsk$ induces a map 
\[\phi^r \co E^r \to {F}^r\] 
from the  B\"okstedt spectral sequence  computing $\thh^{\hzpl}(\hzplsmk)$ to the B\"okstedt spectral sequence computing  $\thh^{\hzpl}(\hzplsk)$.  Since the weight grading on the THH of an $m$-graded ring spectrum comes from a weight grading on the corresponding cyclic bar construction, the B\"okstedt spectral sequence computing THH of an $m$-graded ring spectrum admits an $m$-grading, i.e.\ it splits into $m$ summands in a canonical way. Therefore, in our situation, it is sufficient to show that $\phi^2$  is an isomorphism on weight zero. 

We have
\begin{equation}\label{eq bokstedt ss for polynomial dgas}
    \phi^2 \co \zpl[\sgmmk]\otimes \Lambda_{\zpl}(d(\sgmmk)) \to \zpl[\sgmk] \otimes \Lambda_{\zpl}(d(\sgmk))
\end{equation}
where $d$  denotes the Connes operator. The degrees of the classes above are given by the following.  \[\text{deg}(\sgmmk) = (0,mk)\  \ \   \text{deg}(d(\sgmmk)) = (1,mk)\]  \[\text{deg}(\sgmk) = (0,k) \textup{\ \ \ \  } \text{deg}(d(\sgmk)) = (1,k)\] 
Furthermore, $\sgmmk$ and $d(\sgmmk)$ are in weight $0$ and $\sgmk$ and $d(\sgmk)$ are in weight $1$. In particular, all of $E^2$ is weight zero and the weight zero piece of $F^2$ is the $\zpl$-module generated by the classes $\sigma_k^{im}$ and $\sigma_k^{(i+1)m-1}d(\sigma_k)$ over $i \geq 0$. 

Since $\phi^2(\sgmmk)  = \sgmk^{m}$, we obtain that  
\[\phi^2(d(\sgmmk)) = d (\phi^2(\sigma_{mk})) =  d(\sigma_k^m)=  m \sgmk^{m-1}d(\sgmk).\]
Therefore, we have 
\[\phi^2(\sigma_{mk}^i) = \sigma_k^{im}\textup{\ \ \ and\  \ } \phi^2(\sigma_{mk}^id(\sigma_{mk})) = m \sigma_k^{(i+1)m-1} d(\sigma_k).\]
Since $p\nmid m$, we have that $m$ is a unit.  Using this, one observes that $\phi^2$ is an isomorphism after restricting and corestricting to weight zero as desired.
\end{proof}

 In the situation of Hypothesis  \ref{hypo root adjunction}, $A \to \artmx$ is a map of $m$-graded $E_1$-rings and $A$ is concentrated in weight zero. Therefore, there is a map
\begin{equation}\label{eq weight zero inclusion thh referee}
    \thh(A) \to \thh(\artmx)_0
\end{equation}
 where $A$ above denotes the underlying $E_1$-ring of $A$. 

\begin{theo}\label{theo weight zero of THH after root adjunction}
Assume Hypothesis \ref{hypo root adjunction} and that $A$ is $p$-local and $p\nmid m$. Then \eqref{eq weight zero inclusion thh referee} is an equivalence.
\end{theo}
\begin{proof}
Recall that $\artmx$ is given by  
\[A \wdgsphpsmk \sphpsk\]
where $A$ and $\sphpsmk$ are
concentrated in weight zero.
Due to Proposition \ref{prop thh of relative smash product}, we have 
\[\thh(\artmx)_0 \simeq \thh(A) \wdg_{\thh(\sphpsmk)} (\thh(\sphpsk)_0)\]
and it follows by Proposition \ref{prop spherical root degree zero thh} that the map  
\[\thh(\sphpsmk)\xrightarrow{\simeq} \thh(\sphpsk)_0\]
is an equivalence. This identifies $\thh(\artmx)_0$ with $\thh(A)$ as desired.
\end{proof}

\section{Adjoining roots and algebraic \texorpdfstring{$K$}{K}-theory}\label{sect adj roots and algebraic k theory}

We now prove Theorem \ref{theo INTRO root adj k theory inclusion} from the introduction. For the rest of this section, assume Hypothesis \ref{hypo root adjunction}. We established that $\artmx$ is an $m$-graded ring spectrum and therefore $\thh(\artmx)$ is an $S^1$-equivariant $m$-graded spectrum, see \eqref{eq graded thh to non graded thh}. One might define $\ntc(\artmx)$ as an $m$-graded spectrum given by: 
\[\ntc(\artmx)_i \simeq \thh(\artmx)_i^{hS^1}.\]

Since $m$ is finite, the underlying spectrum of an $m$-graded spectrum, provided by the functor $D$, is given by a finite coproduct which is equivalent to the corresponding finite product. In particular, $D$ commutes with all limits and colimits. Because of this, we have
\[D(\ntc(\artmx)) \simeq \ntc(D(\artmx))\]
and therefore, we often omit $D$ in our notation.

Similarly, $\tp(\artmx)$ and $(\thh(\artmx)^{tC_p})^{hS^1}$ admit the structure of $m$-graded spectra and these constructions commute with the functor $D$ as above. 
Combining this, with Theorem \ref{theo weight zero of THH after root adjunction}, we obtain the following result. 

\begin{theo}\label{theo degree 0 of negative TC and TP after root adjunction}
Assume Hypothesis \ref{hypo root adjunction}, that $A$ is $p$-local, and that $p\nmid m$. The maps
\begin{gather*}
    \ntc(A) \xrightarrow{\simeq} \ntc(\artmx)_0\\
    \tp(A) \xrightarrow{\simeq} \tp(\artmx)_0\\
    (\thh(A)^{tC_p})^{hS^1}  \xrightarrow{\simeq} ((\thh(\artmx)^{tC_p})^{hS^1})_0
    \end{gather*}
 induced by $A \to \artmx$ are all equivalences.    \qed
\end{theo}

When $A$ is connective and $p$-local, $\artmx$ is also connective and $p$-local. By \cite[Corollary 1.5]{nikolausscholze2018topologicalcyclic} (and the discussion afterwards), the topological cyclic homology of $\artmx$ is defined via the following fiber sequence.
\begin{equation}\label{eq nikolaus scholze fiber sequence}
    \tc(\artmx)  \to \thh(\artmx)^{hS^1} \xrightarrow{\varphi_p^{hS^1}-can} (\thh(\artmx)^{tC_p})^{hS^1}
\end{equation}
As mentioned above, the middle term and the third term above admit canonical splittings into $m$-cofactors. Furthermore, $can$  respects this splitting since it only depends on the $S^1$-equivariant structure of $\thh(\artmx)$.

However, $\tc(\artmx)$ do not necessarily split into $m$-cofactors. This is due to the fact that the Frobenius map does not necessarily respect the grading. Indeed, the Frobenius is given by maps 
\begin{equation*}
    \varphi_p \co \thh(\artmx)_i \to \thh(\artmx)^{tC_p}_{ip},
\end{equation*}
see \cite[Corollary A.9]{antieau2020beilinson}. On the other hand, we obtain the following splitting of $\tc(\artmx)$.

\begin{cons}\label{cons weight zero splitting of TC}
Assume Hypothesis \ref{hypo root adjunction} and that  $A$ is connective and $p$-local where $p \nmid m$.  In this situation, $p$ is a non-zero divisor in $\z/m$. Therefore, the Frobenius map on $\thh(\artmx)$ carries pieces of non-zero weight to non-zero weight pieces. Moreover, $\varphi_p$ carries weight zero to weight zero. Therefore, the map $\varphi_p -can$ splits as a coproduct of their restriction to weight zero and their restriction to non-zero weight. In particular, the fiber sequence in   \eqref{eq nikolaus scholze fiber sequence} admits a splitting as follows. 
\begin{multline*}
\tc(\artmx)_0 \vee \tc(\artmx)_1 \to \thh(\artmx)_0^{hS^1}\vee\thh(\artmx)_{>0}^{hS^1} \\ \xrightarrow{(\varphi_p)_0 -can_0 \vee (\varphi_p)_{>0} - can_{>0}} (\thh(\artmx)_0^{tC_p})^{hS^1} \vee (\thh(\artmx)^{tC_p})_{>0}^{hS^1}
\end{multline*}
Here, $(-)_{>0}$ denotes restriction to weight not equal to $0$. We have 
\begin{equation}\label{eq tc splitting in const}
\tc(\artmx) \simeq \tc(\artmx)_0 \vee \tc(\artmx)_1
\end{equation}
where $\tc(\artmx)_0$ denotes the fiber of the map $(\varphi_p)_0 - can_0$ and $\tc(\artmx)_1$ denotes the fiber of the map $(\varphi_p)_{>0} - can_{>0}$.
\end{cons}

\begin{rema}
There are interesting cases where one obtains further splittings of the topological cyclic homology spectrum $\tc(\artmx)$. For instance, if  $p =1$ in $\z/m$, then and one obtains that $\tc(\artmx)$ splits into $m$-summands. This happens to be the case when $m=p-1$ or when $p$ is odd and $m=2$. We exploit this in Construction \ref{cons splitting of kthry ku} to obtain a splitting of $\tc(ku_p)$ into $p-1$ summands. Moreover, if $m = p^n-1$, then one obtains an underlying $p-1$-grading of $\thh(\artmx)$ by Kan extending through $\z/(p^n-1) \to \z/(p-1)$. This provides a  $p-1$-grading for  $\tc(\artmx)$.
\end{rema}
\begin{theo}\label{theo weight zero splitting for tc}
Assume Hypothesis \ref{hypo root adjunction} with $p\nmid m$ and that $A$ is $p$-local and connective.  Under the equivalence \eqref{eq tc splitting in const}, the canonical map \[\tc(A) \to \tc(\artmx)\]
is equivalent to the inclusion of the first wedge summand. %Namely, this map is given by a composite
%\[\tc(A) \xrightarrow{\simeq} \tc(\artmx)_0 \to \tc(\artmx)_0 \vee \tc(\artmx)_1\simeq \tc(\artmx)\]
%where the first map is an equivalence and the second map is the inclusion of the wedge summand $\tc(\artmx)_0$ defined in Construction \ref{cons weight zero splitting of TC}. 
\end{theo}
\begin{proof}
Since $A$ is concentrated in weight zero, the map $\tc(A) \to \tc(\artmx)$ factors through the map
\begin{equation}\label{eq tc map to weight zero}
    \tc(A) \to \tc(\artmx)_0
\end{equation}
induced by the canonical map $\thh(A) \to \thh(\artmx)_0$. 
 The map $\thh(A) \to \thh(\artmx)_0$ of cyclotomic spectra is an equivalence due to Theorem \ref{theo weight zero of THH after root adjunction}. Considering the construction of $\tc(\artmx)_0$, one observes that this  equivalence induces an equivalence between the fiber sequences defining $\tc(A)$ and $\tc(\artmx)_0$. In other words, \eqref{eq tc map to weight zero} is an equivalence as desired.
\end{proof}

Finally, we obtain the desired splitting for $\kth(\artmx)$.
\begin{theo}[Theorem \ref{theo INTRO root adj k theory inclusion}]\label{theo root adj k theory inclusion}
Assume Hypothesis \ref{hypo root adjunction} with $p\nmid m$ and $k>0$. Furthermore, assume that $A$ is $p$-local and connective. In this situation, the following map 
\[\kth(A) \to \kth(\artmx)\]
is the inclusion of a wedge summand.
\end{theo}
\begin{proof}
Since $\lv a \rv = mk$ and since $k>0$, we have 
\begin{equation}\label{eq pizero of rootadj}
\pi_0 \artmx = \pi_0 A.
\end{equation}
We start by constructing a map of $m$-graded $E_1$-algebras
\begin{equation}\label{eq graded postnikov section for root adj}
\artmx \to H\pi_0A
\end{equation}
that induces an isomorphism on $\pi_0$ where $H\pi_0A$ is concentrated in weight $0$. Weight $0$ Postnikov truncation \cite[Lemma B.0.6]{hahn2020redshift} provides a  map of graded $E_2$-rings $\sphsk \to \sph$ that we consider as a map of $m$-graded $E_2$-rings by left Kan extending through $\z\to\z/m$. 

 This provides a map of $m$-graded $E_1$-rings
\[\artmx \simeq A \wdg_{\sphsmk} \sphsk \to A \wdg_{\sphsmk} \sph\]
(see \eqref{eq unit of the extension of scalars functor}) where the right hand side is concentrated in weight $0$. Postcomposing with the ordinary Postnikov truncation, we obtain \eqref{eq graded postnikov section for root adj}. 

Due to the Dundas-Goodwillie-McCarthy theorem \cite{dundasgwlliemccrthy2013localstrctkthry}, there is a pullback square
\begin{equation*}
    \begin{tikzcd}
    \kth(\artmx) \ar[r] \ar[d] &\tc(\artmx) \simeq \tc(A) \vee \tc(\artmx)_1 \ar[d]\\
    \kth(H\pi_0A) \ar[r] &\tc(H\pi_0A)
    \end{tikzcd}
\end{equation*}
provided by the map \eqref{eq graded postnikov section for root adj}.
The equivalence on the upper right corner follows by Construction \ref{cons weight zero splitting of TC} and  Theorem \ref{theo weight zero splitting for tc}.

The map $\artmx \to H\pi_0A$ induces a map of $m$-graded spectra 
\[f \co \thh(\artmx) \to \thh(H\pi_0A).\]
Since $H\pi_0A$ is concentrated in weight zero, $\thh(H\pi_0A)$ is also concentrated in weight zero. Therefore, the map $f$ is trivial  on $\thh(\artmx)_{>0}$. This shows that the right vertical map above induces the trivial map on $\tc(\artmx)_1$. Using this, we obtain that the pullback square above splits as a coproduct of the  pullback squares
\begin{equation*}
\begin{tikzcd}
    \kth(A) \ar[r]\ar[d] & \tc(A)\ar[d]\\
    \kth(H\pi_0A) \ar[r] &\tc(H\pi_0A)
    \end{tikzcd}
\end{equation*}
and 
\begin{equation*}
\begin{tikzcd}
    \tc(\artmx)_1 \ar[r,"\simeq"]\ar[d] & \tc(\artmx)_1\ar[d]\\
    * \ar[r] &*.
    \end{tikzcd}
\end{equation*}
This shows that 
\[\kth(\artmx) \simeq \kth(A) \vee \tc(\artmx)_1\]
as desired.

\end{proof}

Using the Purity theorem for algebraic $K$-theory, we obtain the following for non-connective $A$.

\begin{corr}\label{corr nonconnective alg kthry inclusion for root adj}
Assume Hypothesis \ref{hypo root adjunction} with $p\nmid m$ and $k>0$. If $A$ is $p$-local, then the  map 
\[L_{T(i)}\kth(A) \to L_{T(i)}\kth(\artmx)\]
is the inclusion of a wedge summand for every $i\geq 2$. In particular, if $A$ is of height larger than $0$ and $A$ satisfies the redshift conjecture, then $\artmx$ also satisfies the redshift conjecture.

\end{corr}
\begin{proof}
    Let $cA$ denote the connective cover of $A$ in $\sphsmk$-algebras. We consider the following commuting diagram of $m$-graded $E_1$-rings.
\begin{equation*}
    \begin{tikzcd}
        cA  \ar[r]\ar[d] &A \ar[d]\\
        (cA)(\sqrt[m]{a}) \ar[r]& \artmx
    \end{tikzcd}
\end{equation*}
Every spectrum with bounded above homotopy is $T(i)$-locally trivial for every $i\geq1$. Taking fibers, one obtains that the horizontal arrows above are $T(i)$-equivalence for every $i\geq1$, see \eqref{eq weight pieces of root adj}. 

It follows by \cite[Purity Theorem]{land2020purity} that the horizontal maps above induce $T(i)$-equivalences in algebraic $K$-theory for every $i\geq 2$. The result follows by applying Theorem \ref{theo root adj k theory inclusion} to the left vertical map. 
\end{proof}

\section{A variant of  logarithmic THH}\label{sect log thh}

Here, we introduce our definition of logarithmic THH and identify $\thh(\artmx)$ using $\thh(A)$ and logarithmic THH of $A$ whenever $A$ is $p$-local and $p \nmid m$. Through our definition, logarithmic THH admits a canonical structure of a cyclotomic spectrum; in upcoming work, Devalapurkar and the third author develop a very general notion of logarithmic structures for $E_2$-algebras and a corresponding theory of log THH  which subsumes the definition we use here. This will in particular recover the variant due to Rognes, which is defined by way of the replete bar construction, cf. \cite{rognes2009topologicallogarithmic,rognes2015localization}

Our definition of log THH starts with a definition of the log THH of the free algebra $\sph[\sigma_k]$ where $k\geq0$ is even as before. We consider $\sigma_k$ to be in weight 1.

For a graded $E_n$-ring spectrum $E$, we denote the \emph{weight connective cover} of $E$ by $E_{\geq 0}$. Indeed, the weight connective cover is obtained by  restricting and then left Kan extending through the inclusion $\mathbb{N} \to \Z$. The counit of this adjunction provides a map $E_{\geq 0}\to E$ of graded $E_n$-algebras.  

\begin{cons}\label{cons spherical laurant polynomials}
  Analogous to Construction \ref{cons of graded e2 polynomials}, let $\sph[\sigma_k^{\pm}]:= \on{sh}^k(\sph[t^\pm])$. The graded $E_\infty$-map $\sph[t] \to \sph[t^\pm]$  provides a graded $E_2$-map $\sph[\sigma_k] \to \sph[\sigma_k^\pm]$. Furthermore, by the definition of the shearing functor, $\sph[\sigma_k^\pm]$ is indeed given by $\phi^k$ of Variant \ref{var higher shearing}; in particular, $\sph[\sigma_{mk}^\pm]$ is the restriction of $\sph[\sigma_k^\pm]$ along $\z \xrightarrow{\cdot m} \z$. Applying the adjunction $L_m \dashv R_m$ induced by $\cdot m$ to the map $\sphsmk \to \sphsk$, one observes that $\sphsmk$ is also the restriction of $\sphsk$ along $\cdot m$. Therefore, the counit of  $L_m \dashv R_m$ provides a commutative diagram of graded $E_2$-rings:
\begin{equation*}
    \begin{tikzcd}
    \sphsmk \ar[r]\ar[d] & \ar[d] \sph[\sigma_{mk}^\pm]\\
    \sphsk \ar[r] & \sph[\sigma_k^\pm].
    \end{tikzcd}
\end{equation*}

\begin{comment}
{\blue Let $\sph[\sigma_k]$ be as above. Since this is an $E_2$ algebra, the sequence 
$$
 \sph[\sigma_k] \to \Sigma^{-2k}\sph[\sigma_k]  \to \Sigma^{-4k}\sph[\sigma_k] \to \cdots
$$
may be regarded as a diagram of right $\sph[\sigma_k]$ modules. We let $\Sph[\sigma_k^\pm]$ denote the colimit of this sequence. This procedure will be compatible with the monoidal structure on $\on{RMod}_{\sph[\sigma_{k}]}$, hence $\Sph[\sigma_k^\pm]$ acquires an $E_1$ algebra structure. By \cite[Remark 3.5.6]{lurie2015rotation}, the resulting $\infty$-category $\on{RMod}_{\sph[\sigma_{k}^\pm]}$ will acquire a central action of $\on{Gr}(\Sp)$, endowing $\sph[\sigma_{k}^\pm]$ with an $E_2$ -algebra structure. Furthermore, the resulting map $\sph[\sigma_{k}] \to \sph[\sigma_{k}^\pm]$ is one of $E_2$ algebras.

Now, let  $\sph[\sigma_{mk}] \to \sph[\sigma_{k}]$  be the map of Proposition \ref{prop map between free spheres}. By construction, the composite map of $E_2$-algebras
$$
\sph[\sigma_{mk}] \to \sph[\sigma_{k}] \to \sph[\sigma_{k}^\pm],
$$
will factor through the $E_2$-algebra map $\sph[\sigma_{mk}] \to \sph[\sigma_{mk}^\pm]$, so there will be an induced map $\sph[\sigma_{mk}^
\pm] \to \sph[\sigma_{k}^\pm]$. }

\end{comment}
\end{cons}

\begin{rema}\label{rema map from spherical thh to spherical log thh}
One may also take weight connective covers in the $\infty$-category of graded $S^1$-equivariant spectra by using the left Kan extension/restriction adjunction induced by $\mathbb{N}^{ds} \to \Z$. This provides a map 
\[\thh(\sph[\sigma_k]) \to \thh(\sph[\sigma_k^\pm])_{\geq 0}\]
of graded $E_1$-algebras in $S^1$-equivariant spectra factoring the map $   \thh(\sph[\sigma_k]) \to \thh(\sph[\sigma_k^\pm])$. 
\end{rema}

The following is analogous to the description of the replete bar construction of commutative $\mathcal{J}$-space monoids generated by a single element; c.f.\ \cite[Proposition 3.21]{rognes2009topologicallogarithmic},  \cite[Section 8.5]{sagave2019virtualandgddthom}  and \cite[Sections 6 and 7]{rognes2015localization}.

\begin{defi}\label{defi logarithmic thh of free spherical algebras}
Let $k \geq 0$ be even. The logarithmic  THH of $\sph[\sigma_k]$ with respect to $\sigma_k \in \pi_k \sph[\sigma_k]$ is the weight connective cover of the topological Hochschild homology of $\sph[{\sigma_k}^{\pm}]$.  In other words, it is the $S^1$-equivariant $E_1$-algebra:
\[\thh(\sph[\sigma_k] \mid \sigma_k) := \thh(\sph[\sigma_k^\pm])_{\geq 0}.\]
Similarly, the $p$-local counterpart is defined as follows. 
\[\thh(\sphpsk \mid \sigma_k) := \thh(\sphpl[\sigma_k^\pm])_{\geq 0}\]
\end{defi}

The following example provides a justification for this definition of logarithmic THH by showing that its $\hz$-homology provides what should be the logarithmic Hochschild homology of the free algebra $\z[\sigma_k]$,  c.f.\ \cite[Example 10.3]{krause2019b}. 

\begin{exam}\label{exam log thh of the free polynomial  algebra}
Considering $\hz$ as a graded $E_\infty$-algebra concentrated in weight $0$, we deduce that 
\[\hz \wdg (\thh(\sph[\sigma_k^\pm])_{\geq 0})\simeq (\hz \wdg \thh(\sph[\sigma_k^\pm]))_{\geq 0}.\]
Therefore, $\hz_*\thh(\sph[\sigma_k]\mid \sigma_k)$ is given by the weight connective cover of 
\begin{equation}\label{eq dg laurant poly thh}
    \thh^{\hz}_*(\hz[\sigma_k^\pm]) \cong \z[\sigma_k^\pm] \otimes \Lambda(d \sigma_k)
\end{equation}
where $d \sigma_k$ is of weight $1$ and  degree $k+1$ and $\sigma_k$ is of weight $1$ and degree $k$. The isomorphism above follows by the usual B\"okstedt spectral sequence considerations applied together with the HKR theorem. Taking the weight connective cover of \eqref{eq dg laurant poly thh}, we obtain:
\[\hz_*\thh(\sph[\sigma_k] \mid \sigma_k) \cong \z[\sigma_k] \otimes \Lambda( \textup{dlog} \sigma_k)\]
where $\text{dlog} \sigma_k$ is of weight $0$ and homotopical degree $1$ and it corresponds to $(d \sigma_k) /\sigma_k$. Furthermore, the map 
\[\hz_* \thh(\sphsk) \to \hz_*\thh(\sph[\sigma_k] \mid \sigma_k)\]
carries $d\sigma_k$ to $d\sigma_k = \sigma_k\text{dlog}\sigma_k$.
\end{exam}

Recall from Construction \ref{cons thh of an algebra is a module} that when $A$ is a $\sphsk$-algebra, $\thh(A)$ admits the structure of a right $\thh(\sphsk)$-module. We use this structure in the following definition. Recall that Proposition \ref{prop hypothesis} provides various cases of interest where the assumptions on $A$ in the following definition are satisfied. 
\begin{defi}\label{defi log thh of general X}
Let $A$ be an $E_1$ $\sph[\sigma_k]$-algebra and assume that the unit map $\sph[\sigma_k] \to A$ carries $\sigma_k \in \pi_k \sph[\sigma_k]$ to $a \in \pi_k A$ with even $k \geq 0$. We define the logarithmic THH of $A$ relative to $a$ as the following $S^1$-equivariant spectrum. 
\[\thh(A \mid a) := \thh(A) \wdg_{\thh(\sphsk)} \thh(\sphsk \mid \sigma_k)\]
If $A$ is assumed to be $p$-local, we use the following equivalent definition
\[\thh(A \mid a) := \thh(A) \wdg_{\thh(\sphpsk)} \thh(\sphpsk \mid \sigma_k).\]
\end{defi}

The definition of logarithmic THH we provide above is analogous to the definitions used in \cite{rognes2009topologicallogarithmic,sagave2019virtualandgddthom,rognes2015localization}.
\begin{rema}
We remark that $\thh(\sphsk \mid \sigma_k)$ should be a cyclotomic spectrum as the Frobenius maps of THH multiply the weight by $p$ and this should provide $\thh(A \mid a)$ above with the structure of a cyclotomic spectrum. However, since we don't explicitly need this for our application, displaying this will take us too far afield, and so, we leave the details to the future work of Devalapurkar and the third author. 
%{\red We are supposed to be able to do this log thh cyclotomic. We can remark that the behaviour of the frobenius with respect to weight grading show that the log thh of $\sphsk$ does admit the structure of a cyclotomic spectrum and furthermore, it is possible to make the map $\thh(\sphsk) \to \thh(\sphsk \mid \sigma_k)$ a map of $E_1$-algebras in cyclotomic spectra by using the ideas in \cite[Appendix A]{antieau2020beilinson} which makes $\thh(X \mid x )$ a cyclotomic spectrum. Since we don't use the cyclotomic structure on log THH we leave the details to the future work of Sanath and Tasos.  }
\end{rema}
\begin{comment}
\begin{prop}\label{prop plocal log thh of X}
In the situation of Definition \ref{defi log thh of general X}, if $A$ is $p$-local, then 
\[\thh(A) \wdg_{\thh(\sphsk)} \thh(\sphsk \mid \sigma_k)\simeq \thh(A) \wdg_{\thh(\sphpsk)} \thh(\sphpl[\sigma_k^\pm])_{\geq 0}.\]
\end{prop}
\begin{proof}
Since $p$-localization is a smashing localization,
 we have the following equivalences of  $E_1$-algebras:
\[\thh(\sphpl[\sigma_{k}^{\pm}])_{\geq 0} \simeq \thh(\sphpl) \wdg \big( \thh(\sph[\sigma_k^{\pm}])_{\geq 0 }\big) \]
and 
\[\thh(\sphpsk) \simeq \thh(\sphpl) \wdg \thh(\sphpsk).\]
Here, $\sphpl$ is concentrated in weight $0$. Due to the relative smash product definition of THH, $\thh(\sphpl) \simeq \sphpl$ as $\sphpl \wdg \sphpl \simeq \sphpl$. Rest of the proof follows as in the proof of  Proposition \ref{prop p local root adjunction alt def}.
\end{proof}
\end{comment}

 Since the definition of logarithmic THH  is given by the extension of scalars functor: 
\[- \wdg_{\thh(\sphsk)} \thh(\sphsk \mid \sigma_k) \co \rmod_{\thh(\sphsk)} \to \rmod_{\thh(\sphsk \mid \sigma_k)},\]
corresponding to the  $E_1$-algebra map $\thh(\sph[\sigma_k]) \to \thh(\sph[\sigma_k]\mid \sigma_k)$, we deduce that $\thh(A \mid a)$ is equipped with the structure of a  right $\thh(\sph[\sigma_k] \mid \sigma_k)$-module. Furthermore, the unit of the adjunction given by the extension of scalars functor above and the corresponding forgetful functor provides a map 
\begin{equation}\label{eq map from thh to logthh}
    \thh(A)\to \thh(A \mid a)
\end{equation}
of right $\thh(\sphsk)$-modules, see \eqref{eq unit of the extension of scalars functor}.

\begin{rema}
Using $MU[\sigma_k]$ mentioned in Remark \ref{rema bpnadj root is ethree}, it is possible to equip logarithmic THH with the structure of an $E_n$-algebra for $n>0$ in favorable cases. For instance,  for the $E_3$ $MU[\sigma_{2(p^n-1)}]$-algebra form of $\bpn$ constructed in \cite{hahn2020redshift},  $\thh(\bpn \mid v_n)$ admits the structure of an $E_1$-ring. Indeed, using the map of $E_2$-rings $\thh(MU[\sigma_{2(p^n-1)}]) \to \thh(\bpn)$, we obtain an $E_1$-ring:
\[\thh(\bpn) \wdg_{\thh(MU[\sigma_{2(p^n-1)}])} \thh(MU[\sigma_{2(p^n-1)}^\pm])_{\geq 0},\]
  equivalent to $\thh(\bpn \mid v_n)$. This equivalence follows by the following chain of equivalences 
\begin{equation*}
    \begin{split}
   \thh&(\bpn) \wdg_{\thh(MU[\sigma_{2(p^n-1)}])} \thh(MU[\sigma_{2(p^n-1)}^\pm])_{\geq 0}  \\
   &\simeq \thh(\bpn) \wdg_{\thh(MU) \wdg \thh(\sph[\sigma_{2(p^n-1)}])} \thh(MU) \wdg \thh(\sph[\sigma_{2(p^n-1)}^\pm])_{\geq 0} \\
   &\simeq \thh(\bpn)
   \wdg_{\thh(\sph[\sigma_{2(p^n-1)}])} \thh(\sph[\sigma_{2(p^n-1)}])
    \end{split}
\end{equation*}
obtained from the equivalence of $E_2$ $MU$-algebras $MU[\sigma_{2(p^n-1)}] \simeq MU \wdg \sph[\sigma_{2(p^n-1)}]$ mentioned in Remark \ref{rema bpnadj root is ethree}.

Furthermore, Hahn and  Yuan   \cite[1.11 and 1.12]{hahn2020exotic} show that 
there is an $E_\infty$-map $MU[\sigma_2]\to ku_p$ for $p=2$ and claim that their methods provide such a map for odd primes too. In this situation, $\thh(ku_p \mid u_2)$ is equipped with the structure of an $E_\infty$-ring (by arguing as above) where $u_2$ denotes the Bott element. Note that the logarithmic THH of $ku_p$ relative to $u_2$ is also constructed as an $E_\infty$-ring in \cite{rognes2018logthhofku}. 

\end{rema}

\begin{rema}
In work in progress, S. Devalapurkar and the second author show in a general context, that for every $E_2$-ring with even homotopy, logarithmic THH, as in our definition, may be equipped with a canonical $E_1$-algebra structure in cyclotomic spectra. 
\end{rema}

To study the log THH of $E_1$-rings obtained via root adjunctions, we use the following constructions.

\begin{cons} \label{cons map to the weight connective covers}
\begin{comment}
As mentioned earlier, the canonical inclusion $\mathbb{N} \to \z$ induces an adjunction, given by  restriction and left Kan extension functors, between the $\infty$-categories of  spectra graded over $\mathbb{N}$ and graded spectra. Through this adjunction, one obtains a map of $E_1$-ring  spectra 
\[\thh(\sph[\sigma_k]) \to \thh(\sph[\sigma_k]\mid \sigma_k)\]
factoring the canonical map
\[\thh(\sph[\sigma_k]) \to \thh(\sph[\sigma_k^{\pm}]).\]
\end{comment}
Using Construction \ref{cons spherical laurant polynomials} and the weight connective cover adjunction mentioned in Remark \ref{rema map from spherical thh to spherical log thh}, we obtain the following commuting diagram of graded $E_1$-algebras. 
\begin{equation}\label{diag thh to log thh is natural}
    \begin{tikzcd}
    \sphsmk \ar[d] \ar[r] &\thh(\sph[\sigma_{mk}])\ar[d] \ar[r]& \thh(\sph[\sigma_{mk}]\mid \sigma_{mk})\ar[d]\\
    \sphsk \ar[r] &\thh(\sph[\sigma_k]) \ar[r]& \thh(\sph[\sigma_k]\mid \sigma_k)
    \end{tikzcd}
\end{equation}
\end{cons}

\begin{cons}\label{cons graded structure on log thh}
Assume Hypothesis \ref{hypo root adjunction}. By Proposition \ref{prop thh of relative smash product}, there is an equivalence:
\[\thh(\artmx) \simeq \thh(A) \wdg_{\thh(\sphsmk)} \thh(\sphsk).\]
This equips $\thh(\artmx)$ with the structure of a right $\thh(\sphsk)$-module in $m$-graded spectra. Considering the map
\[\thh(\sphsk) \to \thh(\sphsk \mid \sphsk)\]
as a map of $m$-graded $E_1$-ring spectra,  
Definition \ref{defi log thh of general X} may be employed at the level of $m$-graded spectra. This  shows that $\thh(\artmx \mid \sqrt[m]{a})$ admits a canonical structure of a right $\thh(\sphsk \mid \sigma_{k})$-module in $m$-graded spectra. Furthermore, the map $\thh(\artmx) \to \thh(\artmx \mid \sqrt[m]{a})$ is a map of $m$-graded $\thh(\sphsk)$-modules.
\end{cons}

\begin{comment}

\begin{defi}
For an $E_2$-algebra $Z$, let $X$ be a $Z$-algebra, we defined the relative THH of $X$ with respect to $Z$ via the following formula 
\[\thh^Z(X) = \thh(X) \wdg_{\thh(Z)} Z \]
where the smash product is taken using the map provided by Proposition \ref{prop map to the connective cover thh}.
\end{defi}

\end{comment}

\subsection{Logarithmic THH-\'etale root adjunctions}

Here, our goal is to show that when $A$ is $p$-local and $p \nmid m$, root adjunction is logarithmic THH-\'etale. In other words, we show that there is an equivalence of $m$-graded spectra:
\[  \thh(A \mid a)\wdg_{\sphpsmk} \sphpsk\simeq  \thh(\artmx \mid \sqrt[m]{a}).\]
%\[\artmx = X \wdg_{\sph[\sigma_k]} \sph[\sigma_{mk}]\]
 \begin{rema}
A notion of logarithmic THH-\'etalenes is already defined in \cite{rognes2018logthhofku}. In the language of Rognes, Sagave and Schlichtkrull \cite{rognes2018logthhofku}, logarithmic THH-\'etaleness of $A\to \artmx$ would be  expressed by an equivalence:
\begin{equation}\label{eq RSS log thh etaleness}
\thh(\artmx \mid \sqrt[m]{a}) \simeq \artmx \wdg_A \thh(A \mid a).
\end{equation}
Since we only assume $A$ to be $E_1$, $\thh(A\mid a)$ may not admit an $A$-module structure and therefore, the right hand side above may not be defined in our generality. On the other hand, if one starts with an $E_3$-algebra $A$ with even homotopy, the logarithmic THH of $A$ may be given an $A$-module structure and we obtain that $A\to \artmx$ is logarithmic THH-\'etale in the sense of \eqref{eq RSS log thh etaleness} whenever $A$ is $p$-local and $p \nmid m$. 
\end{rema}

\begin{prop}\label{prop log thh of free are connective }
For $k \geq 0$, the spectra $\thh(\sphsk \mid \sigma_k)$ and $\thh(\sphpsk \mid \sigma_k)$ are connective in homotopy. 
\end{prop}
\begin{proof}
This follows by the fact that the weight connective part of the cyclic bar construction on $\sph[\sigma_k^\pm]$ ($\sphpl[\sigma_k^\pm]$) is connective in homotopy in each simplicial degree.  
\end{proof}
 We start with  proving a logarithmic THH \'etaleness result  for the $p$-localized  free $E_1$-algebra $\sphpsmk$.
\begin{prop}\label{prop free algebra log thh etale}
Let $k\geq 0$ be even  and let $m>0$ with  $p \nmid m$. In this situation,  there is an equivalence  of left  $\thh(\sphpl[\sigma_{mk}] \mid \sigma_{mk})$-modules in $m$-graded spectra:
\[ \thh(\sphpsmk \mid \sigma_{mk})  \wdg_{\sphpsmk} \sphpsk \simeq  \thh(\sphpsk \mid \sigma_k) \]
where the left $\thh(\sphpl[\sigma_{mk}] \mid \sigma_{mk})$-module structure on the right hand is provided by Construction \ref{cons map to the weight connective covers}.
\end{prop}
\begin{proof}
We start by constructing the desired map. First, there is a composite map of $m$-graded $E_1$-ring spectra, 
\begin{equation}\label{eq unit map}
\sphpsk \to \thh(\sphpsk) \to \thh(\sphpsk \mid \sigma_k)
\end{equation}
which is in particular a map of left $\sphpsmk$-modules in $m$-graded spectra by forgetting structure trough the $m$-graded $E_1$-ring map $\sphpsmk \to \sphpsk$. Using the extension of scalars functor induced by the map 
\begin{equation}\label{eq left module structure}
\sphpl[\sigma_{mk}] \to \thh(\sphpl[\sigma_{mk}]\mid \sigma_{mk})
\end{equation}
of $m$-graded $E_1$-algebras, we obtain the desired map:
\begin{equation*}
    f \co \thh(\sphpsmk \mid \sigma_{mk})  \wdg_{\sphpsmk} \sphpsk \to \thh(\sphpsk \mid \sigma_k),
\end{equation*}
of left $\thh(\sphpsmk \mid \sigma_{mk})$-modules in $m$-graded spectra  from \eqref{eq unit map}. Here, we used the fact that the  left  $\sphpl[\sigma_{mk}]$-module structure on $\thh(\sphpsk \mid \sigma_{k})$ used in \eqref{eq unit map} is compatible with the one obtained by forgetting the canonical left $\thh(\sphpsmk \mid \sigma_{mk})$-module structure on $\thh(\sphpsk \mid \sigma_{k})$ through \eqref{eq left module structure}; this follows by the $p$-local version of Diagram \eqref{diag thh to log thh is natural}.

What remains is to show that $f$ is an equivalence. Since $f$ is a map between $p$-local connective spectra (Proposition \ref{prop log thh of free are connective }),  it is sufficient to show that $\hzpl \wdg f$ is an equivalence, see Lemma \ref{lemm map of bounded below spectra}. 

By inspection on the two sided bar construction defining relative smash products, one obtains that 
\begin{align*}
\hzpl \wdg \thh(\sphpsmk \mid \sigma_{mk}) & \wdg_{\sphpsmk} \sphpsk \simeq \\
&(\hzpl \wdg \thh(\sphpsmk \mid \sigma_{mk}) ) \wdg_{\hzpl[\sigma_{mk}]} \hzpl[\sigma_{k}].
\end{align*}
Using the base change formula for THH, we obtain that $\hzpl \wdg f$ is given by the canonical map 
\[\hzpl \wdg f \co   \thh^{\hzpl}(\hzpl[\sigma_{mk}^\pm])_{\geq 0} \wdg_{\hzpl[\sigma_{mk}]} \hzpl[\sigma_{k}]  \to \thh^{\hzpl}(\hzpl[\sigma_k^\pm])_{\geq 0}.\]

To prove that $\hzpl \wdg f$ is an equivalence, we argue as in the proof of Proposition  \ref{prop spherical root degree zero thh}. The  map of B\"okstedt spectral sequences computing the map  
\begin{equation}\label{eq map between laurant polynomials}
     \thh^{\hzpl}_*(\hzpl[\sigma_{mk}^\pm])  \to \thh^{\hzpl}_*(\hzpl[\sigma_k^\pm])
\end{equation}
is given on the second page, due to the HKR theorem, by the ring map 
\begin{equation}\label{eq bokstedt ss for laurant polynomials}
\phi \co \zpl[\sgmmk^\pm]\otimes \Lambda_{\zpl}(d(\sgmmk)) \to \zpl[\sgmk^\pm] \otimes \Lambda_{\zpl}(d(\sgmk))
\end{equation}
satisfying 
\[\phi(\sigma_{mk}) = \sigma_k^m \textup{ \ and \ } \phi(d(\sigma_{mk})) = m \sgmk^{m-1}d(\sgmk)\]
where $m$ is a unit as $p \nmid m$. In particular, 
\begin{equation}\label{eq image of the log derivative}
    \phi(\sigma_{mk}^{-1}d(\sigma_{mk})) = \sigma_k^{-m} m \sigma_k^{m-1} d(\sigma_k) = m \sigma_k^{-1}d(\sigma_k).
\end{equation}

Here, $\sigma_{mk}$ and $\sigma_k$ are in degrees $(0,mk)$ and $(0,k)$ respectively and $d(\sigma_{mk})$ and $d(\sigma_k)$ are in degrees $(1,mk)$ and $(1,k)$ respectively. Furthermore, $\sigma_{mk}$ and $d(\sigma_{mk})$ are of weight $m$ and $\sigma_k$ and $d(\sigma_k)$ are of weight $1$. In particular, both B\"okstedt spectral sequences degenerate on the second page and the map $\phi$ provides the map \eqref{eq map between laurant polynomials}.

Taking connective covers in weight direction and identifying $\sigma_{mk}^{-1} d(\sigma_{mk})$ as $\text{dlog}\sigma_{mk}$ and $\sigma_k^{-1} d(\sigma_k)$ as $\text{dlog}\sigma_k$, we obtain that the map 

 \[\thh^{\hzpl}_*(\hzpl[\sigma_{mk}^\pm])_{\geq 0}  \to \thh^{\hzpl}_*(\hzpl[\sigma_{k}^\pm])_{\geq 0}\]
is given by a  map 
\[\zpl[\sigma_{mk}] \otimes \Lambda_{\zpl}(\textup{dlog} \sigma_{mk}) \to \zpl[\sigma_{k}] \otimes \Lambda_{\zpl}(\textup{dlog}\sigma_{k}) \]
that carries $\sigma_{mk}$ to $\sigma_{k}^m$ and $\text{dlog}\sigma_{mk}$ to $\text{dlog}\sigma_k$ up to a unit due to \eqref{eq image of the log derivative} as $p \nmid m$.  

Upon extending  scalars with respect to the map $\z_{(p)}[\sigma_{mk}] \to \z_{(p)}[\sigma_{k}]$, this map becomes an isomorphism. In other words, $\pis(\hzpl \wdg f)$ is an isomorphism and therefore, $f$ is an equivalence.

\end{proof}

The following provides the logarithmic THH-\'etaleness of root adjunction in ring spectra.

\begin{theo}\label{theo root adj is log thh etale }
Assume Hypothesis \ref{hypo root adjunction} with $p\nmid m$ and that $A$ is $p$-local.  In this situation,  there is an equivalence of $m$-graded spectra 
\[  \thh(\artmx \mid \sqrt[m]{a})\simeq \thh(A \mid a)\wdg_{\sphpsmk} \sphpsk  . \]
In other words, as an $m$-graded spectrum, $\thh(\artmx \mid \sqrt[m]{a})$ is given by 
\[\thh(\artmx \mid \sqrt[m]{a})_i \simeq  \Sigma^{ik} \thh(A\mid a) \]
for every $0\leq  i<m$.
\end{theo}
\begin{proof}
We have the following chain of equivalences
\begin{equation*}
    \begin{split}
    &\thh(\artmx \mid \sqrt[m]{a}) \simeq \thh(\artmx) \wdg_{\thh(\sphpsk)} \thh(\sphpsk \mid \sigma_k)\\
    &\simeq \big(\thh(A) \wdg_{\thh(\sphpsmk)}\thh(\sphpsk) \big) \wdg_{\thh(\sphpsk)} \thh(\sphpsk \mid \sigma_k)\\
    &\simeq \thh(A) \wdg_{\thh(\sphpsmk)} \thh(\sphpsk \mid \sigma_k)\\
    & \simeq \thh(A) \wdg_{\thh(\sphpsmk)} \thh(\sphpsmk \mid \sigma_{mk})  \wdg_{\thh(\sphpsmk \mid \sigma_{mk})} \thh(\sphpsk \mid \sgmk)\\
    & \simeq \thh(A \mid a) \wdg_{\thh(\sphpsmk \mid \sgmmk)} \thh(\sphpsk \mid \sgmk)\\
    & \simeq \thh(A \mid a) \wdg_{\thh(\sphpsmk\mid \sigma_{mk})} \big( \thh(\sphpsmk \mid \sigma_{mk}) \wdg_{\sphpsmk} \sphpsk \big)\\
    & \simeq \thh(A\mid a)  \wdg_{\sphpsmk} \sphpsk\\
    \end{split}
\end{equation*}
The first and the fifth  equivalences follow by the definition of logarithmic THH, the second equivalence follows by our definition of root adjunction and Proposition \ref{prop thh of relative smash product} and the sixth equivalence follows by Proposition \ref{prop free algebra log thh etale}.
\end{proof}
\begin{rema}
     In \cite{rognes2018logthhofku}, the authors show that $\ell \to ku_{(p)}$ is logarithmic THH-\'etale. This compares to our result above since we show that $ku_p \simeq \ell_p(\sqrt[p-1]{v_1})$ in Theorem \ref{theo complex kthry as root adjunction}. 
\end{rema}

%If one starts with an $E_3$-ring $A$ with even homotopy, then $\thh(\sphsmk) \to \thh(A)$ can be chosen to be an $E_1$-map. In this case, $\thh(A\mid a)$ is a left $\thh(A)$-module; in particular, a left  $A$-module. We deduce that the equivalence above holds under the hypothesis of Theorem \ref{theo root adj is log thh etale } if $A$ is an $E_2$-ring with even homotopy. 

\subsection{Relating THH and logarithmic THH}

The goal of this section is to show that there is a fiber sequence 
\[\thh(A) \to \thh(A \mid a) \to \Sigma \thh(A/a)\]
under our usual assumptions. The $E_1$-ring $A/a$ above is described in the following construction which is analogous to \cite[Lemmas 6.14 and 6.15]{rognes2015localization}. 
\begin{cons}\label{cons quotient algebra}
Let $A$ be an $\sphsk$-algebra where $\sigma_k$ acts through $a \in \pi_kA$ where $k \geq 0$ is even. The weight $0$ Postnikov section  of $\sphsk$ provides a map $\sphsk \to \sph$ of $E_2$-rings \cite[B.0.6]{hahn2020redshift}. Considering the extension of scalars functor $- \wdg_{\sphsk}\sph$ from the $\infty$-category of  $E_1$  $\sphsk$-algebras to the $\infty$-category of $E_1$ $\sph$-algebras, one equips 
\[A/a := A \wdg_{\sphsk} \sph \]
with the structure of an $E_1$-ring spectrum. Since $\sph$ is the cofiber of the map $\sphsk \xrightarrow{\cdot \sigma_k } \sphsk$, $A/a$ is indeed the cofiber of the map $A \xrightarrow{\cdot a} A$.
\end{cons}

Considering $\sphsk$ as a graded $E_2$-ring, we have 
\[\thh(\sphsk)_0 \simeq \sph.\]
This can be observed by inspection on the cyclic bar construction on $\sphsk$ or by computing the $\hz$-homology of the left hand side above. This is used in the statement of the following proposition.

\begin{rema}
 The following proposition is analogous to \cite[Proposition 6.11]{rognes2015localization}. We remark that unlike in loc cit., we do not take $S^1$-equivariance into account, which leads to a  simpler proof. 
\end{rema}
\begin{prop}\label{prop sphere thh to log thh cofiber}
The cofiber of the map 
\[\thh(\sph[\sigma_k]) \to \thh(\sphsk \mid \sigma_k)\]
 is given by $\Sigma \sph$ concentrated in weight  $0$ as a left $\thh(\sph[\sigma_k])$-module in graded spectra. Here, the left $\thh(\sph[\sigma_k])$-module structure on $\sph$ is given by the weight-Postnikov truncation map of graded $E_1$-rings
 \[\thh(\sph[\sigma_k]) \to \thh(\sph[\sigma_k])_{0} \simeq \sph.\]
\end{prop}
\begin{proof}
Let $M$ be the cofiber of the map $f$ below in left $\thh(\sphsk)$-modules in graded spectra.
\[ \thh(\sph[\sigma_k]) \xrightarrow{f}   \thh(\sphsk, \mid \sigma_k) \to M\]
We start by computing $\hz_*M$. The map 
\[\hz_*f \co \thh^{\hz}_*(\hz[\sigma_k]) \to \thh^{\hz}_*(\hz[\sigma_k^\pm])_{\geq0}\]
is given by the ring map 
\[\z[\sigma_k] \otimes \Lambda(d(\sigma_k)) \to \z[\sigma_k] \otimes \Lambda(\textup{dlog}\sigma_k)\]
that carries $\sigma_k$ to $\sigma_k$ and $d(\sigma_k)$ to $\sigma_k \text{dlog}\sigma_k$; this follows by the B\"oksedt spectral sequences in \eqref{eq bokstedt ss for polynomial dgas} and  \eqref{eq bokstedt ss for laurant polynomials}. This map is injective and the only class that is not in the image is $\text{dlog} \sigma_k$. We obtain, 
\[\hz \wdg M \simeq \Sigma \hz\]
where the right hand side is concentrated in weight $0$.
Due to Proposition \ref{prop log thh of free are connective }, $f$ is a map between connective spectra. In particular, $M$ is connective and we obtain an equivalence of spectra 
\[M \simeq \Sigma \sph.\]

We need to improve this to an equivalence of  left $\thh(\sphsk)$-modules in graded spectra. Since $M$ is a left $\thh(\Sph[\sigma_{k}])$-module in graded spectra, there is a map 
\[\Sigma \thh(\sphsk) \to M\]
of  left $\thh(\sphsk)$-modules in graded spectra carrying $1$ to $1$ in homotopy. Taking weight $0$ Postnikov sections \cite[B.0.6]{hahn2020redshift}, we obtain an equivalence  of left $\thh(\sphsk)$-modules in graded spectra 
\[ \Sigma \thh(\sphsk)_0 \xrightarrow{\simeq } M.\]
This map is an equivalence because it carries $1$ to $1$ in homotopy by construction and since both sides are equivalent as spectra to $\Sigma \sph$.

\end{proof}

We are ready to provide the cofiber sequence relating THH to logarithmic THH.
\begin{theo}\label{theo cofiber seq for thh to log thh}
Let $A$ be an $\sphsk$-algebra where $\sigma_k$ acts through $a \in \pi_kA$ with even $k \geq 0$. In this situation, there is a cofiber sequence of spectra:
\[\thh(A) \to \thh(A \mid a) \to \Sigma \thh(A/a).\]
The corresponding cofiber sequence for $\thh(\artmx \mid \sqrt[m]{a})$ is a cofiber sequence of $m$-graded spectra.
\end{theo}
\begin{proof}
 Proposition \ref{prop sphere thh to log thh cofiber} provides the following cofiber sequence of  left $\thh(\sphsk)$-modules in graded spectra.
\[\thh(\sphsk) \to \thh(\sphsk \mid \sigma_k) \to \Sigma \sph\]
 Applying the functor $\thh(A) \wdg_{\thh(\sphsk)} -$ to this cofiber sequence, we obtain the following cofiber sequence
\begin{equation}\label{eq first eq in pf of cof seq of log thh}
    \thh(A) \to \thh(A \mid a) \to \thh(A) \wdg_{\thh(\sphsk)} \Sigma \sph.
\end{equation}

What is left is to identify the cofiber above as $\thh(A/a)$. We have 
\begin{equation}\label{eq moving suspension out}
    \begin{split}
        \thh(A) \wdg_{\thh(\sphsk)} \Sigma \sph &\simeq \Sigma \thh(A) \wdg_{\thh(\sphsk)}  \sph.\\
    \end{split}
\end{equation}
Here, $\sph$ on the right hand side denotes the de-suspension of $\Sigma\thh(\sphsk)_0$ as a left $\thh(\sphsk)$-module in graded spectra, see Proposition \ref{prop sphere thh to log thh cofiber}. This is $\thh(\sphsk)_0$ which admits the structure of a graded $E_1$-ring spectrum equipped with a map $\thh(\sphsk) \to \thh(\sphsk)_0$ of graded $E_1$-ring spectra given by the relevant weight $0$ Postnikov section map. Indeed, due to the universal property of Postnikov sections, this weight $0$ Postnikov section map factors the  map of graded $E_1$-rings $\thh(\sphsk) \to \thh(\sph)$ induced by the weight $0$ Postnikov section map $\sphsk \to \sph$; i.e.\ we have a factorization of this map of graded $E_1$-algebras as 
\[\thh(\sphsk) \to \thh(\sphsk)_{0} \xrightarrow{\simeq } \thh(\sph).\]
The second map above is an equivalence as its domain and codomain are equivalent to $\sph$ as spectra and it carries the unit to the unit by construction. In particular, we can replace $\sph$ on the right hand side of \eqref{eq moving suspension out} with $\thh(\sph)$. This provides the first equivalence below.
\begin{equation}\label{eq second eq for log thh cofiber sequence}
    \begin{split}
         \Sigma \thh(A) \wdg_{\thh(\sphsk)}  \sph & \simeq \Sigma \thh(A) \wdg_{\thh(\sphsk)} \thh(\sph)\\
          &\simeq \Sigma \thh(A \wdg_{\sphsk} \sph)\\
        & \simeq \Sigma \thh(A/a)
    \end{split}
\end{equation}
The second equivalence  follows by Proposition \ref{prop thh of relative smash product} and the third equivalence follows by our description of the $E_1$-algebra $A/a$ in Construction \ref{cons quotient algebra}. Equations \eqref{eq moving suspension out} and \eqref{eq second eq for log thh cofiber sequence} identify the cofiber in \eqref{eq first eq in pf of cof seq of log thh} as $\thh(A/a)$ providing the cofiber sequence claimed in the theorem.

The statement regarding the  cofiber sequence in $m$-graded spectra for $\artmx$ follows by utilizing same arguments.
\end{proof}

\begin{rema}
The above localization sequence is of fundamental importance in the theory of log THH. A proof of the above localization sequence for general $E_2$ log structures, using more general methods, will be supplied in \cite{logstructures}. 
\end{rema} 

\subsection{THH after root adjunction}

Here, we identify $\thh(\artmx)$ in terms of $\thh(A)$ and $\thh(A \mid a)$. \begin{theo}[Theorem \ref{theo intro thh after root adjunction}]\label{theo thh after root adjunction}
Assume Hypothesis \ref{hypo root adjunction} with $p\nmid m$ and that  $A$ is $p$-local. In this situation, the $m$-graded spectrum $\thh(\artmx)$ is given by 
\[\thh(\artmx)_0 \simeq \thh(A)\]
and
\[\thh(\artmx)_i \simeq \Sigma^{ik} \thh(A\mid a) \textup{ \ for \ } 0<i <m.\]
In particular, there is  an equivalence of spectra:
\[\thh(\artmx) \simeq \thh(A) \vee \big(\bigvee_{0<i<m}\Sigma^{ik}\thh(A \mid a) \big).\]
\end{theo}
\begin{proof}
The identification of $\thh(\artmx)_0$ is provided by Proposition \ref{theo weight zero of THH after root adjunction}. Therefore, it is sufficient to provide the identification of $\thh(\artmx)_i$ for $i\neq 0$.

Due to Theorem \ref{theo root adj is log thh etale }, 
\begin{equation}\label{eq refer back to logthhetale}
    \thh(\artmx \mid \sqrt[m]{a})_i \simeq \Sigma^{ik} \thh(A \mid a).
\end{equation}
Therefore, it is sufficient to show that 
\begin{equation}\label{eq thh is log thh in nonzero}
    \thh(\artmx)_i \simeq \thh(\artmx \mid \sqrt[m]{a})_i
\end{equation}
whenever $i \neq 0$. 
This follows once we show that the cofiber of the the map 
\begin{equation}\label{eq thh to log thh for root adj}
    \thh(\artmx) \to \thh(\artmx \mid \sqrt[m]{a})
\end{equation}
of $m$-graded spectra is concentrated in weight $0$. Due to Theorem \ref{theo cofiber seq for thh to log thh}, the cofiber of this map is given by \[\thh(\artmx/\sqrt[m]{a})\]
where $\artmx/\sqrt[m]{a}$ is defined to be $\artmx \wdg_{\sphsk} \sph$. Therefore, we have 
\[\artmx/\sqrt[m]{a} := \artmx \wdg_{\sphsk} \sph \simeq A \wdg_{\sphsmk} \sphsk \wdg_{\sphsk} \sph \simeq A \wdg_{\sphsmk} \sph.\]
Since $A$, $\sphsmk$ and $\sph$ are concentrated in weight $0$,
we obtain that $\artmx/\sqrt[m]{a}$ and therefore $\thh(\artmx/\sqrt[m]{a})$ are also concentrated in weight $0$. This proves that the cofiber of \eqref{eq thh to log thh for root adj} is concentrated in weight $0$ which proves \eqref{eq thh is log thh in nonzero} and this, together with \eqref{eq refer back to logthhetale} proves the theorem.
\end{proof}

\section{Algebraic $\kth$-theory of complex and real topological $\kth$-theories}\label{sec alg k theory of cmplx and real k theories}

%Adams show that $KU_p$, the $p$-completion of the complex $\kth$-theory spectrum, admits a splitting $KU_p \simeq \vee_{0\leq i<p-1} \Sigma^{2i} L_p$. The spectrum $\kth(ku)$ represents a $2$-categorical  analogue of  complex $\kth$-theory \cite{baas2011stablebundlesrig}. As a $2$-categorical analogue of the splitting of $KU_p$, we show that $\kth(ku_p)$ also admits a splitting into $p-1$ non-trivial summands. From this, we obtain the $V(1)$-homotopy of $\kth(ko_p)$ where $ko_p$ denotes the $p$-completion of the real $\kth$-theory spectrum. This is also of interest as $K(ko)$ represents a $2$-categorical analogue of real $\kth$-theory \cite{baas2011stablebundlesrig}.

Here, we start by showing that $\kth(ku_p)$ splits into $p-1$ non-trivial summands.  Afterwards, we show that $ku_p$ may be constructed from $ko_p$ via root adjunction. We use this to  obtain an explicit description of the $V(1)$-homotopy of $\kth(ko_p)$ from the first authors computation of $V(1)_*\kth(ku_p)$ \cite{ausoni2010kthryofcomplexkthry}.

\subsection{Adams' splitting result for $2$-vector bundles}

 Recall that $\pis \kup \cong \zp[u]$ and $\pis \ell_p \cong \zp[v_1]$ where $\lv u \rv = 2$ and $\lv v_1 \rv = 2p-2$. The $E_\infty$-map $\ell_p \to ku_p$ carries $v_1$ to $u^{p-1}$ in homotopy. For the rest of this section, we fix a map $\sph[\sigma_{2(p-1)}] \to \ell_p$ of $E_2$-algebras carrying $\sigma_{2(p-1)}$ to $v_1$ and perform root adjunction using this map. Recall from Theorem \ref{theo complex kthry as root adjunction} that there is an equivalence of $E_1$ $\ell_p$-algebras \[\ell_p(\sqrt[p-1]{v_1}) \simeq ku_p.\]

This equips $\kup$ with the structure of a $p-1$-graded $E_1$ $\ell_p$-algebra which further equips $\thh(\kup)$ with the structure of a $p-1$-graded $S^1$-equivariant spectrum. 

Let $p>3$ be a prime and let $V(1)$ denote the type-$2$ finite spectrum used in \cite{ausoni2005thhofku}; $V(1)$ is a homotopy ring spectrum. 

There is another grading on $V(1)_*\thh(\kup)$ that the first author calls the $\delta$-grading \cite{ausoni2005thhofku}. The group  $\Delta := \z/(p-1)$ acts on the $E_\infty$-ring  $ku_p$ through Adams operations. Let $\delta \in \Delta$ be a chosen generator and let $\alpha \in \fp^\times$ satisfy  $\pis (\sph/p \wdg \delta )(u) = \alpha u$ where
\[\pis (\sph/p \wdg \delta) \co \pis (\sph/p \wdg \kup) \to \pis (\sph/p \wdg \kup) \cong \fp[u].\]

We say $u^i$ has $\delta$-weight $i$ as $\pis (\sph/p \wdg \delta)(u^i) = \alpha^i u^i$. Similarly, one says $x \in V(1)_*\thh(\kup)$ has $\delta$-weight $i$ if the self map of $V(1)_*\thh(\kup)$ induced by $\delta$ carries $x$ to $\alpha^ix$. One defines $\delta$-weight in a similar way on other invariants of $ku_p$ \cite[Definition 8.2]{ausoni2005thhofku}.

\begin{prop}\label{prop weight and delta weight agree on k of ku}
The group $V(1)_*\thh(\kup)_i$ is given by the classes of $\delta$-weight $i$ in $V(1)_* \thh(\kup)$.
\end{prop}
\begin{proof}

Since $\hfp \wdg \kup$ is a $p-1$ graded $E_1$ $\hfp$-algebra, there is a $p-1$-grading on $\hhfps(\hfps\kup)$. By inspection on the Hochschild complex, one observes that the $\delta$-weight grading on $\hhfps(\hfps ku_p)$ agrees with the weight grading. In particular, the $\delta$-weight grading and the weight grading agree on the second page of the B\"okstedt spectral sequence computing  $\hhfp(\hfp \wdg ku_p)$. Due to \cite[Section 9]{ausoni2005thhofku}, this shows that the $\delta$-weight grading and the weight grading agree on $\hhhfps(\hfp \wdg \kup)$.  Furthermore, there is a basis of  $\hhhfps(\hfp \wdg \kup)$ as an $\fp$-module where $\delta$-weight is defined for each basis element. Therefore, the $\hfp$-module $\hhhfp(\hfp \wdg \kup)$ splits as a coproduct of  suspensions of $\hfp$ in a way that the map  $\hhhfp(\hfp \wdg \delta)$ is given by the respective multiplication map corresponding to the $\delta$-weight on each cofactor. Using this, one observes that the $\delta$-weight and the weight grading agree on  $H_*(V(1) \wdg \thh(ku_p);\fp)$. 

The Hurewicz map 
\[V(1)_* \thh(ku_p) \to H_*(V(1) \wdg \thh(ku_p);\fp)\]
is injective and this map preserves both gradings. From this, we deduce that the weight grading and the $\delta$-weight grading agree on $V(1)_* \thh(ku_p)$.
\end{proof}

In general, THH of $m$-graded ring spectra may not result in an $m$-graded cyclotomic spectrum as the Frobenius map do not preserve the grading; it multiplies the grading by $p$. On the other hand, for $ku_p$, $\thh(\kup)$ is $p-1$-graded and $p= 1$ in $\z/(p-1)$. In particular, the Frobenius map preserves the grading and one obtains that $\thh(ku)$ is a $p-1$-graded cyclotomic spectrum. 
\begin{prop}\label{prop thh ku is graded cyclotomic}
The $S^1$-equivariant structure on  $\thh(\kup)_i$ lifts to a cyclotomic structure for which there is an equivalence 
\[\thh(ku) \simeq \prod_{i \in \z/(p-1)}\thh(ku)_i\]
of cyclotomic spectra.
\end{prop}
\begin{proof}
The monoid $\Z/ (p-1)$ satisfies the conditions in \cite[Appendix A]{antieau2020beilinson} needed  endow $\thh(ku)$ with an $L_p$ twisted cyclotomic  structure. However, since $ p \cong 1 \mod p-1$, this ends up being the identity functor on $\Z /(p-1)$-graded spectra. Thus one obtains a sequence of $S^1$-equivariant maps 
$$
\thh(ku)_i \to \thh(ku)^{t C_p}_{i}
$$
for each $i \in \Z / (p -1)$, which is precisely the relevant additional piece of structure needed to view this as a cyclotomic object.
\end{proof}
\begin{cons}\label{cons splitting of kthry ku}
Here, we construct a splitting of $\kth(\kup)$ using Proposition \ref{prop thh ku is graded cyclotomic}. Since the product mentioned in Proposition \ref{prop thh ku is graded cyclotomic} is a finite product, it is at the same time a coproduct. In particular, it commutes with all limits and colimits. Therefore, the fiber sequence defining $\tc(\kup)$ splits into a product of fiber sequences 
\[\tc(\kup)_i \to \thh(\kup)_i^{hS^1} \xrightarrow{(\varphi_p)_i - can_i } (\thh(\kup)_i^{tC_p})^{hS^1}.\]
Hence, there is a splitting of $\tc(ku_p)$:
\[\tc(\kup) \simeq \prod_{i \in \z/(p-1)}\tc(ku_p)_i\]
where $\tc(ku_p)_i := \tc(\thh(ku_p)_i)$. Arguing as in the proof of Theorem \ref{theo root adj k theory inclusion}, one obtains a map $ku_p \to H\zp$ of $p-1$-graded $E_1$-rings where $\hzp$ is concentrated in weight $0$. Therefore, the induced map $\thh(\kup) \to \thh(\zp)$  of $p-1$-graded  spectra is trivial in non-zero weight. By inspection on the product splitting of the fiber sequence defining $\tc(ku_p)$, we  consider $\tc(ku_p) \to \tc(\zp)$  as a map of $p-1$ graded spectra where $\tc(\zp)$ is concentrated in weight $0$.  Again, as in the proof of Theorem \ref{theo root adj k theory inclusion}, this splits the pull-back square (from Dundas-Goodwillie-McCarthy theorem) relating $\tc(\kup)$ to $\kth(\kup)$ resulting in a splitting of $\kth(\kup)$ that we denote by 
\[\kth(\kup) \simeq \bigvee_{i \in \z/(p-1)}\kth(\kup)_i.\]
Here, $\kth(\kup)_0 \simeq \kth(\ell_p)$ due to Theorem \ref{theo root adj k theory inclusion}. 
\end{cons}

To understand the resulting splitting of $\kth(ku_p)$, we  identify the $V(1)$-homotopy of each weight piece. The computation of $V(1)_*\kth(\kup)$ is due to the first author \cite[Theorem 8.1]{ausoni2010kthryofcomplexkthry} and these groups are given below. 
\begin{equation} \label{eq vone htpy of kthry of ku}
    \begin{split}
        V(1)_* \kth(\kup) \cong & \fp[b] \otimes \Lambda(\lambda_1,a_1) \oplus \fp[b] \otimes \fp \{\partial \lambda_1,\partial b, \partial a_1, \partial \lambda_1 a_1\} \\
        &\oplus \fp[b] \otimes \Lambda(a_1) \otimes \fp \{t^d \lambda_1 \mid 0<d<p \} \\
        & \oplus \fp[b] \otimes \Lambda(\lambda_1) \otimes \fp \{\sigma_n, \lambda_2 t^{p^2-p} \mid 1 \leq n \leq p-2 \}\\
        &\oplus \fp \{s \}
    \end{split}
\end{equation}
Here, $\lv b \rv = 2p+2$, $\lv \partial \rv = -1$, $\lv \lambda_1 \rv = 2p-1$, $\lv a_1 \rv = 2p+3$, $\lv \sigma_n \rv = 2n+1$, $\lv t \rv = -2$, $\lv \lambda_2 \rv = 2p^2-1$ and $\lv s \rv = 2p-3$. We assign weights to these classes in a way that turns $V(1)_*\kth(\kup)$ into a $p-1$-graded abelian group. The weights of $\sigma_n$, $b $, $a_1$, $ \partial$, $\lambda_1$, $ t$, $\lambda_2$ and $s$ are given by $n$, $1$, $1$, $0$, $0$, $0$, $0$ and $0$ respectively. Classes denoted by tensor products or products above have the canonical degrees and weights. Furthermore, the isomorphism above is that of $\fp[b]$-modules and $b^{p-1} = -v_2$.

\begin{theo}\label{theo kthry splitting for ku}
For the equivalence of spectra 
\[\kth(ku_p) \simeq \bigvee_{i \in \z/(p-1)} \kth(ku_p)_i\]
provided by Construction \ref{cons splitting of kthry ku}, there is an equivalence:
\[\kth(ku_p)_0\simeq \kth(\ell_p)\]
and there are isomorphisms 
\[\vos( \kth(ku_p)_i) \cong \big(\vos\kth(\kup)\big)_i\]
for each $i \in \z/(p-1)$ where the right hand side denotes the weight $i$ piece of the  $p-1$-grading on $\vos \kth(\kup)$ described above. 
\end{theo}
\begin{proof}
The identification of $\kth(\kup)_0$ is given in Construction \ref{cons splitting of kthry ku}. This provides the identification of $\vos \kth(\kup)_0$ as $\big(\vos\kth(\kup)\big)_0$ since this is precisely the image of the map 
\[\vos \kth(\ell_p) \to \vos \kth(\kup),\]
see \cite[Theorem 10.2]{ausoni2005thhofku}. The identification of $\vos (\kth(ku_p)_i)$ for $i \neq 0 $ follows by noting from Proposition \ref{prop weight and delta weight agree on k of ku} that it is sufficient to keep track of the contribution of $\delta$-weight $i$ classes in $\vos \thh(\kup)$ to $\vos \tc(\kup)$. This  follows by inspection on \cite[Section 7]{ausoni2010kthryofcomplexkthry} and \cite[Section 5]{ausoni2010kthryofcomplexkthry}.
\end{proof}

\subsection{Algebraic $\kth$-theory of real $\kth$-theory}
\begin{comment}
Use the 
\href{http://math.uchicago.edu/~amathew/tmfhomotopy.pdf}{Page 2 Click} which says that the $C_2$-action is compatible with the $\delta$-action on $\kup$ which should give you the desired result. 
\end{comment}
Let $p>3$. Using Theorem \ref{theo root adj k theory inclusion}, the splitting of $\kth(\kup)$ discussed above and our root adjunction formalism, we obtain a straightforward computation of $V(1)_*\kth(ko_p)$ from our knowledge of $\vos \kth(\kup)$ from \cite{ausoni2010kthryofcomplexkthry}. Here, $ko_p$ denotes the connective cover of the $p$-completed real topological $\kth$-theory spectrum $KO_p$. We have $\pis KO_p \cong \zp[\alpha^\pm]$ with $\lv \alpha \rv = 4$.

There is a subgroup of $C_2$ of $\Delta\cong \z/(p-1)$ such that $KO_p \simeq KU_p^{hC_2}$.  Through this, the induced map $KO_p \to KU_p$ carries $\alpha$ to $u^2$ up to a unit that we are going to omit. Since $L\simeq (KU_p)^{h\Delta}$, we obtain a sequence of $E_\infty$-maps
\[L_p \to KO_p \to KU_p\]
where the first map carries $v_1$ to $\alpha^{\frac{p-1}{2}}$ in homotopy. 
\begin{theo}\label{prop real kthry as root adjunction}
For $p>3$, there is an equivalence 
\[ko_p \simeq \ell_p (\sqrt[\frac{p-1}{2}]{v_1})\]
of $E_1$ $\ell_p$-algebras.
\end{theo}
\begin{proof}
This follows as in the proof of  Theorem \ref{theo complex kthry as root adjunction}  by noting that $p \nmid \pmot$.
\end{proof}
Furthermore, $ku_p$ may also be obtained from $ko_p$ via root adjunction; for this root adjunction, we use the $\sphp[\sigma_4]$-algebra structure on $ko_p$ provided by Theorem \ref{prop real kthry as root adjunction}. To identify the resulting $2$-graded $E_1$-ring structure on $ku_p$, we use the symmetric monoidal functor \[D' \co \on{Fun}(\z/(p-1), \Sp) \to \on{Fun}(\z/2, \Sp) \]
given by left Kan extension through the canonical map $\z/(p-1) \to \z/2$.

\begin{prop}\label{prop adj root to kop}
For $p>3$, there is an equivalence 
\[ko_p(\sqrt[2]{\alpha}) \simeq D'(ku_p)\]
of $2$-graded $E_1$-algebras where $D'$ is  defined above and the $p-1$-grading on  $ku_p$ is given by Theorem \ref{theo complex kthry as root adjunction}.
\end{prop}
\begin{proof}
Due to Theorem \ref{prop real kthry as root adjunction}, $ko_p$ is  an  $\sph[\sigma_4]$-algebra  given by
\[\ell_p \wdg_{\sph[\sigma_{2(p-1)}]}\sph[\sigma_4].\]
 To adjoin a root to $ko_p$ using this structure, we use the sequence of maps 
\[\sph[\sigma_{2(p-1)}] \to \sph[\sigma_4]\to D'(\sph[\sigma_2])\]
of $2$-graded $E_2$-ring spectra where $\sph[\sigma_{2(p-1)}]$ and $\sph[\sigma_4]$ are concentrated in weight $0$ and $\sph[\sigma_2]$ above is given its canonical $p-1$-grading so that $\sigma_2$ in $D'(\sph[\sigma_2])$ lies in weight $1$. 

We obtain the following equivalences of $2$-graded $E_1$-rings.
\begin{equation}\label{eq real ktheory root adjunction gives complex ktheory}
ko_p(\sqrt[2]{\alpha}) \simeq \ell_p \wdg_{\sph[\sigma_{2(p-1)}]} \sph[\sigma_{4}] \wdg_{\sph[\sigma_{4}]} D'(\sph[\sigma_2]) \simeq \ell_p \wdg_{\sph[\sigma_{2(p-1)}]} D'(\sph[\sigma_{2}])
\end{equation}

The functor $D'$ is a left adjoint and it is symmetric monoidal. Therefore, it commutes with the two sided bar construction defining relative smash products. This provides the second equivalence in the following equivalences of $2$-graded $E_1$-algebras. 
\begin{equation}\label{eq weight reduction for kup}
    D'(\kup)  \simeq D'(\ell_p \wdg_{\sph[\sigma_{2(p-1)}]} \sph[\sigma_{2}]) \simeq D'(\ell_p) \wdg_{D'(\sph[\sigma_{2(p-1)}])} D'(\sph[\sigma_{2}])
\end{equation} 
The first equivalence above follows by Theorem \ref{theo complex kthry as root adjunction} and the relative smash product in the middle is taken in $p-1$-graded spectra. Since $\ell_p$ and $\sph[\sigma_{2(p-1)}]$ are concentrated in weight $0$, the right hand side above is equivalent to the right hand side of \eqref{eq real ktheory root adjunction gives complex ktheory}; this follows by Lemma \ref{lem underlying of conc in weight zero}. In other words, \eqref{eq real ktheory root adjunction gives complex ktheory} and \eqref{eq weight reduction for kup} agree. 
\end{proof}

Recall that the spectra $\kth(ku_p)_i$ are given in Construction \ref{cons splitting of kthry ku} and the groups $V(1)_*\kth(ku_p)_i$ are identified in Theorem \ref{theo kthry splitting for ku}.
\begin{theo}\label{theo kthry of ko}
For $p>3$, there is an equivalence of spectra:
\[\kth(ko_p) \simeq \bigvee_{0\leq i<(p-1)/2}\kth(ku_p)_{2i}.\]
Therefore, we have 
    \[V(1)_*\kth(ko_p)\cong \bigoplus_{0 \leq i<(p-1)/2} V(1)_* \kth(ku_p)_{2i}.\]
   and $V(1)_*\kth(ko_p)$, as an abelian group, is given by:
    
\begin{equation*}
    \begin{split}
        V(1)_* \kth(ko_p) \cong & \fp[b^2] \otimes \Lambda(\lambda_1,ba_1) \oplus \fp[b^2] \otimes \fp \{\partial \lambda_1,b\partial b, b\partial a_1, b\partial \lambda_1 a_1\} \\
        &\oplus \fp[b^2] \otimes \Lambda(ba_1) \otimes \fp \{t^d \lambda_1 \mid 0<d<p \} \\
        & \oplus \fp[b^2] \otimes \Lambda(\lambda_1) \otimes \fp \{b^{\epsilon(n)}\sigma_n, \lambda_2 t^{p^2-p} \mid 1 \leq n \leq p-2 \}\\
        &\oplus \fp \{s \},
    \end{split}
\end{equation*}
where $\epsilon(n) = 1$ if $n$
 is odd and $\epsilon(n) = 0$ if $n$ is even. Here, the class denoted by  $(b^2)^{(p-1)/2}$ is $ -v_2$.
 
 As a consequence, we have an  isomorphism of abelian groups:
 \[T(2)_* \kth(ko) \cong T(2)_* \kth(\ell_p)[b^2]/((b^2)^{(p-1)/2}+v_2).\]
\end{theo}
\begin{proof}
We start by identifying $\thh(ko_p)$ as a cyclotomic spectrum. We have the following chain of equivalences \begin{equation*}
    \begin{split}
        \thh(ko_p) &\simeq \thh(ko_p(\sqrt[2]{\alpha}))_0 \\
        &\simeq \thh(D'(ku_p))_0\\
        &\simeq \big(D'(\thh(ku_p))\big)_0\\
        &\simeq \prod_{0 \leq i < (p-1)/2} \thh(ku_p)_{2i}
    \end{split}
\end{equation*}
The first equivalence above follows by Theorem \ref{theo weight zero of THH after root adjunction}, the second equivalence follows by Proposition \ref{prop adj root to kop} and the third equivalence is a consequence of \cite[Corollary A.15]{antieau2020beilinson}. The last equivalence above follows by the description of $D'$ as a left Kan extension, see Section \ref{subsec manipulations on grade objects}. Indeed, this shows that the following composite map of cyclotomic spectra is an equivalence. 
\[\thh(ko_p) \to \thh(ku_p) \simeq \prod_{0\leq i<p-1}\thh(ku_p)_i \to \prod_{0 \leq i < (p-1)/2} \thh(ku_p)_{2i}\]
Here, the equivalence in the middle follows by Proposition \ref{prop thh ku is graded cyclotomic}. The last map above is the canonical projection.

The composite equivalence of cyclotomic spectra above shows that 
\[\tc(\thh(ko_p)) \simeq \tc(\prod_{0\leq i<(p-1)/2} \thh(ku_p)_{2i}) \simeq \prod_{0\leq i<(p-1)/2} \tc(ku_p)_{2i}.\]
Considering the  Dundas-Goodwillie-McCarthy theorem with respect to the composite $ko_p\to ku_p \to H\zp$,  we obtain that the splitting of the pullback square relating $\kth(ku_p)$ with $\tc(ku_p)$ (mentioned in Construction \ref{cons splitting of kthry ku}) provides a splitting for $\kth(ko_p)$ given by
\begin{equation}\label{eq intermediate splitting of kthry of kop}
\kth(ko_p) \simeq \prod_{0\leq i<(p-1)/2}\kth(ku_p)_{2i}.
\end{equation}
The first and the second statements in the theorem follow from this splitting. The third statement follows by this, and by inspection on \eqref{eq vone htpy of kthry of ku}.

For the last statement, note that $T(2)_*\kth(ko) \cong T(2)_*\kth(ko_p)$ due to the purity of algebraic $\kth$-theory and \cite[Lemma 2.2 (\romannumeral 6)]{land2020purity}. It follows by  Theorem \ref{theo kthry splitting for ku}, that the map 
\[T(2)_*\kth(ku_p) \xrightarrow{\cdot b^i} T(2)_*\kth(ku_p)\]
carries $T(2)_*\kth(ku_p)_0$ to $T(2)_*\kth(ku_p)_i$  for $i<p-1$ where the map above multiplies by $b^i$. Using this fact, together with   \cite[Proposition 1.2 (b)]{ausoni2010kthryofcomplexkthry}, provides isomorphisms 
\[T(2)_*\kth(\ell_p) \cong T(2)_*\kth(ku_p)_0 \xrightarrow{\cong} T(2)_*\kth(ku_p)_i\]
given by $\cdot b^i$ for $i<p-1$. This, together with  \eqref{eq intermediate splitting of kthry of kop} provides the desired identification of  $T(2)_*\kth(ko)\cong T(2)_*\kth(ko_p)$ as $T(2)_* \kth(\ell_p)[b^2]/((b^2)^{(p-1)/2}+v_2)$.
\end{proof}

\section{Root adjunction and Lubin-Tate spectra}\label{sect kthry and redshift for LT spectra}

%Here, we show that various Lubin-Tate spectra may be obtained from $E_3$-algebra forms of $\bpn$  via root adjunction. Through this, we obtain a splitting of the algebraic $K$-theory of Lubin-Tate spectra. Furthermore, we use the redshift result of Hahn and Wilson for $\bpn$ to obtain a new proof of Yuan's result stating that the  Lubin-Tate spectra satisfy the redshift conjecture \cite{yuan2021examples}. As mentioned earlier, we say an $E_1$-ring $A$ of height $n$ satisfies the redshift conjecture if $L_{T(n+1)}\kth(A) \not \simeq *$. 
\begin{comment}
There is a retract $BP \to MU_{(p)} \to BP$ of $E_4$-rings \cite{basterra2013BP}. 
Following \cite[Definition 2.0.1]{hahn2020redshift}, we say an $E_k$ $MU_{(p)}$-algebra $B$ is an $E_k$ $MU$-algebra  form of $\bpn$ if the composite 
\begin{equation}\label{eq mualgebra forms of bpn}
\zpl[v_1,v_2,\dots,v_n] \to \pis BP \to \pis MU_{(p)} \to \pis B
\end{equation}
is an isomorphism. As mentioned in \cite[Remark 2.0.2]{hahn2020redshift}, we are free to choose the generators $v_1,..,v_n$ in $\pis BP$ as we wish. 
\end{comment}

Recall that in  \cite{hahn2020redshift}, Hahn and Wilson prove that there are $E_3$ $MU$-algebra forms of $\bpn$. Furthermore, their constructions provide an $E_3$ $MU[\sigma_{2(p^n-1)}]$-algebra form of $\bpn$ where $\sigma_{2(p^n-1)}$ acts through $v_n$ \cite[Remark 2.1.2]{hahn2020redshift}. 

To relate particular forms of $\bpn$ to  Lubin-Tate spectra, we use the spherical Witt vectors constructed by Lurie \cite[Example 5.2.7]{lurie2018ellipticII}.  For a given discrete perfect $\fp$-algebra $B_0$, this provides an $E_\infty$-ring $\sph_{W(B_0)}$ that is flat over $\sph$ in the sense of \cite[Defnition 7.2.2.10]{lurie2016higher}.  Therefore, it follows by \cite[Proposition 7.2.2.13]{lurie2016higher} and \cite[Proposition 2.7]{mao2020perfectoid} that 
\begin{equation}\label{eq sph witt flatness}
    \pi_n (\sph_{W(B_0)} \wdg F) \cong W(B_0) \otimes \pi_n F.
\end{equation}
for every spectrum $F$. We would like to thank Jeremy Hahn for showing us the proof of the following proposition.
\begin{prop}\label{prop bpn with witt coefficients satisfy the redshift conjecture}
    Fix an $E_3$ $MU$-algebra form of $\bpn$. Then $\sph_{W(B_0)} \wdg \bpn$ satisfies the redshift conjecture for every discrete perfect $\fp$-algebra $B_0$.  
\end{prop}
\begin{proof}
    There is an equivalence of $E_2$-algebras in $S^1$-equivariant spectra
    \[\sph_{W(B_0)}\wdg \thh^{MU}(\bpn)\simeq \thh^{\sph_{W(B_0)} \wdg MU}( \sph_{W(B_0)} \wdg \bpn).\]
    Therefore, it follows by  \eqref{eq sph witt flatness} that we have
    \begin{equation}\label{eq relative thh groups for spherical witt bpn}
    \thh^{\sph_{W(B_0)} \wdg MU}_*( \sph_{W(B_0)} \wdg \bpn) \cong W(B_0) \otimes \thh^{MU}_*(\bpn)
    \end{equation}
    and this is a polynomial algebra over $W(B_0)[v_1,...,v_n]$ due to \cite[Theorem 2.5.4]{hahn2020redshift} where one of the generators is denoted by $\sigma^2v_{n+1}$ . The rest of the argument follows as in the proofs of \cite[Theorems 2.5.4 and 5.0.1, Corollary 5.0.2]{hahn2020redshift} by considering 
    \begin{equation}\label{eq htpy of rel thh fixedpoint of witt bpn}
    \pis \big(\thh^{\sph_{W(B_0)} \wdg MU}( \sph_{W(B_0)} \wdg \bpn)^{hS^1}\big)
    \end{equation}
     instead of $\pis\big(\thh^{MU}(\bpn)^{hS^1}\big)$.

     Namely, we obtain from \cite[Theorem 2.5.4]{hahn2020redshift} that \eqref{eq relative thh groups for spherical witt bpn} is concentrated in even degrees. Therefore, the corresponding homotopy fixed point spectral sequence degenerates on the second page providing the $W(B_0)[v_1,...,v_n]$-algebra  \eqref{eq htpy of rel thh fixedpoint of witt bpn} as 
     \[W(B_0) \otimes \thh^{MU}_*(\bpn) \llbracket t \rrbracket\]
     since \eqref{eq relative thh groups for spherical witt bpn} is polynomial.
     Using the map from  $\pis \big(\thh^{MU}(\bpn)^{hS^1}\big )$ to \eqref{eq htpy of rel thh fixedpoint of witt bpn}, we deduce from \cite[Theorem 5.0.1]{hahn2020redshift} that $v_{n+1}$ in \eqref{eq htpy of rel thh fixedpoint of witt bpn} is represented by the class $t\sigma^2v_{n+1}$. 
     
     Considering the action of $v_0,...v_{n+1}$ on \eqref{eq htpy of rel thh fixedpoint of witt bpn} described above, one observes that 
     \[L_{T(n+1)}\thh^{\sph_{W(B_0)} \wdg MU}( \sph_{W(B_0)} \wdg \bpn)^{hS^1} \not \simeq *.\]
Using this and the $E_2$-map 
     \[L_{T(n+1)}\kth(\sph_{W(B_0)} \wdg \bpn) \to L_{T(n+1)}\thh^{\sph_{W(B_0)} \wdg MU}(\sph_{W(B_0)} \wdg \bpn)^{hS^1},\]
     we deduce that $L_{T(n+1)}\kth(\sph_{W(B_0)} \wdg \bpn) \not \simeq *$
     as desired.
\end{proof}

\begin{cons}\label{cons mrv e theory out of bpn}
    Fix an $E_3$ $MU[\sigma_{2(p^n-1)}]$-algebra form of $\bpn$ where $\sigma_{2(p^n-1)}$ acts through $v_n$ and let $k$ be a perfect field of characteristic $p$. Recall that in this situation, degree $p^n-1$-root adjunction to $v_n$ provides a $p^n-1$-graded $E_3$ $MU[\sigma_2]$-algebra, see Remark \ref{rema bpnadj root is ethree}. We consider the $E_3$ $MU$-algebra: 
    \[E := \big(L_{K(n)} (\sph_{W(k)} \wdg \bpn)\big)(\sqrt[p^n-1]{v_n}).\]
    It follows by \cite[Theorem 1.5.4]{hovey1997vnelementsbordism} and \cite[Theorem 1.9]{hovey1995bousfieldhopkins} that 
    \[\pis E \cong W(k)[\lv u_1,...,u_{n-1}\rv][u^{\pm}]\]
    where $\lv u_i \rv = 0$ and $\lv u \rv = -2$.
    Furthermore, the resulting  $E_3$-map $MU\to\bpn \to  E$ provides a formal group law over $\pis E$.
    
\end{cons}

For a given perfect $\fp$-algebra  $B_0$  and a height $n$ formal group law $\Gamma$ over $B_0$, we let $E_{(B_0,\Gamma)}$ denote the corresponding Lubin-Tate spectrum. By Lurie's generalization \cite[Section 5]{lurie2018ellipticII} of the  Goerss-Hopkins-Miller theorem \cite{rezk1998notes,goersshopkins2004modulispcs}, $E_{(B_0,\Gamma)}$ is an $E_\infty$-ring.

\begin{prop}\label{prop bpn adj root is lubin tate}
In the setting of Construction \ref{cons mrv e theory out of bpn},   $E$ is equivalent to   $E_{(k,\Gamma)}$ as an $E_1$-ring for some height $n$ formal group law $\Gamma$ over $k$.
\end{prop}
\begin{proof}
%Burklunds notes in https://math.mit.edu/research/undergraduate/spur/documents/2015burklund.pdf are helpful. Especially, Theorem 2.11, this allows one to show that our formal group law over k is indeed height n.
    By construction, the map $\pis MU_{(p)} \to \pis E$ carries $v_i$ to $u_iu^{-(p^i-1)}$ for $0<i<n$ and $v_n$ to $u^{-(p^n-1)}$. Therefore, the corresponding formal group law on $\pi_0E$ is the universal deformation of the resulting height $n$ formal group law $\Gamma$ over $k$. This follows by \cite[Theorem 5 in Lecture 21]{lurie2010chromaticlectures}. Alternatively, one can directly check the conditions given in \cite[Proposition 1.1]{lubintate1966formalmodulionparameter}. It follows from Hopkins-Miller theorem that there is an equivalence of $E_1$-rings $E\simeq E_{(k,\Gamma)}$ \cite[Theorem 7.1]{rezk1998notes}.

    % For this, Discussion before Theorem 2.2 in https://math.jhu.edu/webarchive/grad/Torii.thesis.pdf of Takeshi Torii notes on 1 dimensional formal group laws... also, lemma 1.31 in https://ncatlab.org/nlab/files/GreenbergEllipticCohomology.pdf Matthew Greenberg notes constructing elliptic cohomology
    
    %\cite[Proposition 11 in Lecture 18]{lurie2010chromaticlectures} that as a homotopy ring spectrum (i.e.\ a monoid in the stable homotopy category), $E$ is given by the Lubin-Tate spectrum corresponding to $(k,\Gamma)$. Due to the Hopkins-Miller theorem \cite{rezk1998notes}, the homotopy ring structure on $E_{(k,\Gamma)}$ lifts to a unique $E_1$-ring structure. This shows that $E\simeq E_{(k,\Gamma)}$ as $E_1$-rings.
\end{proof}

Burklund, Hahn, Levy and Schlank are going to  use the following example in  their  construction of a counterexample to the telescope conjecture.
\begin{exam}\label{exa for the telescope conjecture}
   Let $k$ be a perfect algebraic extension of $\fp$. 
    We know from \cite[Corollary 4.31]{ramzi2023separability} that the $E_1$-algebra structure on $E_{(k,\Gamma)}$ lifts to a unique $E_d$-algebra structure for every $1 \leq d\leq \infty$. Since $E$ in  Construction \ref{cons mrv e theory out of bpn} is an $E_3$-ring, it follows from Proposition \ref{prop bpn adj root is lubin tate} that there is an $E_3$-equivalence
   \[E\simeq E_{(k,\Gamma)}\]
   for $\Gamma$ as in Proposition \ref{prop bpn adj root is lubin tate}. Through this equivalence, we may equip $E_{(k,\Gamma)}$ with the structure of an $E_3$ $MU$-algebra. In particular, we obtain a map of $E_3$ $MU$-algebras
   \[\bpn \to E_{(k,\Gamma)}.\]
   Furthermore, $E_{(k,\Gamma)}$ is a $\z/(p^n-1)$-graded $E_3$ $MU[\sigma_2]$-algebra where the weight $1$ class $\sigma_2$ acts through $u^{-1}$. 
\end{exam}

\begin{theo}\label{theo kthry splitting from bpn to lubin tate}
    In the setting of Construction \ref{cons mrv e theory out of bpn} and Proposition \ref{prop bpn adj root is lubin tate}, the canonical map 
    \[L_{T(n+1)}\kth(\sph_{W(k)} \wdg \bpn) \to L_{T(n+1)}\kth(E_{(k,\Gamma)})\]
    is the inclusion of a non-trivial wedge summand. 
\end{theo}

\begin{proof}
It follows by Corollary \ref{corr nonconnective alg kthry inclusion for root adj} and Proposition \ref{prop bpn adj root is lubin tate} that \[L_{T(n+1)}\kth\big (L_{K(n)}(\sph_{W(k)} \wdg \bpn)\big) \to L_{T(n+1)}\kth(E_{(k,\Gamma)})\]
is the inclusion of a wedge summand. Since 
\[\sph_{W(k)} \wdg \bpn \to L_{K(n)} (\sph_{W(k)} \wdg \bpn)\]
is a $T(n) \vee T(n+1)$-equivalence, the result follows by \cite[Purity Theorem]{land2020purity}.
\end{proof}

We finally prove the following theorem of Yuan.

\begin{theo}[\cite{yuan2021examples}]\label{theo lubin tate redshift}
For every perfect $\fp$-algebra $B_0$ and height $n$ formal group law $\Gamma$ over $B_0$, the Lubin-Tate spectrum $E_{(B_0,\Gamma)}$ satisfies the redshift conjecture. 
\end{theo}
\begin{proof}

 There is an $E_\infty$-map $E_{(B_0,\Gamma)} \to E_{(k,\Gamma')}$ for some algebraically closed field $k$ of characteristic $p$ and $\Gamma'$ is the corresponding height $n$ formal group law on $k$. Since there is an induced $E_\infty$-map $\kth(E_{(B_0,\Gamma)}) \to \kth(E_{(k,\Gamma')})$, it suffices to prove the redshift conjecture for $E_{(k,\Gamma')}$.  

 Since $k$ is algebraically closed, there is a unique formal group law of height $n$ over $k$ \cite[Theorem \RomanNumeralCaps{4}]{lazard1955surlegroupesfrmlprmtrs}. Using Proposition \ref{prop bpn adj root is lubin tate}, we deduce that $E_{(k,\Gamma')} \simeq E$ for $E$ as in Construction \ref{cons mrv e theory out of bpn}. Combining Proposition \ref{prop bpn with witt coefficients satisfy the redshift conjecture} and Theorem  \ref{theo kthry splitting from bpn to lubin tate} we obtain that $E_{(k,\Gamma')}$ satisfies the redshift conjecture as desired. 
\end{proof}

We remark that it should be possible to  prove the redshift conjecture for all $E_1$ $MU$-algebra forms of  $\bpn$ by constructing maps $\bpn \to E_{(k,\Gamma)}$ through root adjunction.

\section{Algebraic $\kth$-theory and THH of Morava $E$-theories}\label{sec algebraic kthry of morava e theories}
In this section, we work with a particular form of Lubin-Tate spectra. This is the Morava $E$-theory spectrum  $E_n$ and $E_n$ is central in the Ausoni-Rognes program for the computation of $\kth(\sph)$. When we say Morava $E$-theory $E_n$, we mean the Lubin-Tate spectrum corresponding to the height $n$ Honda formal group. This formal group is characterized by its $p$-series
$$
[p]_n(x) = x^{p^n},
$$
and admits a canonical form over $\fpn$, in the sense that all of its endomorphisms are defined over this field. In this section, we prove a  splitting  result 
for the algebraic $\kth$-theory of the Morava $E$-theory $E_n$ and the corresponding two periodic Morava $K$-theory. Furthermore,  we   show that the THH of $E_n$ may be obtain from the THH of the $K(n)$-localized Johnson-Wilson spectrum   through base change.  
\subsection{An identification of Morava \texorpdfstring{$E$}{E}-theory}
Here, we provide an alternate  description of $E_n$ in terms of its fixed points and spherical Witt vectors.  We have 
\[\pis E_n\cong W(\fpn)[\lv u_1, \dots,u_{n-1}\rv][u^\pm]\]
where $\lv u_i \rv = 0$ and $\lv u \rv = -2$. 

\begin{prop}\label{prop spherical witt to morava e-theory}
 The map 
 \[\pis \sphwpn \cong W(\fpn)\otimes_{\zp}\pis \sphp \to \pis E_n\]
obtained via the map $\pis \sphp \to \pis E_n$ and the canonical $\wfpn$-module structure on $\pis E_n$ lifts to a  map of $E_\infty$ $\sphp$-algebras
\[\sphwpn \to E_n.\]

\end{prop}
\begin{proof}
This is a consequence of  Lurie's theory of thickenings of relatively perfect morphisms \cite[Section 5.2]{lurie2018ellipticII}.  
\begin{comment}
the canonical diagram 
\begin{equation*}
    \begin{tikzcd}
    \sphp \ar[r] \ar[d]& \sphwpn \ar[d] \\
    \hfp \ar[r] & H\fpn
    \end{tikzcd}
\end{equation*}
\end{comment}
Indeed, $\sphp \to \sphwpn$ is an $\sphp$-thickening of  $H\fp \to H\fpn$ in the sense of \cite[Definition 5.2.1]{lurie2018ellipticII}, see \cite[Example 5.2.7]{lurie2018ellipticII}. 

In particular, 
this implies that the space of $E_\infty$ $\sphp$-algebra maps from $\sphwpn$ to the connective cover $cE_n$ of $E_n$ is given by the set of $\fp$-algebra maps 
\begin{equation}\label{eq hom for thickenings}
    \hom_{\fp\textup{-} \mathcal{A}\textup{lg}}(\fpn, \fpn[\lvert u_1, \dots, u_{n-1} \rv])
\end{equation}
where this correspondence is given by the functor $\pi_0(-)/p$.

Let $f \co \sphwpn \to cE_n$ be the map of $E_\infty$ $\sphp$-algebras corresponding to the  canonical $\fp$-algebra map in \eqref{eq hom for thickenings}; in particular,   $\pi_0 (f)/p$ is the  canonical map in \eqref{eq hom for thickenings}. We first show that $\pi_0 f$ is given by  the canonical inclusion 
\[W(\fpn) \to W(\fpn)[\lvert u_1, \dots, u_{n-1} \rv].\]

Since  $W(\fp)$ is the  ring of integers of the unique unramified extension $\mathbb{Q}_p[\mu_{p^n-1}]$ of $\mathbb{Q}_p$ of degree $d$, $W(\fpn)$ is generated as a $\zp$-algebra by a primitive $p^n-1$ root of the unit. Since the roots of the unit of $\pi_0 cE_n$ are all in the image of the canonical inclusion $W(\fpn) \to W(\fpn)[\lvert u_1, \dots, u_{n-1} \rv]$, one observes that the map $\pi_0 f$ has to factor through the canonical inclusion $W(\fpn) \to W(\fpn)[\lvert u_1, \dots, u_{n-1} \rv] $. Furthermore, there is a unique ring map $W(\fpn) \to W(\fpn)$ that lifts the identity map on $\fpn$. This shows that $\pi_0 f$ is given by the canonical inclusion $W(\fpn) \to  W(\fpn)[\lvert u_1, \dots, u_{n-1} \rv]$. Since $\pis f$ is a map of $\pis \sphp$-modules, it follows that the composition of $f$ with the map $c E_n \to E_n$ provides the map claimed in the proposition.

%Given a map $\sphwq \to E_n$ that induces the desired map in homotopy, it is clear that it provides the canonical $\fp$-algebra map from $\fp$ to $\fq[\lvert u_1, \dots, u_n \rv][u^\pm]$. 
\end{proof}
Let $Gal$ denote the Galois group $\text{Gal}(\fpn,\fp)$. Due to  Goerss-Hopkins-Miller theorem, there is an action of $Gal$ on the $E_\infty$-algebra $E_n$ for which 
\[\pis E_n^{hGal} \cong \zp[\lv u_1 \dots , u_{n-1}\rv][u^\pm]\]
where the degrees of the generators are as in $\pis E_n$.

\begin{prop}\label{prop smash splitting for morava e theory}
 There is an equivalence of $E_\infty$-$\sphp$-algebras:
\[\sphwpn \wdg_{\sphp} E_n^{hGal} \simeq E_n.\]
\end{prop}
\begin{proof}
This equivalence is given by the following composite map of $E_\infty$ $\sphp$-algebras 
\begin{equation}\label{eq map from smash splitting of en to en}
\sphwpn \wdg_{\sphp} E_n^{hGal} \to E_n \wdgp E_n \to E_n
\end{equation}
where the first map is induced by the map provided by Proposition \ref{prop spherical witt to morava e-theory} and the second map is given by the multiplication map of $E_n$. Due to the flatness of $\sphwpn$, this map induces the canonical map
\begin{equation}\label{eq map of htpy groups from smash splitting of en to en}
    \wfpn \otimes_{\zp} \zp[\lv u_1, \dots,u_{n-1}\rv][u^\pm] \to W(\fpn)[\lv u_1, \dots,u_{n-1}\rv][u^\pm].
\end{equation}
at the level of homotopy groups \cite[7.2.2.13]{lurie2016higher}.
Since $\wfpn$ is a free $\zp$-module of finite rank,  the functor $\wfpn \otimes_{\zp} - $ is given by taking a $n$-fold product of $-$. In particular, the functor $\wfpn \otimes_{\zp} - $ commutes  with completions. This shows that \eqref{eq map of htpy groups from smash splitting of en to en} is an isomorphism as desired. \begin{comment}We deduce that the isomorphism 
\[\wfpn \otimes_{\zp} \zp[u_1, \dots,u_{n-1}][u^\pm] \xrightarrow{\cong} W(\fpn)[u_1, \dots,u_{n-1}][u^\pm]\]
provides an isomorphism after completing with $u_1, \dots,u_{n-1}$, i.e.\ \eqref{eq map of htpy groups from smash splitting of en to en} is an isomorphism. Therefore,  the composite map \eqref{eq map from smash splitting of en to en} is an equivalence as desired.\end{comment} 
\end{proof}
\subsection{Algebraic $\kth$-theory of Morava $E$-theories}

There is a finite subgroup $\mathbb{F}_{p^n}^\times$ of the Morava stabilizer group such that $K = \mathbb{F}_{p^n}^\times \rtimes Gal$ acts on the $E_\infty$-algebra $E_n$. Furthermore, 
\[E_n^{hK} \simeq \widehat{E(n)}\]
where $\widehat{E(n)}$ denotes the $K(n)$-localization of the Johnson-Wilson spectrum $E(n)$, see \cite[Section 5.4.7]{rognes2008galois}. We have 
\[\pis E(n) \cong \zpl[v_1,\dots,v_{n-1}][v_n^\pm] \textup{\ \ and \ }\pis \widehat{E(n)} \cong \zpl[v_1,\dots,v_{n-1}][v_n^\pm]_I^{\land}\]
where $I$ denotes the ideal $(p,v_1,\dots,v_{n-1})$.
Since $\widehat{E(n)}$ is given by $E_n^{hK}$, there is a  $\kth$-equivariant map of  $E_\infty$-algebras  $\widehat{E(n)} \to E_n$. In particular, this provides a map  
\[\widehat{E(n)} \to E_n^{hGal}\]
of $E_\infty$-algebras. This map carries $v_n$ to $u^{-(p^n-1)}$ and $v_i$ to $u_i u^{-(p^i-1)}$ for $0<i<n$.

For the following, we fix an $E_2$-map  $\sph[\sigma_{2(p^n-1)}] \to \widehat{E(n)}$ to adjoin roots. Recall from Remark \ref{rema e2rings root adj are xalgebras} that in this situation, $\widehat{E(n)}(\sqrt[p^n-1]{v_n})$ is an $\widehat{E(n)}$-algebra.

\begin{theo}\label{theo morava E theories as root adjunction}
There are equivalences of $E_1$ $\widehat{E(n)}$-algebras:
\begin{align*}
    E_n^{hGal} \simeq& \  \widehat{E(n)}(\sqrt[p^n-1]{v_n})  \\
     E_n \simeq&\  \sphwq \wdg_{\sphp}\widehat{E(n)}(\sqrt[p^n-1]{v_n})
\end{align*}
where the class $u^{-1}$ corresponds to $\sqrt[p^n-1]{v_n}$ at the level of  homotopy groups for both of these equivalences.

In particular, $E_n^{hGal}$ and $E_n$ are $p^n-1$-graded $E_1$ $\widehat{E(n)}$-algebras with 
\[(E_n^{hGal})_i \simeq \Sigma^{2i} \widehat{E(n)}\]
and 
\[(E_n)_i \simeq \Sigma^{2i} \sphwq \wdg_{\sphp}\widehat{E(n)}\]
for every $0\leq i<p^n-1$.

\end{theo}

\begin{proof}
By inspection, one observes that 
\[\pis (\widehat{E(n)}(\sqrt[p^n-1]{v_n}))\cong \pis E_n^{hGal},\]
see \cite[5.4.9]{rognes2008galois}. Furthermore, the map of rings, 
\[\pis \widehat{E(n)} \to \pis (\widehat{E(n)}(\sqrt[p^n-1]{v_n}))\cong (\pis \widehat{E(n)})[z]/(z^{p^n-1}-v_n)\]
is a map of \'etale rings as $v_n$ and $p^n-1$ are invertible in $\pis \widehat{E(n)}$. Through \cite[Theorem 1.10]{hesselholt2022dirac}, we obtain the first equivalence in the theorem. The second equivalence follows by the first equivalence and Proposition \ref{prop smash splitting for morava e theory}. The statement on graded ring structures follows by the fact that root adjunction results in $m$-graded ring spectra, see Construction \ref{cons adjroots}.
\end{proof}

We are ready to prove our  result on the $\kth$-theory of Morava $E$-theories. For this, we use the following composite map 
\begin{equation}\label{eq map from jw to mrve}
    E(n) \to \widehat{E(n)} \to \widehat{E(n)}(\sqrt[p^n-1]{v_n}) \xrightarrow{\simeq} E_n^{hGal}
\end{equation}
of $E_1$-rings where the last map above is given by Theorem \ref{theo morava E theories as root adjunction}. Using Proposition \ref{prop smash splitting for morava e theory}, we obtain the following composite: 
\begin{equation}\label{eq map from witt jw to mrve}
\sphwpn \wdg E(n) \to \sphwpn \wdg \widehat{E(n)} \to \sphwpn \wdg E_n^{hGal} \to \sphwpn \wdg_{\sphp} E_n^{hGal} \xrightarrow{\simeq} E_n.
\end{equation}

\begin{theo}\label{theo K theory of Morava E theories}
The maps
\begin{align*}
\kth(E(n)) \to&\  K(E_n^{hGal})\\
\kth(\sphwpn \wdg E(n)) \to & \  \kth(E_n)
\end{align*}
induced by those above are inclusions of wedge summands after $T(n+1)$-localization. 
\end{theo}
\begin{proof}
The first map in \eqref{eq map from jw to mrve} is a $T(n) \vee T(n+1)$-equivalence and hence induces a $T(n+1)$-equivalence in algebraic $K$-theory \cite[Purity Theorem]{land2020purity}. Therefore, the first equivalence in the theorem follows by applying Corollary \ref{corr nonconnective alg kthry inclusion for root adj} to $\widehat{E(n)}(\sqrt[p^n-1]{v_n})$.

Similarly, the first and the third maps in \eqref{eq map from witt jw to mrve} induce $T(n+1)$-equivalences in algebraic $K$-theory. The second equivalence in the theorem follows by observing that there is an equivalence of $E_1$-algebras:
\[\sphwpn \wdg \big(\widehat{E(n)}(\sqrt[p^n-1]{v_n})\big) \simeq \big(\sphwpn \wdg \widehat{E(n)}\big )(\sqrt[p^n-1]{v_n})\]
and then applying Corollary \ref{corr nonconnective alg kthry inclusion for root adj}.
\end{proof}
%\subsection{Algebraic $K$-theory of $KO$} Using the map $ko_p \to ku_p$, we should be able to deduce the algebraic $K$-theory of $ko_p$ and show that $T(2)$-locally, this is given by adjoining a root to $K$-theory of $l$.
\subsection{Two-periodic Morava $K$-theories}\label{sec two periodic morava ktheory}
 We obtain analogous results for two-periodic Morava K-theories. 
 Taking a quotient with respect to a regular sequence in $\pi_* \sphwpn \wdgp \widehat{E(n)}$, one obtains an $\widehat{E(n)}$-algebra $K(n)$ \cite{lazarev2003towersofmualgebras,angeltveit2008thhofainfty,hahn2018quotients}. Here, $K(n)$ is the Morava $K$-theory spectrum with coefficients in $\fpn$. We have 
 \[\pis K(n) \cong \fpn[v_n^{\pm}].\]
 Using the $\widehat{E(n)}$-algebra structure on $K(n)$, we adjoin roots and define the two periodic Morava $K$-theory as follows 
 \[K_n := K(n)(\sqrt[p^n-1]{v_n}).\]
 In this case, 
 \[\pis K_n \cong \fpn[u^{\pm}]\]
 where $\lv u \rv = -2$.
Together with Theorem \ref{theo morava E theories as root adjunction}, this provides a commuting diagram of $E_1$-rings 
 \begin{equation*}
     \begin{tikzcd}
         \sphwpn \wdgp \widehat{E(n)} \ar[r]\ar[d] & E_n \ar[d]\\
         K(n) \ar[r] & K_n
     \end{tikzcd}
 \end{equation*}
 which justifies our definition of $K_n$.
 In particular, $K_n$ is a $p^n-1$-graded $E_1$-ring in a non-trivial way and we obtain the following from Corollary \ref{corr nonconnective alg kthry inclusion for root adj}.

 \begin{theo}\label{theo two periodic mrv ktheory alg kthry inclusion}
 The following map 
 \[\kth(K(n)) \to K(K_n)\]
     is the inclusion of a wedge summand after $T(i)$-localization for every $i\geq 2$.
 \end{theo}
 \begin{corr}\label{corr redshift for two periodic mrv kthry}
      If $K(n)$ satisfies the redshift conjecture, then so does $K_n$.
 \end{corr}
 The $V(1)$-homotopy of $\kth(k(1))$ is computed by  Ausoni and Rognes in \cite{ausoni2021kthrymrvKthry} for $p>3$. In particular, their computation shows that $K(1)$ satisfies the redshift conjecture. We obtain the following. 
 \begin{corr}\label{corr first two perdc mrv ktrhy redshift}
     The two periodic Morava $K$-theory $K_1$ of height one satisfies the redshift conjecture for $p>3$.
 \end{corr}
 
\subsection{THH descent for Morava $E$-theories}
Theorem \ref{theo cofiber seq for thh to log thh}  identifies THH of various periodic ring  spectra with their logarithmic THH. For instance, the  Morava $E$-theory spectrum $E_n$ is periodic with a unit $u$ in degree $-2$. Since $E_n/(u^{-1}) \simeq 0$, the canonical map  
\begin{equation*}
    \thh(E_n) \xrightarrow{\simeq} \thh(E_n \mid u^{-1})
\end{equation*}
is an equivalence. Using this, together with our result on logarithmic THH-\'etaleness of root adjunction, we show that $\thh(E_n)$ may be obtained from $\thh(\widehat{E(n)})$ via base-change up to $p$-completion. Such base-change formulas and their relationship with Galois descent problems for THH  were studied by Mathew in \cite{mathew2017thhbasechng}. %Recall from Theorem \ref{theo morava E theories as root adjunction} that we have equivalence of $\widehat{E(n)}$-algebras 
\begin{comment}
\begin{equation*}
 E_n \simeq \sphwpn \wdgp \widehat{E(n)}(\sqrt[p^n-1]{v_n}).
\end{equation*}
\end{comment}

\begin{theo}\label{theo thh of galois fixed points}
The canonical map: 
\[\thh(\widehat{E(n)}) \wdg_{\widehat{E(n)}} E_n^{hGal} \xrightarrow{\simeq} \thh(E_n^{hGal}), \]
is an equivalence. 
\end{theo}
\begin{proof}
Recall from Theorem \ref{theo morava E theories as root adjunction} that  there is an equivalence of $\widehat{E(n)}$-algebras 
\[E_n^{hGal} \simeq \widehat{E(n)}(\sqrt[p^n-1]{v_n}). \]
Therefore, it follows by Construction \ref{cons adjroots} that
\begin{equation}\label{eq  thh basechange of jw spectrum}
\begin{split}
    \thh(\widehat{E(n)})\wdg_{\widehat{E(n)}} E_n^{hGal} &\simeq \thh(\widehat{E(n)})\wdg_{\widehat{E(n)}} \widehat{E(n)} \wdg_{\sphpl[\sigma_{2(p^n-1)}]} \sphpl[\sigma_2] \\ &\simeq \thh(\widehat{E(n)}) \wdg_{\sphpl[\sigma_{2(p^n-1)}]} \sphpl[\sigma_2].
\end{split}
\end{equation}
    Since $v_n$ is a unit in $\widehat{E(n)}$ and $u^{-1}$ is a unit in $E_n^{hGal}$, Theorem \ref{theo cofiber seq for thh to log thh} provides the equivalences:
\begin{equation*}
       \thh(\widehat{E(n)}) \xrightarrow{\simeq} \thh(\widehat{E(n)} \mid v_n)  \textup{\ \ and \ }\thh(E_n^{hGal}) \xrightarrow{\simeq} \thh(E_n^{hGal} \mid u^{-1}).
\end{equation*}
Using these equivalences together with Theorem \ref{theo root adj is log thh etale }, we obtain that the following canonical map is an equivalence.
\[\thh(\widehat{E(n)}) \wdg_{\sphpl[\sigma_{2(p^n-1)}]} \sphpl[\sigma_{2}] \xrightarrow{\simeq } \thh(E_n^{hGal})\]
This, together with \eqref{eq  thh basechange of jw spectrum}, provides the desired result. 
\end{proof}
\begin{theo}\label{theo thh basechange for morava e theory}
The canonical map: 
\[\thh(\widehat{E(n)})\wdg_{\widehat{E(n)}} E_n \xrightarrow{\simeq_p} \thh(E_n),\]
is an equivalence after $p$-completion.
\end{theo}
\begin{proof}
The first equivalence below follows by  Proposition \ref{prop smash splitting for morava e theory}. 
 and the second follows by Theorem \ref{theo thh of galois fixed points}.
 \begin{equation*}
    \begin{split}
        \thh(\widehat{E(n)}) \wdg_{\widehat{E(n)}} E_n &\simeq \thh(\widehat{E(n)})\wdg_{\widehat{E(n)}}( E_n^{hGal} \wdgp \sphwpn) \\
       & \simeq \thh(E_n^{hGal}) \wdgp \sphwpn
    \end{split}
\end{equation*}

Therefore, it is sufficient to show that the canonical map 
\[\thh(E_n^{hGal}) \wdgp \sphwpn \to \thh(E_n^{hGal} \wdgp \sphwpn)\]
is an equivalence after $p$-completion. This follows by the following canonical diagram of $\sph/p$-equivalences. 
\begin{equation*}
    \begin{tikzcd}
    \thh(E_n^{hGal}) \wdgp \sphwpn \ar[r] \ar[d,"\simeq_p"] &\thh(E_n^{hGal} \wdgp \sphwpn) \ar[d,"\simeq_p"] \\
    \thh^{\sphp}(E_n^{hGal}) \wdgp \sphwpn \ar[r] \ar[d,"\simeq_p"] & 
    \thh^{\sphp}(E_n^{hGal} \wdg_{\sphp} \sphwpn) \\
    \thh^{\sphp}(E_n^{hGal}) \wdgp \thh^{\sphp}(\sphwpn) \ar[ru,"\simeq"]&
    \end{tikzcd}
\end{equation*}
The right hand vertical map and the upper left  vertical map are $\sph/p$-equivalences due to \cite[Lemma 5.20]{mao2020perfectoid}. The fact that the lower left vertical map is an $\sph/p$-equivalence follows by \cite[proof of Lemma 5.20]{mao2020perfectoid} and the fact that the composite $\sphwpn \to \thh(\sphwpn) \to \sphwpn$ is an equivalence. This shows that the upper horizontal map is an $\sph/p$-equivalence proving the theorem. 
\end{proof}

\bibliographystyle{amsalpha}
{\footnotesize \bibliography{references}}
\end{document}